\documentclass[leqno,english,a4paper]{amsart}
\usepackage[text={13cm,21cm},centering]{geometry}
\usepackage{amsfonts,amssymb,amsmath,amsgen,amsthm}
\usepackage{graphicx}
\usepackage{mathrsfs}
\usepackage{float}
\usepackage{tikz,tkz-tab}
\usepackage{chngcntr}

\usepackage{pgfplots}
\pgfplotsset{compat=newest} 
\pgfplotsset{plot coordinates/math parser=false} 
\newlength\figureheight 
\newlength\figurewidth 
\newlength\malargeur
\usepackage{hyperref,cases}
\usepackage{color}
\newcommand{\msc}[2][2000]{%
  \let\@oldtitle\@title%
  \gdef\@title{\@oldtitle\footnotetext{#1 \emph{Mathematics subject
        classification.} #2}}%
}

\theoremstyle{plain}
\newtheorem{theorem}{Theorem} [section]
\newtheorem{definition}[theorem]{Definition}
\newtheorem{assumption}[theorem]{Assumption}
\newtheorem{lemma}[theorem]{Lemma}

\newtheorem{proposition}[theorem]{Proposition}

\theoremstyle{remark}
\newtheorem{remark}[theorem]{Remark}

\def\C{{\mathbb C}}
\def\R{{\mathbb R}}
\def\N{{\mathbb N}}
\def\Z{{\mathbb Z}}
\def\T{{\mathbb T}}
\def\O{\mathcal O}

\def\A{\mathcal A}
\def\B{\mathcal B}
\def\CC{\mathcal C}

\def\cont{\mathscr{C}}

\def\({\left(}
\def\){\right)}
\def\<{\left\langle}
\def\>{\right\rangle}
\def\le{\leqslant}
\def\ge{\geqslant}
\def\1{\mathbf 1}
\renewcommand{\leq}{\le}
\renewcommand{\geq}{\ge}

\def\Tend#1#2{\mathop{\longrightarrow}\limits_{#1\rightarrow#2}}

\def\d{{\partial}}
\def\eps{\varepsilon}

\DeclareMathOperator{\Sign}{Sign}

\numberwithin{equation}{section}

\newcommand\ut{{\underline t}}
\newcommand\uu{{\underline u}}

\newcommand\ux{{\underline x}}

\newcommand\To{{T_*}}

\begin{document}
\counterwithin{figure}{section}

\title[a conservation law with localized damping]{A conservation law with
  spatially localized sublinear damping}

\author[C. Besse]{Christophe Besse}
\address[C. Besse]{Institut de Math\'ematiques de Toulouse, UMR5219\\ Universit\'e de Toulouse, CNRS \\ UPS IMT, F-31062 Toulouse Cedex
9\\ France} 
\email{christophe.besse@math.univ-toulouse.fr}

\author[R. Carles]{R\'emi Carles}
\address[R. Carles]{CNRS\\ IMAG, UMR 5149\\ Montpellier\\ France}
\email{Remi.Carles@math.cnrs.fr}

\author[S. Ervedoza]{Sylvain Ervedoza}
\address[S. Ervedoza]{Institut de Math\'ematiques de Toulouse, UMR5219\\ Universit\'e de Toulouse, CNRS \\
UPS IMT, F-31062 Toulouse Cedex 9, France}
\email{sylvain.ervedoza@math.univ-toulouse.fr}

\date{\today}

\begin{abstract}
 We consider a general conservation law on the circle, in the presence
 of a sublinear damping. If the damping acts on the whole circle, then
 the solution becomes identically zero in finite time, following the
 same mechanism as the corresponding ordinary differential
 equation. When the damping acts only locally in space, we show a
 dichotomy: if the flux function is not zero at the origin, then the
 transport mechanism causes the extinction of the solution in finite
 time, as in the first case. On the other hand, if zero is a
 non-degenerate critical point of the flux function, then the solution
 becomes extinct in finite time only inside the damping zone, decays
 algebraically uniformly in space, and we exhibit a boundary layer, shrinking
 with time, around the damping zone. Numerical illustrations show how
 similar phenomena may be expected for other equations. 
\end{abstract}
\thanks{S. E.  was partially supported by IFSMACS ANR-15-CE40-0010 of
  the French National 
  Research Agency (ANR)} 
\maketitle

\section{Introduction}
\label{sec:intro}

  We consider a general conservation law on the torus $\T=\mathbb{R}/ \mathbb{Z}$, in the presence
of a sublinear damping, possibly localized in space,
\begin{equation}
  \label{eq:flux-general}
 \d_t u + \d_x\left( f(u) \right) +a(x)\frac{u}{|u|^\alpha}=0,\quad (t,x)\in
 \R_+\times \T,
\end{equation}
with a smooth flux $f\in \mathscr{C}^\infty(\R,\R)$, $0<\alpha\le 1$ and $a = a(x) \ge 0$. 
For the Cauchy problem,  we
prescribe the initial datum
\begin{equation}
  \label{eq:ci}
  u_{\mid t=0}=u_0, \quad x \in \T.
\end{equation}
In the case where $a>0$ is constant, the sublinear nonlinearity is
motivated by the effect of friction forces that occur in almost every
mechanism with moving parts, this process arising between all surfaces
in contact. The first concepts go back to the work of Leonardo
da~Vinci on
friction, rediscovered by Amontons \cite{Amontons} at the end of the
17th century, and 
then developed by Coulomb \cite{Coulomb} in the 18th century. The main
idea is that 
the friction is opposed to the movement and that the friction force is
independent of the speed $v$ and the contact surface. The friction force,
known today as Coulomb friction, is therefore described as $ F =
F_c\, \text{sgn}(v)$.
Depending on how the sign function is defined, it
can be zero  or take any value in the interval $[-F_c,F_c]$. In the
19th century, the theory of hydrodynamics was developed leading to
expressions for the frictional force caused by the viscosity of
lubricants, and is usually modeled by $ F = F_v v $. The linearity
with respect to speed is not always correct and a more general
relation is $F =F_v|v|^{\delta_v}\text{sgn} (v) $ where
$\delta_v $ depends on the geometry of the application 
(see e.g. \cite{Olsson,AdAtCa06,AmDi03} and references therein). The
basic model for the motion of a body lying on a surface is given by the
Newton law. It reduces to the
ordinary differential equation, for $\alpha \in (0,1)$,
\begin{equation}
	\label{Model-ODE}
	\dot u = -\frac{u}{|u|^\alpha}, \quad t \in \R, \qquad u(0)= u_0\in \R.
\end{equation}
By separating the variables, explicit integration yields, in terms of
$\rho=u^2$, since $\rho\ge 0$,
\begin{equation}\label{eq:ode-explicit}
  \dot \rho = -2\rho^{1-\alpha/2},\quad \text{hence }\rho(t)= 
\left\{
  \begin{aligned}
    &\(|u_0|^\alpha-\alpha t\)^{2/\alpha}& \text{ if } t\le
    |u_0|^\alpha / \alpha,\\
&0 & \text{ if } t> |u_0|^\alpha / \alpha. 
  \end{aligned}
\right.
\end{equation}
Therefore, $\rho$ becomes zero in finite time, and so does $u$. Note
that for $\alpha = 1$, the equation \eqref{Model-ODE} should be understood in the sense
of Filippov (see \cite[Chapter 2]{Filippov}): $u' \in - \Sign(u)$, in
which $\Sign(u)$ is defined by  
\begin{equation}
	\label{Def-Sign-u}
	 \Sign(u) = 
	\left\{ 
\begin{aligned}
&\{1\}  &\text{ if } u >0,\\ 
&\{-1\}  & \text{ if } u < 0, \\
& [-1,1] & \text{ if } u =  0, 
\end{aligned}\right. 
\end{equation}
and the same argument as above still applies. Besides, note that
solutions of the above ODE \eqref{Model-ODE}, whether $\alpha \in
(0,1)$ or $\alpha = 1$, are unique in positive time even if the source
term is not $\cont^1$ with respect to $u$, as a consequence of the
one-sided Lipschitz condition satisfied by $h_\alpha(u ) = -
u/|u|^\alpha$, see \cite[Chapter 2, Section 10, Theorem 1]{Filippov},
which reads as follows: for all $(u,v) \in \R^2$,  
\[
	(u-v) (h_\alpha(u) - h_\alpha(v)) \leq 0.
\]

Such sublinear damping models have been considered for some partial
differential equations: in 
the case of the wave equation \cite{BaCaDi07,Perrollaz-Rosier-2014}, in the case of various
parabolic equations \cite{AgEs86,Be01,Be10,BeHeVe01,BeSh07,BeSh10},
and in the case of the Schr\"odinger equation
\cite{CaGa11,CaOz15}. The aspect that we now wish to investigate is
the effect of such a damping when it is localized in space. Typically,
the function $a$ in \eqref{eq:flux-general} can be thought of as an
indicating function. 

\smallbreak
Some of the results that we present can be adapted to the case where
the space variable belongs to the whole line $\R$. The reason why we
consider the periodic case is the following. On the whole line, the characteristics of the 
solution of \eqref{eq:flux-general} may cross the support of $a$
without undergoing such a strong affect as in \eqref{eq:ode-explicit},
that is, the sublinear damping occurs in too small a region to put $u$
to zero. On the other hand, in a periodic box, and in the case where
transport phenomenon is present, the solution will meet the support of
$a$ as long as it is not zero. These are typically the possibilities
which we want to understand. 
\begin{assumption}\label{hyp:a}
  The function $a$ is nonnegative, $a(x)\ge 0$ for all $x\in \T$,
  bounded, $a\in L^\infty(\T)$, and satisfies
  \begin{equation*}
  \sup_{y >0} \frac{1}{y} \int_\T |a(x+y)-a(x)| \ dx < \infty .
  \end{equation*}
\end{assumption}
Typically, this condition is satisfied for $a \in BV(\T)$, see \cite[Chapter 1 Theorem 1.7.1]{Da00} (in fact, this is nearly equivalent of being in $BV(\T)$). In particular, $a$ may be an indicating function,
$a(x)={\mathbf 1}_\omega(x)$ for some measurable set $\omega\subset \T$. 

\subsection{Cauchy problem}

 The notion of solution, as well as the vanishing viscosity method used
to solve the Cauchy problem, follow from standard arguments (which we
borrow from \cite{Da00}). We shall see that the presence of the
damping term in \eqref{eq:flux-general} requires only slight
modifications of this approach. We emphasize however that the case
$\alpha=1$ is specific, and we shall treat it by adapting the
approach of Filippov \cite{Filippov}. 
\begin{definition}[Notion of solution, $0<\alpha<1$]\label{def:solution}
  Let $\alpha \in (0,1)$. A bounded measurable function $u$
  on $[0,T]\times\T$ is an \emph{admissible weak solution} of
  \eqref{eq:flux-general}--\eqref{eq:ci}, with $u_0\in L^\infty(\T)$,
  if the inequality 
  \begin{equation}
  \label{AdmissibleWeak-Identity} 
     \int_0^T\!\!\!\!\int_\T \(\d_t \psi \eta(u)+\d_x \psi q(u)- a\psi
   \eta'(u) \frac{u}{|u|^{\alpha}} \)dxdt+\int_\T
    \psi(0,x)\eta\(u_0(x)\)dx\ge 0
  \end{equation}
holds for every convex function $\eta\in W^{1,\infty}$, with $q'=f'\eta'$, and all
nonnegative Lipschitz continuous test function $\psi$ on $[0,T]\times
\T$. 
\end{definition}
\begin{definition}[Notion of solution, $\alpha=1$]\label{def:solution-alph=1}
  Let $\alpha = 1$. A bounded measurable function $u$
  on $[0,T]\times\T$ is an \emph{admissible weak solution} of
  \eqref{eq:flux-general}--\eqref{eq:ci}, with $u_0\in L^\infty(\T)$,
  if there exists $h \in L^\infty((0,T)\times \T)$ such that
  \[
  	\partial_t u + \partial_x(f(u)) + h = 0, \quad \text{ in } \mathscr{D}'((0,T)\times \T),
  \]
  with
  \[
  	h(t,x) \in a(x) \Sign(u(t,x)),\quad \text{a.e. } (t,x) \in (0,T) \times \T,
  \]
  where $\Sign$ is defined in \eqref{Def-Sign-u}, 
  and such that the inequality 
  \begin{equation}
  \label{AdmissibleWeak-Identity-alph=1} 
     \int_0^T\!\!\!\!\int_\T \(\d_t \psi \eta(u)+\d_x \psi q(u)- \psi
   \eta'(u) h(t,x) \)dxdt+\int_\T
    \psi(0,x)\eta\(u_0(x)\)dx\ge 0
  \end{equation}
holds for every convex function $\eta\in W^{1,\infty}$, with $q'=f'\eta'$, and all
nonnegative Lipschitz continuous test function $\psi$ on $[0,T]\times
\T$. 
\end{definition}
In all that follows, the notion of solution refers either to
Definition~\ref{def:solution} (case $0<\alpha<1$), or to
Definition~\ref{def:solution-alph=1} (case $\alpha=1$). 
We show that the Cauchy problem is well-posed, regardless of the
value of $\alpha\in (0,1]$. 
\begin{proposition}[Cauchy problem]\label{prop:cauchy}
	Assume that $a$ satisfies Assumption \ref{hyp:a} and $\alpha \in (0,1]$.  Let $u_0\in L^\infty(\T)$. There exists a unique, global, admissible weak
  solution $u$ of \eqref{eq:flux-general}--\eqref{eq:ci}, $u\in
  \mathscr{C}^0(\R_+;L^1(\T))$. 
\end{proposition}
We will also need  the following comparison result.  
\begin{proposition}[Comparison principles]\label{prop:comparaison}
	Let $\alpha \in (0,1]$.
	\\
	1. Let $a$ satisfying Assumption \ref{hyp:a}, $u$ and $v$ be solutions of \eqref{eq:flux-general} with
 respective initial data $u_0,v_0\in L^\infty(\T)$ such that $u_0 \leq v_0$. Then
		 \begin{equation*}
			   u(t,x)\le v(t,x),\quad \forall t\ge 0,\
                           \text{ a.e. } x\in \T.  
		 \end{equation*}
	Besides,
\begin{equation*} 
			|u(t,x)| \leq \| u_0 \|_{L^\infty(\T)}, \quad
                        \forall t\ge 0,\ 
                        \text{a.e. } x\in \T. 
		\end{equation*}
	2. Let $a_1$ and $a_2$ satisfying Assumption \ref{hyp:a} such that
        for almost all $x \in \T$, $a_1(x) \le a_2(x)$. Then, denoting by
        $u_1$ and $u_2$ the respective solutions to
        \eqref{eq:flux-general}--\eqref{eq:ci} with the same initial
        datum $u_0 \in L^\infty(\T)$ with $u_0\ge 0$, we have 
		 \begin{equation*}
			   u_1(t,x)\ge u_2(t,x)\ge 0,\quad \forall
                           t\ge 0, \ \text{a.e. } x\in  \T. 
		 \end{equation*}	
\end{proposition}
\begin{remark}[BV solutions]\label{rem:BV}
  As a straightforward consequence of the proof of Proposition \ref{prop:comparaison}, given in
  Section~\ref{sec:cauchy}, one can show that if $u_0\in BV(\T)$, then the
  solution remains in $BV$,  $u\in L^\infty(\R_+;BV(\T))$. However, we
  shall not use this property in this paper. 
\end{remark}

\subsection{Extinction results}

 We  now focus on the core of this article, and give several results
 regarding the possible extinction of the solutions $u$ of
 \eqref{eq:flux-general}--\eqref{eq:ci}. The results  depend on the
 flux $f$ (its behavior near the origin), and
 the damping coefficient $a$, which is always assumed to satisfy
 Assumption \ref{hyp:a}.  
\smallbreak
The first case which we consider is the one corresponding to a damping
coefficient acting everywhere. 

\begin{proposition}[Finite time extinction with  damping
  everywhere]\label{prop:partout} 
	Suppose that there exists $\delta>0$ such that
\begin{equation}\label{eq:partout}
	a(x)\ge \delta>0,\quad \forall x\in \T.
\end{equation}
  Let $u_0\in L^\infty(\T)$. There exists
  $T>0$ such that the solution 
  to \eqref{eq:flux-general}--\eqref{eq:ci}  satisfies
  \begin{equation*}
    u(t,x)=0, \quad \forall t\ge T,\ \text{a.e. }x\in \T. 
  \end{equation*}
  Besides, $T$ can be chosen as 
  \begin{equation}
  	\label{Extinct-Time-a-partout}
  	T  =\frac{1}{\alpha \delta} \|u_0 \|_{L^\infty(\T)}^\alpha.
  \end{equation}
\end{proposition}
The proof of Proposition \ref{prop:partout} is presented in Section
\ref{sec:partout} and is based on a Lyapunov approach. More precisely,
we derive $L^p(\T)$ estimates on the solutions of
\eqref{eq:flux-general}--\eqref{eq:ci},  and let then $p$ go to
infinity, so that we obtain a differential inequality for the
$L^\infty(\T)$-norm of the solutions of
\eqref{eq:flux-general}--\eqref{eq:ci}, which in turn implies its
extinction in finite time. 
\smallbreak

Next, as motivated above, we  consider the case in which the damping
coefficient acts only 
in some part of the domain: 
\begin{equation}
	\label{eq:local}
	\exists \hbox{ an open interval } \omega \subset \T \hbox{ and } \delta >0 \, \hbox{ s.t. } \quad
	a(x) \ge \delta, \quad \forall x \in \omega.
\end{equation}
The extinction of the solution of
\eqref{eq:flux-general}--\eqref{eq:ci} in this case will depend on the
flux. Namely, we will treat two different cases, depending whether
$f'(0)$ vanishes or not. 
The easier case corresponds to the presence of transport at the origin, 
\begin{equation}
	\label{eq:transport}
	f'(0) \neq 0.
\end{equation}
In this case, one expects that the transport phenomenon will steer the
solution through the set $\omega$ an arbitrary number of times, so that
the strong friction term will make the solution vanish after some finite
time.  
%
%
%

%
In agreement with these insights, in Section~\ref{sec:transport} we prove the following result:
\begin{theorem}[Finite time extinction by transport]
	\label{thm:Transport}
	Assume that the flux $f$ is smooth and satisfies
        \eqref{eq:transport}, and the damping profile $a$ satisfies
        Assumption \ref{hyp:a} and \eqref{eq:local}. Let $K$ be such
        that  
	\begin{equation}
		\label{eq:def-K}
		\inf_{ s \in [-K,K]} |f'(s)| >0.
	\end{equation}
	\\
	Then for any initial datum $u_0 \in L^\infty(\T)$ satisfying 
	\begin{equation}
		\label{eq:ci-small}
		\|u_0\|_{L^\infty(\T)} \le K, 
	\end{equation}
	there exists $T>0$ such that the solution $u$ of \eqref{eq:flux-general}--\eqref{eq:ci} satisfies
	\begin{equation}
		\label{eq:null-after-some-time}
		u(t, x) = 0, \quad \forall t \ge T,\, a.e.\, x \in \T.
	\end{equation}
\end{theorem}
To illustrate the typical behavior of such a solution, we plot on Figure \ref{fig:evo_transport} the evolution of the solution of the transport equation corresponding to $f(u)=2u$ with initial datum $u_0(x)=1.25$ for $(t,x)\in [0,10]\times(0,1)$, $a(x)=\1_{(0,1/4)}$ and $\alpha=1$. The solution is computed using the numerical procedure described in Section \ref{sec:num}. The dashed line indicates the position of the support of $a$.
\begin{figure}[h]
  \centering
\input{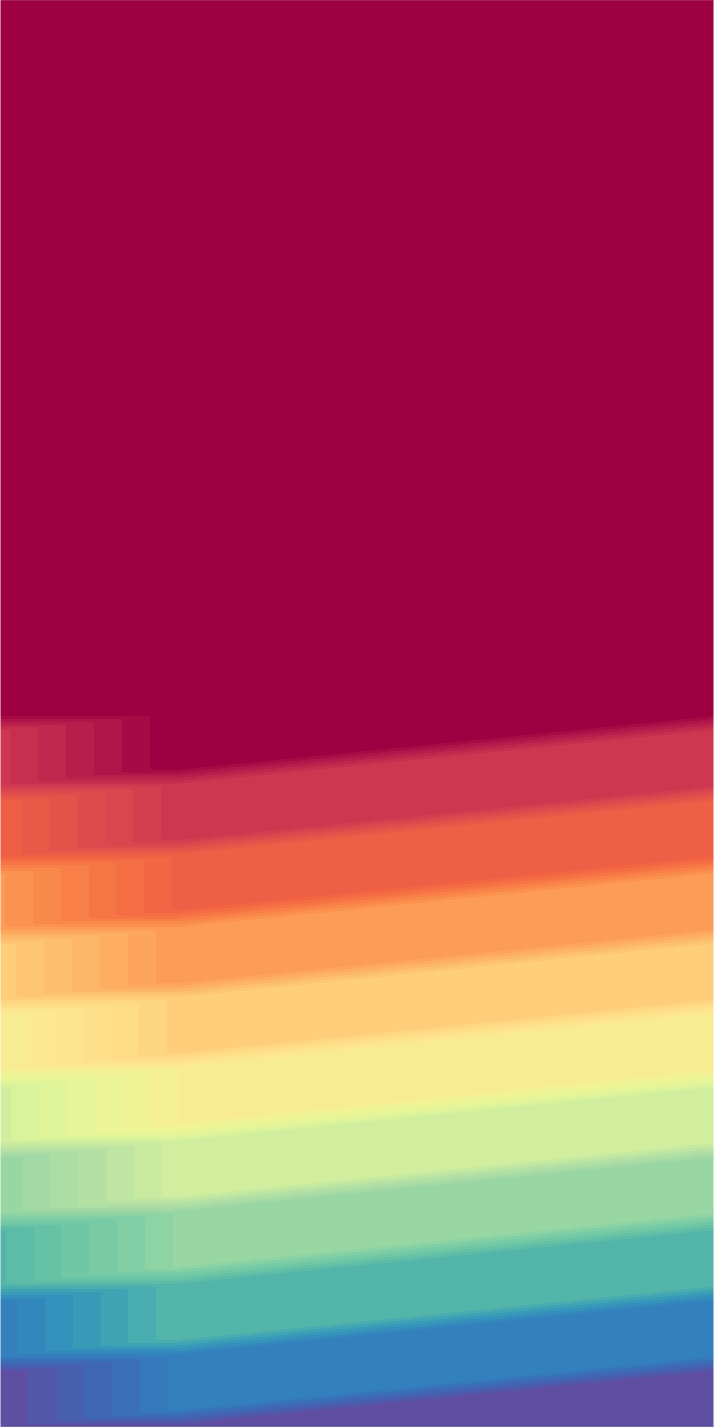}
  \caption{Evolution of the solution for transport equation}
  \label{fig:evo_transport}
\end{figure}
Note that Theorem \ref{thm:Transport} is a semi-global result, as it
is valid for any initial datum $u_0$ whose $L^\infty(\T)$-norm is
bounded by the constant $K$, which is chosen in \eqref{eq:def-K} to
guarantee that the transport phenomenon really occurs. 
\smallbreak

Our proof of Theorem \ref{thm:Transport} relies on Lyapunov
functionals, in a similar spirit as the one developed in Section
\ref{sec:partout} to address the proof of
Proposition~\ref{prop:partout}. However, as the transport phenomenon
is now essential to the 
decay process, we will introduce some weights in space in the
functional. This approach is
 inspired by some recent works on the stabilization of
hyperbolic systems of conservation laws, namely
\cite{Coron-Andrea-Bastin-IEEE-2007,Coron-Bastin-Andrea-SICON-2008}
(see the recent book \cite{Bastin-Coron-book} for further
references). 
\smallbreak

In the case 
\begin{equation}
	\label{eq:ConditionPourrie}
	f'(0) = 0, 
\end{equation}
corresponding for instance to the celebrated example of Burgers equation
\begin{equation}
	\label{eq:flux=Burgers}
	f(u) = \frac{u^2}{2}, 
\end{equation}
the transport phenomenon competes with the dissipation of the
solution, as the smaller the solution is, the slower the characteristics propagate. Our goal
thus is to understand the interplay between these phenomena. 
\begin{theorem}
	\label{thm:Burgers-Main}
	Let $f$ be a smooth function such that 
	\begin{equation}
		\label{Cond-Flux}
		f'(0) = 0 
		\quad \hbox{ and } \quad 
		\exists K >0, \quad \inf_{s \in [-K,K]} |f''(s) | >0. 
	\end{equation}
	Assume that the damping profile $a = a(x)$ satisfies Assumption \ref{hyp:a} and \eqref{eq:local}. 
	Then, for any initial datum $u_0 \in L^\infty (\T)$ satisfying \eqref{eq:ci-small}, 
	the solution $u$ of \eqref{eq:flux-general}--\eqref{eq:ci} satisfies the following property: There exist a time $t_*>0$ and a constant $C>0$ such that for all $t \ge t_*$, there exists an open subinterval $\omega(t) \subset \omega$ such that $u(t)_{\mid \omega(t)} = 0$, and, for all $t \geq t_*$,
	\begin{equation*}
		|\omega \setminus \omega(t) | \le \frac{C}{t^{1+ \alpha}}, 
		\quad 
		\| u(t) \|_{L^\infty(\T)} \le \frac{C}{t}.
	\end{equation*} 
\end{theorem}
The proof of Theorem \ref{thm:Burgers-Main} is based on a precise description of the solution corresponding to $u_0 = K$, where $K$ is constant, $\T$ is identified with $(0,1)$ with periodic boundary conditions, and 
\begin{equation}
	\label{eq:damping-profile}
	a(x) = \delta \1_{\omega}, \hbox{ with } \omega = (0, A). 	
\end{equation}
We can then deduce Theorem~\ref{thm:Burgers-Main} from a simple
comparison argument based on Proposition~\ref{prop:comparaison}. 
\bigbreak

It is clear from these results that there still are cases which are
not covered from our results, in particular cases in which the initial
datum has an $L^\infty(\T)$-norm which is larger than the best $K$ in
\eqref{eq:def-K} or \eqref{Cond-Flux}. Some particular instance is
numerically studied in Section \ref{sec:num}, as well as other models
for which the complete understanding of a (localized) strong friction on the
dynamics of a system is still not well understood.  

\subsection{Outline of the paper}

The Cauchy problem is addressed in Section~\ref{sec:cauchy}, where
Propositions~\ref{prop:cauchy} and \ref{prop:comparaison} are
established.  The proof of Proposition~\ref{prop:partout} (finite time
extinction with damping everywhere) is given in
Section~\ref{sec:partout}, thanks to suitable Lyapunov functionals.
Theorem~\ref{thm:Transport} (finite time 
extinction by transport) is proved in Section~\ref{sec:transport}, by
introducing refined Lyapunov functionals. For the case of the
generalized Burgers equation, Theorem~\ref{thm:Burgers-Main}, a longer
Section~\ref{sec:burgers} is needed, where we first construct a rather
explicit 
solution by following characteristics, which  then turns out to be the
solution provided by Proposition~\ref{prop:cauchy}.  Finally,
Section~\ref{sec:num} provides numerical illustrations in the case of
\eqref{eq:flux-general} studied in this paper, as well as in the case
of other equations for which the corresponding analysis turns out to
be a challenging issue. 


\subsection*{Acknowledgements}  
The authors wish to thank Jean-Fran\c cois Coulombel, Fr\'ed\'eric
Lagouti\`ere,  Philippe Lauren\c cot and Vincent Perrollaz  for
fruitful discussions.

\section{Cauchy problem and comparison principle}
\label{sec:cauchy}

In this section, we prove Propositions~\ref{prop:cauchy} and
\ref{prop:comparaison}. 

We consider more generally the Cauchy problem
\begin{equation}
  \label{eq:existence}
  \d_t u+\d_x f(u)+h(x,u)=0 \text{ in } \R_+ \times \T, \quad u_{\mid
    t=0}=u_0 \text{ in } \T,
\end{equation}
with a fairly general semilinear term $h$, possibly depending on $x$,
in order to generalize the nonlinearity, typically of the form 
\begin{equation}\label{eq:notre-h}
	h_\alpha(x,u)=a(x)\frac{u}{|u|^\alpha},
\end{equation}
where $\alpha \leq 1$ corresponds to the case of \eqref{eq:flux-general}, with the
modification detailed in Definition~\ref{def:solution-alph=1} in the case
$\alpha=1$. We will use the following properties on the
source term $h$, which encompass the framework of
Definition~\ref{def:solution} when $\alpha < 1$.
\begin{assumption}\label{hyp:gen}
  The map $h=h(x,u)$ satisfies:
  \begin{itemize}
\item $h\in L^\infty_{\rm loc}(\T\times\R)$.
  \item For all $u\in \R$, $h(x,u)u\ge 0$, for almost all $x\in \T$.
\item For almost all fixed $x\in \T$, the map $u\mapsto h(x,u)$ is
  nondecreasing on $\R$.
\item For every $R>0$,
  \begin{equation*}
    \sup_{|u|\le R} \sup_{y >0} \frac{1}{y} \int_\T\left| h(x+y,u)-h(x,u)\right|dx < \infty. 
  \end{equation*}
  \end{itemize}
\end{assumption}
In the case \eqref{eq:notre-h}, the first property corresponds to the
assumption $a\in L^\infty(\T)$, the second property to the
fact that $h$ is a damping term, the third property is
straightforward (even in the case $\alpha=1$ with the approach of
Filippov), and the last property is a consequence of 
Assumption~\ref{hyp:a}. 

\begin{definition}[Notion of solution]\label{def:solution-gen}
 Let $h$ satisfy Assumption \ref{hyp:gen}. A bounded measurable function $u$
  on $[0,T]\times\T$ is an \emph{admissible weak solution} of
  \eqref{eq:existence}, with $u_0\in L^\infty(\T)$,
  if the inequality 
  \begin{equation}\label{eq:weak-sol}
     \int_0^T\!\!\!\!\int_\T \(\d_t \psi \eta(u)+\d_x \psi q(u)- \psi
   \eta'(u) h(x,u) \)dxdt+\int_\T
    \psi(0,x)\eta\(u_0(x)\)dx\ge 0
  \end{equation}
holds for every convex function $\eta\in W^{1,\infty}$, with $q'=f'\eta'$, and all
nonnegative Lipschitz continuous test function $\psi$ on $[0,T]\times
\T$. 
\end{definition}

\smallbreak   

Propositions~\ref{prop:cauchy} and \ref{prop:comparaison} stem  from
the following result, which  is slightly more general in view of
Assumption~\ref{hyp:gen}. 

\begin{proposition}\label{prop:cauchy-gen}
  Let  Assumption~\ref{hyp:gen} be
  satisfied.\\
 $(i)$  Let $u_0\in L^\infty(\T)$. There
  exists a unique, global, admissible weak 
  solution $u$ of \eqref{eq:existence}, $u\in
  \mathscr{C}^0(\R_+;L^1(\T))$. \\
 $(ii)$ Let $u$ and $v$ be solutions of
 \eqref{eq:existence} with 
 respective initial data $u_0,v_0\in L^\infty(\T)$, and $u_0(x)\le v_0(x)$ for almost
 all $x\in \T$. Then
 \begin{equation}
 	\label{Comparison-Gale}
   u(t,x)\le v(t,x),\quad \forall t\ge 0,  \ \text{ a.e. } x\in \T. 
 \end{equation}
	Besides,
\begin{equation} 
	\label{ComparisonCstts}
			|u(t,x)| \leq \| u_0 \|_{L^\infty(\T)}, \quad
                        \forall t\ge 0,\ 
                        \text{a.e. } x\in \T. 
\end{equation}
$(iii)$ If $h^{(1)}$ and $h^{(2)}$ satisfy Assumption~\ref{hyp:gen}, and in
addition,
\begin{equation*}
 (0\le )\,   h^{(1)}(x,u)\le  h^{(2)}(x,u),\quad \text{a.e. }(x,u)\in \T\times (0,\infty),
\end{equation*}
then denoting by
        $u_1$ and $u_2$ the respective solutions to
        \eqref{eq:existence} with the same initial
        datum $u_0 \in L^\infty(\T)$, $u_0\ge 0$, we have 
		 \begin{equation}
		 \label{Comparisons-u1-u2}
			   u_1(t,x)\ge u_2(t,x)\ge 0,\quad
\forall t\ge 0,  \ \text{ a.e. } x\in \T. 
		 \end{equation}
$(iv)$ The same properties hold true for solutions of
\eqref{eq:flux-general}--\eqref{eq:ci} for $\alpha=1$, when
considering solutions in the sense of Definition
\ref{def:solution-alph=1} and $a$ satisfies Assumption \ref{hyp:a}. 
\end{proposition}
\begin{proof}
  When Assumption \ref{hyp:gen} is satisfied, we follow very closely
  the approach of \cite[Section~6.3]{Da00}, 
  based on the method of vanishing viscosity.
\smallbreak

\noindent{\bf  $\bullet$ Uniqueness.} For $u$ and $\uu$ two solutions
in the sense of Definition~\ref{def:solution-gen}, introduce the
entropy-entropy flux pair
\begin{equation*}
  \eta(u,\uu)=|u-\uu|,\quad q(u,\uu) =\mathrm{sign}
  (u-\uu)\(f(u)-f(\uu)\). 
\end{equation*}
This is an entropy-entropy flux pair for $u$ when $\uu$ is
fixed, and, conversely, for $\uu$ when $u$ is fixed. For a
nonnegative test function $\phi=\phi(t,x,\ut,\ux)$,
following successively the two  points of view mentioned above, we have,
\begin{align*}
  \int_0^T\int_\T &\(\d_t \phi (t,x,\ut,\ux)\eta
  \(u(t,x),\uu (\ut,\ux)\) +\d_x \phi  (t,x,\ut,\ux) q\(u(t,x),\uu
                    (\ut,\ux)\)\)dxdt\\
 & -\int_0^T\int_\T \phi
                    (t,x,\ut,\ux)\mathrm{sign}\( 
  u(t,x)-\uu (\ut,\ux)\)  h\(x,u(t,x)\)dxdt\\
&
+\int_\T \phi (0,x,\ut,\ux)\eta\(u_0(x),\uu
  (\ut,\ux)\)  dx\ge 0, \quad (\ut, \ux) \in \R_+ \times \T, 
\end{align*}
\begin{align*}
  \int_0^T\int_\T &\(\d_{\ut} \phi (t,x,\ut,\ux)\eta
  \(u(t,x),\uu (\ut,\ux)\) +\d_{\ux} \phi  (t,x,\ut,\ux) q\(u(t,x),\uu
                    (\ut,\ux)\)\)d\ux d\ut\\
 & -\int_0^T\int_\T \phi
                    (t,x,\ut,\ux)\mathrm{sign}\( 
  u(t,x)-\uu (\ut,\ux)\)  h\(x,u(t,x)\)d\ux d\ut\\
&
+\int_\T \phi (t,x,0,\ux)\eta\(u(t,x),\uu_0
  (\ux)\)  d\ux\ge 0,  \quad (t, x) \in \R_+ \times \T. 
\end{align*}
Integrating each of these inequalities with respect to the last two
variables and summing up the two resulting inequalities yields
\begin{align*}
 & \int_0^T\!\!\!\int_0^T\!\!\!\iint_{\T^2} \(\d_t+\d_{\ut}\)  \phi (t,x,\ut,\ux)\eta
  \(u(t,x),\uu (\ut,\ux)\)   dxd\ux dtd\ut \\
&+\int_0^T\!\!\!\int_0^T\!\!\!\iint_{\T^2} \(\d_x+\d_{\ux}\)  \phi (t,x,\ut,\ux)q
  \(u(t,x),\uu (\ut,\ux)\)   dxd\ux dtd\ut \\
&- \int_0^T\!\!\!\int_0^T\!\!\!\iint_{\T^2} \phi
                    (t,x,\ut,\ux) \mathrm{sign}\( 
  u(t,x)-\uu (\ut,\ux)\) \( h\(x,u(t,x)\) - h\(
   \ux,\uu(\ut,\ux)\) \)dxd\ux dtd\ut\\
& +\int_0^T\!\!\!\iint_{\T^2} \phi (0,x,\ut,\ux)\eta\(u_0(x),\uu
  (\ut,\ux)\)  dxd\ux d\ut \\
&+\int_0^T\!\!\! \iint_{\T^2} \phi (t,x,0,\ux)\eta\(u(t,x),\uu_0
  (0,\ux)\)  dxd\ux dt\ge 0.
\end{align*}
Pick $\phi$ of the form
\begin{equation*}
  \phi(t,x,\ut,\ux) = \frac{1}{\eps^2}\psi\(\frac{t+\ut}{2}\)
  \rho\(\frac{t-\ut}{2\eps}\)\rho\(\frac{x-\ux}{2\eps}\) ,
\end{equation*}
where the nonnegative function $\psi$ depends only on time,  and the
nonnegative compactly supported function $\rho$ is such that $\int_\R
\rho =1$,  with an obvious abuse of notation for the last factor
above. 
Letting $\eps\to 0$, we find:
\begin{align*}
  &\int_0^T \int_\T \psi'(t) \eta\(u(t,x),\uu(t,x)\) dxdt\\
&-
\int_0^T\int_{\T} \psi(t) \mathrm{sign}\( 
  u(t,x)-\uu (t,x)\) \( h\(x,u(t,x)\) - h\(
  x,\uu(t,x)\) \)dx dt\\
&+\int_\T
    \psi(0)\eta\(u_0(x),\uu_0(x)\)dx\ge 0. 
\end{align*}
In view of the third point in Assumption~\ref{hyp:gen}, this implies
\begin{equation*}
  \int_0^T \int_\T \psi'(t) \eta\(u(t,x),\uu(t,x)\) dxdt+\int_\T
    \psi(0)\eta\(u_0(x),\uu_0(x)\)dx\ge 0. 
\end{equation*}
For $0<\tau< T$ and $k\in \N$, we define $\psi=\psi_k$ as
\begin{equation*}
  \psi_k(t)=
\left\{
  \begin{aligned}
    1&\text{ if }0\le t<\tau,\\
k(\tau-t)+1&\text{ if } \tau\le t<\tau+\frac{1}{k},\\
0&\text{ if }\tau+\frac{1}{k}\le t <T. 
  \end{aligned}
\right.
\end{equation*}
Letting $k\to \infty$ now yields
\begin{equation*}
  \int_\T
    \eta\(u_0(x),\uu_0(x)\)dx\ge \int_\T
    \eta\(u(\tau,x),\uu(\tau,x)\)dx,
\end{equation*}
that is $\|u(\tau)-\uu(\tau)\|_{L^1(\T)}\le \|u_0-\uu_0\|_{L^1(\T)}$,
hence uniqueness for solutions in the sense of
Definition~\ref{def:solution-gen}, since $\tau\in (0,T)$ is arbitrary.
\smallbreak

\noindent {\bf  $\bullet$ Viscous approximation.} For $\mu>0$, consider the equation
\begin{equation}
  \label{eq:viscosite}
  \d_t u_\mu+\d_x f(u_\mu)+h(x,u_\mu)=\mu
\d_x^2 u_\mu \text{ in } \R_+ \times \T, \quad u_{\mu\mid t=0}=u_0 \text{ in } \T.
\end{equation}
For a fixed $\mu>0$, the solution to \eqref{eq:viscosite} is obtained
by a fixed point argument applied to the associated Duhamel's formula,
\begin{equation}
  \label{eq:duhamel-visc}
  u_\mu(t,x) = e^{\mu t\d_x^2}u_0(x) -\int_0^t e^{\mu
    (t-s)\d_x^2}\(\d_xf(u_\mu(s,x))+h\(x,u_\mu(s,x)\)\)ds,
\end{equation}
where we recall that the heat semigroup on $\T$ acts on Fourier series as
\begin{equation*}
  e^{\mu t\d_x^2}\(\sum_{n\in \Z} a_n e^{i2\pi nx}\) = \sum_{n\in \Z} a_n
  e^{i2\pi nx-4\mu \pi^2 n^2 t}.
\end{equation*}
Up to considering the linear heat flow on $\R$ and its uniqueness
property to relate it to the heat flow on $\T$, we readily see that if
$u_0\in L^\infty(\T)$, there exists $T>0$ depending on
$\|u_0\|_{L^\infty(\T)}$ only and a unique solution $u_\mu\in
  \mathscr{C}^0([0,T];L^\infty(\T))$ to \eqref{eq:duhamel-visc}.
\smallbreak
 
\noindent{\bf  $\bullet$ A priori estimate}. The solution is
  global in time, $u_\mu \in
  \mathscr{C}^0(\R_+;L^\infty(\T))$, in view of the a priori estimate
  \begin{equation}\label{eq:apriorivisc}
    \|u_\mu(t)\|_{L^\infty(\T)}\le \|u_0\|_{L^\infty(\T)},\quad
    \forall t\ge 0,
  \end{equation}
which we now establish. For $p>2$, multiply \eqref{eq:viscosite} by
$|u_\mu|^{p-2}u_\mu$, and integrate over $\T$. This yields
\begin{align*}
  \frac{1}{p}\frac{d}{dt}\(\int_\T |u_\mu|^p \) +\int_\T
  f'(u_\mu)|u_\mu|^{p-2}u_\mu\d_x u_\mu &= -\int_\T
  h(x,u_\mu)|u_\mu|^{p-2}u_\mu\\
&\quad - \mu (p-1)\int_\T
  |u_\mu|^{p-2}(\d_xu_\mu)^2.
\end{align*}
If $g$ is such that $g'(y)=p f'(y)|y|^{p-2}y$, the second term on the
left hand side is
\begin{equation*}
  \int_\T g'(u_\mu)\d_x u_\mu = \int_\T \d_x g(u_\mu)=0.
\end{equation*}
The first term on the right hand side is non-positive in view of
Assumption~\ref{hyp:gen}, and the last term is obviously
non-positive. We infer, for all $p>2$,
\begin{equation*}
  \frac{d}{dt}\|u_\mu\|_{L^p(\T)}^p\le 0,
\end{equation*}
hence \eqref{eq:apriorivisc} by integrating in time and letting $p\to \infty$. 
\smallbreak

\noindent {\bf  $\bullet$ Entropy solution for \eqref{eq:existence}.} The next step
consists in showing that up to extracting a subsequence, 
$u_\mu$ converges to an entropy solution
(Definition~\ref{def:solution-gen}). The above uniqueness result shows that
actually, no extraction is needed. 
\smallbreak

 If $\eta$ is a
smooth convex entropy, with associated entropy flux $q$, multiplying
\eqref{eq:viscosite} by $\eta'(u_\mu)$ yields 
\begin{equation*}
  \d_t \eta (u_\mu)+\d_x q(u_\mu) + \eta'(u_\mu)h(x,u_\mu) =\mu \d_x^2
  \eta(u_\mu) -\mu \eta''(u_\mu)\(\d_x u_\mu\)^2. 
\end{equation*}
Multiply the above equation by a nonnegative  test function $\psi$ then
yields, after an integration by parts and since $\eta''\ge 0$,
\begin{align}
  \int_0^T\int_\T &\(\d_t \psi \eta(u_\mu) +\d_x \psi q(u_\mu) -\psi
  \eta'(u_\mu)h(x,u_\mu)\)dxdt +\int_\T \psi(0,x)\eta(u_0(x))dx \notag\\
&\quad \ge
  -\mu \int_0^T\int_\T \d_x^2 \psi \eta(u_\mu)dxdt. \label{EntropyViscous}
\end{align}
The same holds when we just have $\eta\in W^{1,\infty}$ by using an
approximating entropy, as in the next paragraph.\\
Therefore, if for some
sequence $(\mu_k)$ with $\mu_k \downarrow 0$, $(u_{\mu_k})_k$
converges to some function $u$, boundedly almost everywhere on
$[0,\infty)\times \T$, then $u$ is an admissible weak solution of
\eqref{eq:existence} on $[0,\infty)\times \T$. 
\smallbreak

\noindent {\bf $\bullet$  Comparison for the viscous solution.} 
Introduce as in the proof of \cite[Theorem~6.3.2]{Da00} the function
$\eta_\eps$ defined for $\eps>0$ by
\begin{equation*}
  \eta_\eps(w)= 
  \begin{cases}
    0 & \text{ for }-\infty<w\le 0,\\
\frac{w^2}{4\eps} &\text{ for } 0<w\le 2\eps,\\
w-\eps& \text{ for } 2\eps<w<\infty. 
  \end{cases}
\end{equation*}
If $u_\mu$ and $\bar u_\mu$ solve \eqref{eq:duhamel-visc}, then by multiplying by $\eta_\eps'(u_\mu-\bar u_\mu)$ the equation
satisfied by $u_\mu-\bar u_\mu$, we compute
\begin{multline*}
    \d_t   \eta_\eps (u_\mu-\bar u_\mu) +\d_x \( \eta_\eps'(u_\mu-\bar u_\mu)\(f(u_\mu)-f(\bar
  u_\mu)\) \)
  \\
  -\eta_\eps''  (u_\mu-\bar u_\mu) \(f(u_\mu)-f(\bar   u_\mu)\)\d_x(u_\mu-\bar u_\mu)
    \\
  =
	- \eta_\eps'(u_\mu-\bar u_\mu)\(h(x,u_\mu)-h(x,\bar u_\mu)\) +\mu
  \d_x^2\eta_\eps(u_\mu-\bar u_\mu) -\mu \eta_\eps''(u_\mu-\bar u_\mu) \(\d_x(u_\mu-\bar
  u_\mu)\)^2. 
  \end{multline*}
The new term compared to the proof of \cite[Theorem~6.3.2]{Da00} is of
course the first term of the right hand side (where $h$ is
present). We have more precisely, for $0<s<t<\infty$, after
integration on $(s,t)\times \T$, 
\begin{equation}\label{eq:signe}
\begin{aligned}
 & \int_\T \eta_\eps (u_\mu(t)-\bar u_\mu(t)) -  \int_\T \eta_\eps (u_\mu(s)-\bar
  u_\mu(s)) \\
&\leq \int_s^t \!\!\int_\T \eta_\eps''(u_\mu-\bar u_\mu) \(f(u_\mu)-f(\bar
  u_\mu)\)\d_x(u_\mu-\bar u_\mu)
  \\
  &\quad -\int_s^t  \!\!\int_\T  \eta_\eps'(u_\mu-\bar
  u_\mu)\(h(x,u_\mu)-h(x,\bar u_\mu)\). 
\end{aligned}
\end{equation}
In the limit $\eps\to 0$, the first  term on the right hand side goes
to zero, while the second goes to 
\begin{equation*}
 - \int_s^t  \!\!\int_\T  {\mathbf 1}_{u_\mu>\bar u_\mu}\(h(x,u_\mu)-h(x,\bar u_\mu)\).
\end{equation*}
By Assumption~\ref{hyp:gen} (third point), this term is non-positive,
and we infer
\begin{equation*}
  \int_\T \(u_\mu(t,x)-\bar u_\mu(t,x)\)_+dx \le \int_\T \(u_\mu(s,x)-\bar
  u_\mu(s,x)\)_+dx. 
\end{equation*}
By letting $s\to 0$, this implies 
\begin{equation}\label{eq:monotonieA}
  \int_\T \(u_\mu(t,x)-\bar u_\mu(t,x)\)_+dx \le \int_\T \(u_0(x)-\bar
  u_0(x)\)_+dx,
\end{equation}
and by interchanging the roles of $u_\mu$ and $\bar u_\mu$,
\begin{equation*}
  \|u_\mu(t)-\bar u_\mu(t)\|_{L^1(\T)}\le \|u_0-\bar u_0\|_{L^1(\T)}.
\end{equation*}
Also,  if 
\begin{equation*}
  u_0(x)\le \bar u_0(x), \quad \text{a.e. on }\T,
\end{equation*}
then \eqref{eq:monotonieA} yields
\begin{equation*}
  u_\mu(t,x)\le \bar u_\mu(t,x), \quad  \forall t\ge 0,  \ \text{ a.e. } x\in \T. 
\end{equation*}
This implies in particular uniqueness for \eqref{eq:viscosite}. 
\smallbreak

\noindent {\bf  $\bullet$ Compactness.} 
To obtain compactness in space, as in \cite{Da00},
we consider $\bar u_\mu(t,x)= u_\mu(t,x+y)$. In the case where
$h=0$ (or more generally if $h$ depends on $u$ only), then $\bar u_\mu$ is
a solution to \eqref{eq:viscosite}, so \eqref{eq:monotonieA} can be
used directly. In our case, and precisely because  we want to consider
spatially localized damping, such $\bar u_\mu$ does not solve
\eqref{eq:viscosite}, and we have to resume the
computations. Essentially, we go back to the previous computations,
and replace $\bar u_\mu(t,x)$ with $u_\mu(t,x+y)$, noticing that $h(x,\bar u_\mu)$
has to be replaced by $h(x+y,u_\mu(t,x+y))$. We have
\begin{align*}
  & \int_\T  (u_\mu(t,x)- u_\mu(t,x+y))_+dx -  \int_\T (u_0(x)-u_0(x+y))_+dx \le \\
&\quad 
\limsup_{\eps \to 0}\(-\int_0^t  \!\!\int_\T  \eta_\eps'(u_\mu-\bar
  u_\mu)\(h(x,u_\mu(\tau,x))-h(x+y, u_\mu(\tau,x+y))\)d\tau dx\). 
\end{align*}
In the above integral, insert $\pm h(x+y,u_\mu(\tau,x))$. By the same
argument as above (third point in Assumption~\ref{hyp:gen}), we infer
\begin{align*}
   & \int_\T  (u_\mu(t,x)- u_\mu(t,x+y))_+dx -  \int_\T (u_0(x)-u_0(x+y))_+dx \le \\
&\quad 
\limsup_{\eps \to 0}\(-\int_0^t  \!\!\int_\T  \eta_\eps'(u_\mu-\bar
  u_\mu)\(h(x,u_\mu(\tau,x))-h(x+y, u_\mu(\tau,x))\)d\tau dx\),
\end{align*}
hence
\begin{multline*}
	\int_\T  (u_\mu(t,x)- u_\mu(t,x+y))_+dx -  \int_\T (u_0(x)-u_0(x+y))_+dx 
	    \\ 
	    \le \int_0^t  \int_\T  \left|
	  h(x,u_\mu(\tau,x))-h(x+y, u_\mu(\tau,x))\right|d\tau dx .
\end{multline*}
In view of \eqref{eq:apriorivisc} and of the last point in
Assumption~\ref{hyp:gen}, we conclude 
\begin{equation*}
  \int_\T  (u_\mu(t,x)- u_\mu(t,x+y))_+dx \le  \int_\T (u_0(x)-u_0(x+y))_+dx +
  \O(y),
\end{equation*}
and
\begin{equation}\label{eq:compacitex}
  \int_\T  |u_\mu(t,x)- u_\mu(t,x+y)|dx \le  \int_\T |u_0(x)-u_0(x+y))|dx +
  \O(y).
\end{equation}
Equicontinuity in time is proved similarly by setting $\bar u_\mu(t,x) =
u_\mu(t+\tau,x)$. Since $h$ depends on $x$ and $u_\mu$ only, the only extra term
that we have to estimate is of the form
\begin{equation*}
 \left| \int_t^{t+\tau}\!\!\int_\T h(x,u_\mu(s,x))\phi(x)dsd\tau\right|\le
 \tau C\(\|u_0\|_{L^\infty(\T)}\) \|\phi\|_{L^1(\T)},
\end{equation*}
where we have used \eqref{eq:apriorivisc}. 
\smallbreak

The above properties imply that the sequence $(u_\mu)_\mu$ is uniformly bounded and equicontinuous in $(0,\infty) \times \T$, so there is a subsequence
of $(u_\mu)_\mu$, which converges boundedly almost everywhere on
$(0,\infty)\times \T$ and strongly in $L^1_{\rm loc}((0,\infty)\times
\T)$. We infer from \eqref{EntropyViscous} that the limit $u$ is \emph{an} entropy solution, hence \emph{the} entropy solution by uniqueness of the entropy solution. We have thus proved the item $(i)$, while item $(ii)$ follows from
\eqref{eq:monotonieA} and \eqref{eq:apriorivisc}, after passing to the limit $\mu \to 0$.

\begin{remark}
  Dividing \eqref{eq:compacitex} by $y$ yields the propagation of $BV$
  regularity mentioned in Remark~\ref{rem:BV}. 
\end{remark}
\smallbreak

\noindent {\bf $\bullet$ Comparison when source terms are ordered.} It
remains to 
prove $(iii)$.  Since we assume $u_0\ge 0$, we know from $(ii)$ that
$u_1(t,x),u_2(t,x)\ge 0$ for $(t,x)\in (0,\infty)\times \T$, and so
$h^{(2)}(x,u_j)\ge h^{(1)}(x,u_j)$ for $j=1,2$. We then consider the viscous approximations $u_{1, \mu}$ and $u_{2,\mu}$ of, respectively, $u_1$ and $u_2$. The analogue of
\eqref{eq:signe} reads
\begin{align*}
 & \int_\T \eta_\eps (u_{2,\mu}(t)- u_{1,\mu}(t)) -  \int_\T \eta_\eps (u_{2,\mu}(s)-
  u_{1,\mu}(s))\le \\
& \int_s^t \!\!\int_\T \eta_\eps''(u_{2,\mu}- u_{1,\mu}) \(f(u_{2,\mu})-f(u_{1,\mu})\)
\d_x(u_{2,\mu}-u_{1,\mu})\\
&-\int_s^t  \!\!\int_\T  \eta_\eps'(u_{2,\mu}-
  u_{1,\mu})\(h^{(2)}(x,u_{2,\mu})-h^{(1)}(x,u_{1,\mu})\). 
\end{align*}
Passing to the limits $\eps\to 0$, $s\to 0$, and finally $\mu\to 0$, we
infer, since $u_1$ and $u_2$ have the same initial datum:
\begin{equation*}
  \int_\T  \(u_2(t)- u_1(t)\)_+ \le 
-\int_0^t  \!\!\int_\T \mathbf{1}_{u_2>u_1} \(h^{(2)}(x,u_2)-h^{(1)}(x,u_1)\). 
\end{equation*}
Now since the integrand of the right hand side can be decomposed as
\begin{equation*}
 \underbrace{\mathbf{1}_{u_2>u_1} \(h^{(2)}(x,u_2)-h^{(2)}(x,u_1)\)}_{\ge 0, \text{ by
    Assumption~\ref{hyp:gen}}} +\mathbf{1}_{u_2>u_1} 
\underbrace{\(h^{(2)}(x,u_1)-h^{(1)}(x,u_1)\)}_{\ge 0, \text{ from above}},
\end{equation*}
we conclude that $\int_\T  \(u_2(t)- u_1(t)\)_+ =0$, hence
$u_1\ge u_2\ge 0$ as announced.
\smallbreak

\noindent {\bf $\bullet$ Item $(iv)$: the case $h(x, u) = a(x)
  u/|u|$.}\\
In this case, the uniqueness of admissible weak solutions in the sense
of Definition \ref{def:solution-alph=1} holds without change. The
difficulty then is to prove existence of admissible weak solutions. In
order to do that, instead of approximating \eqref{eq:flux-general} by
its viscous approximation \eqref{eq:viscosite}, we also add an
approximation of the function $h$. Namely, we consider the
approximation given, for $\mu \in (0,1)$, by  
\begin{equation}
  \label{eq:viscosite-mu-1}
  \d_t u_\mu+\d_x f(u_\mu)+h_{1- \mu} (x,u_\mu)=\mu
\d_x^2 u_\mu \text{ in } \R_+ \times \T, \quad u_{\mu\mid t=0}=u_0 \text{ in } \T, 
\end{equation}
where $h_{1 - \mu}(x,u)$ is given by \eqref{eq:notre-h}. For each $\mu
>0$, all the computations performed above can be repeated, so that the sequence
of solutions $(u_\mu)_{\mu}$ is uniformly bounded and equicontinuous
in $(0,\infty) \times \T$, so up to some subsequence, it converges
boundedly  almost everywhere on $(0,\infty) \times \T$ and strongly in
$L^1_{\rm loc} ((0,\infty)\times \T)$ to some $u$ as $\mu \to 0$, so
that a.e. $(t,x) \in (0,\infty) \times \T$,  
\[
  h_{1- \mu} (x, u_\mu(t,x)) \Tend \mu  0 h(t,x),
\]
where $h(t,x) = a(x) \frac{u(t,x)}{|u(t,x)|}$ if $u(t,x) \neq 0$, 
and $h(t,x) \in [-1,1]$ if $u(t,x) = 0$.

%
%
It follows that $u$ is an admissible weak solutions of
\eqref{eq:flux-general}--\eqref{eq:ci} in the sense of
Definition~\ref{def:solution-alph=1}, hence the admissible weak
solution in the sense of Definition~\ref{def:solution-alph=1} by
uniqueness. The comparison results can then be proved as before, by
studying them for the solutions of \eqref{eq:viscosite-mu-1} and
passing to the limit $\mu \to 0$. 
\end{proof}

\section{Proof of Proposition \ref{prop:partout}: The case of a damping acting everywhere}
\label{sec:partout}

\begin{proof}[Proof of Proposition \ref{prop:partout}]
  Let $p>2$. In view of \eqref{eq:flux-general}, we formally have
  \begin{equation*}
    \frac{d}{dt}\int_\T|u(t,x)|^pdx = p\int_\T |u|^{p-2}u\d_t u dx =
    -p\int_\T |u|^{p-2}u \d_x f(u)dx-p\int_\T a(x)|u|^{p-\alpha} dx.
  \end{equation*}
Writing formally $p |u|^{p-2}u \d_x f(u) = p |u|^{p-2}u f'(u)\d_xu$, we have
\begin{equation*}
 p  |u|^{p-2}u \d_x f(u)  = \d_x g_p(u),\quad \text{with
  }g_p'(z)=p |z|^{p-2}zf'(z),
\end{equation*}
and so 
\begin{equation*}
  \int_\T |u|^{p-2}u \d_x f(u)dx=0.
\end{equation*}
As a matter of fact, this reasoning is valid only for sufficiently
smooth solutions. However, the conclusion remains true in our context, as can
be seen by using Definitions~\ref{def:solution} and \ref{def:solution-alph=1} with $p>2$, $\eta(u) = |u|^{p} u$, $\psi(t,x) = 1$:
\begin{equation*}
   \frac{d}{dt}\int_\T|u(t,x)|^pdx \leq -p\int_\T a(x)|u|^{p-\alpha} dx \le
   -\frac{p\delta }{\|u(t)\|_{L^\infty(\T)}^\alpha }\int_\T|u(t,x)|^pdx ,
\end{equation*}
that is
\begin{equation*}
  \frac{d}{dt}\ln \|u(t)\|_{L^p(\T)}\le -\frac{\delta}{\|u(t)\|_{L^\infty(\T)}^\alpha }.
\end{equation*}
By integration,
\begin{equation*}
  \|u(t)\|_{L^p(\T)}\le \|u_0\|_{L^p(\T)} \exp\(-\delta\int_0^t
  \frac{ds}{\|u(s)\|_{L^\infty(\T)}^\alpha }\), 
\end{equation*}
and by letting $p\to \infty$,
\begin{equation}\label{eq:faux-gronwall}
  \|u(t)\|_{L^\infty(\T)}\le \|u_0\|_{L^\infty(\T)} \exp\(-\delta\int_0^t
  \frac{ds}{\|u(s)\|_{L^\infty(\T)}^\alpha }\).
\end{equation}
Set 
\begin{equation*}
  \Phi(t) = \int_0^t
  \frac{ds}{\|u(s)\|_{L^\infty(\T)}^\alpha }.
\end{equation*}
The above inequality reads
\begin{equation*}
  \Phi'(t)\ge \frac{1}{\|u_0\|_{L^\infty(\T)}^\alpha}e^{\alpha\delta
    \Phi(t)}. 
\end{equation*}
Since the solution to $\Psi'= Ce^{\alpha\delta\Psi}$, $\Psi(0)=0$, is given
by
\begin{equation*}
  \Psi(t) = -\frac{1}{\alpha\delta}\ln \(1-C\alpha \delta t\),\quad t\le
  \frac{1}{C\alpha\delta }, 
\end{equation*}
we conclude by comparison that $\Phi(t)\to +\infty$ as $t\to
 \|u_0\|_{L^\infty(\T)}^\alpha/\(\alpha \delta\)$, hence the result thanks
to \eqref{eq:faux-gronwall}, with 
\begin{equation*}
  T= \frac{\|u_0\|_{L^\infty(\T)}^\alpha}{\alpha \delta}, 
\end{equation*}
which may not be the sharp extinction time, but an upper bound for it. 
\end{proof}
\begin{remark}[Whole space]
  The above argument shows that the conclusion of
  Proposition~\ref{prop:partout}  remains valid if
  \eqref{eq:partout} is set up on the whole line, $x\in \R$, provided
  that the solution which we consider goes to zero at $\pm \infty$. 
  \\
  Otherwise, a similar proof can be given,  by considering estimates of $u(t)$ in $L^p(A_-(t), A_+(t))$ where 
  \[
  	A_- (t) = A_-^0 +  t \sup_{s \in [-K,K]} f'(s)\quad \hbox{ and } \quad A_+(t) = A_+^0 + t \sup_{s \in [-K,K]} f'(s)
\]
for any pair $(A_-^0, A_+^0) \in \R^2$; see e.g. the proof of Lemma \ref{Lem:Transport-Extinction} where the same kind of arguments are developed.
\end{remark}
%



\section{Proof of Theorem \ref{thm:Transport}: the transport case}
\label{sec:transport}

The goal of this section is to prove Theorem \ref{thm:Transport}. Thus, we consider the setting of Theorem \ref{thm:Transport}, and we assume in particular that $f$ is smooth, satisfies \eqref{eq:transport}, $K$ satisfies \eqref{eq:def-K}, and $a$ satisfies Assumption \ref{hyp:a} and \eqref{eq:local}.

\subsection{Strategy}
\label{subsec:transport-results}

In order to ease the reading of the proof of Theorem
\ref{thm:Transport}, we decompose it into two lemmas, the first one
stating that in the setting of Theorem \ref{thm:Transport} the
solutions of \eqref{eq:flux-general}--\eqref{eq:ci} decay
exponentially, while the second one will show that if the initial
datum is small enough, then the corresponding solution of
\eqref{eq:flux-general}--\eqref{eq:ci} vanishes in finite time. 
\begin{lemma}
	\label{Lem:Transport-decay}
	Within the setting of Theorem \ref{thm:Transport}, there exist
        $C >0$ and $\mu >0$ such that for any initial datum $u_0 \in
        L^\infty(\T)$ satisfying \eqref{eq:ci-small}, any solution $u$
        of \eqref{eq:flux-general}--\eqref{eq:ci} satisfies 
	\begin{equation}
		\label{Exp-decay}
		\| u(t) \|_{L^\infty(\T)} \le C e^{- \mu t} \| u_0 \|_{L^\infty(\T)}. 
	\end{equation}
\end{lemma}
\begin{lemma}
	\label{Lem:Transport-Extinction}
	Within the setting of Theorem \ref{thm:Transport}, there exist $\varepsilon_0>0$ and $T_0>0$ such that if 
	\begin{equation}
		\label{eq:ci-very-small}
		\|u_0\|_{L^\infty(\T)} \le \varepsilon_0, 
	\end{equation}
	the solution $u$ of \eqref{eq:flux-general}--\eqref{eq:ci} satisfies
	\begin{equation}
		\label{eq:null-after-time-T0}
		u(t, x) = 0, \quad \forall t \ge T_0,\, \text{ a.e. }\, x \in \T.
	\end{equation}
\end{lemma}
The proofs of Lemma~\ref{Lem:Transport-decay} and
Lemma~\ref{Lem:Transport-Extinction} are  given in
Subsections~\ref{subsec:Proof-Lemma-Exp-Decay} and
\ref{subsec:Proof-Lem-Extinct}, respectively. The proof of Theorem
\ref{thm:Transport} is then given in 
Subsection~\ref{subsec:Proof-Thm-Transport}. 
\subsection{Proof of Lemma \ref{Lem:Transport-decay}}
\label{subsec:Proof-Lemma-Exp-Decay}

\begin{proof}
	We first choose a function $\varphi = \varphi(x)$ such that
	\begin{equation*}
		\varphi \in \mathscr{C}^\infty(\T), \quad\text{ with }
				\varphi(x) = x, \quad \forall x \in \T\setminus \omega.
	\end{equation*}
	Note that such a function $\varphi$ satisfies in particular that 
	\[
		\forall x \in \T \setminus \omega, \quad \partial_x \varphi (x) = 1, 
		\qquad
		\hbox{ and } 
		\qquad 
		\partial_x \varphi \in  \mathscr{C}^0(\T). 
	\]
	Now, let $u_0 \in L^\infty(\T)$ and $u$ the corresponding
        solution of \eqref{eq:flux-general}--\eqref{eq:ci}. We
        consider the Lyapunov functionals, indexed by $p\ge 2$ and
        some parameter $\lambda \in \R$ chosen later,   
	\begin{equation}
		\label{eq:def-E-p-lambda}
		E_{p,\lambda}(t) = \int_\T \left| e^{- \lambda \varphi(x)} u(t,x) \right|^p \, dx.
	\end{equation}
	Formally, these  functionals satisfy: 
	\begin{align*}
		\frac{dE_{p,\lambda}(t)}{dt} 
		& = 
		p \int_\T e^{-p \lambda \varphi(x)} |u(t,x)|^{p-2} u(t,x) \partial_t u(t,x) \, dx
		\\
		& =  
		- p \int_\T e^{-p \lambda \varphi(x)} |u(t,x)|^{p-2} u(t,x) \partial_x (f(u(t,x)) \, dx
		\\
		& \quad \hfill 
		- p \int_\T a(x) e^{-p \lambda \varphi(x)} |u(t,x)|^{p-\alpha} \, dx. 
	\end{align*}
	If $u$ were smooth, we would write
	\begin{align*}
		p |u(t,x)|^{p-2} u(t,x) \partial_x (f(u(t,x)) 
		& = 
		p f'(u(t,x)) |u(t,x)|^{p-2} u(t,x) \partial_x u(t,x) 
		\\
		&=  
		\partial_x (g_p(u(t,x))), 
	\end{align*}
	where $g_p$ is defined by 
	\begin{equation}
		\label{eq:def-g}
		g_p( s ) = p \int_0^s f'(\tau) |\tau|^{p-2} \tau \, d\tau, 
	\end{equation}
	so that we would write: 
	\begin{multline*}
		- p \int_\T e^{-p \lambda \varphi(x)} |u(t,x)|^{p-2} u(t,x) \partial_x (f(u(t,x)) \, dx
		\\
		= 
		- p \lambda \int_\T \partial_x \varphi (x) e^{-p \lambda \varphi(x)} g_p(u(t,x))\, dx. 
	\end{multline*}
Note in passing that $g_p$ has the same sign as $f'(0)$.
	As solutions $u$ may contain shocks, these estimates should be
        justified by using the definition of admissible weak
        solutions, i.e. inequalities \eqref{AdmissibleWeak-Identity}
        or \eqref{AdmissibleWeak-Identity-alph=1}, choosing $ p > 2$,
        $\eta(u) = |u|^{p} u$ and $\psi(t,x) = e^{ - p \lambda
          \varphi(x)}$. One obtains in that way: 
	\begin{align*}
		\frac{dE_{p,\lambda}(t)}{dt} 
		& \le 
		- p \lambda \int_\T \partial_x \varphi (x) e^{-p \lambda \varphi(x)} g_p(u(t,x))\, dx
		\\
		& \quad
		- p \int_\T a(x) e^{-p \lambda \varphi(x)} |u(t,x)|^{p-\alpha} \, dx
		\\
		& \le 
		- p \lambda \int_\T e^{-p \lambda \varphi(x)} g_p(u(t,x))\, dx
		\\
		& \quad
		+ p |\lambda| \| \partial_x \varphi -
                  1\|_{L^\infty(\T)}\int_\omega e^{-p \lambda
                  \varphi(x)} \left\lvert g_p(u(t,x))\right\rvert \, dx 
		\\
		& \quad
		- p \delta \int_\omega e^{-p \lambda \varphi(x)} |u(t,x)|^{p-\alpha}  \, dx. 
	\end{align*}
	It is thus natural to study the function $g_p$ in
        \eqref{eq:def-g}. In order to do this, we first note that if
        $\|u_0 \|_{L^\infty(\T)} \le K$, according to
        Proposition~\ref{prop:comparaison}, for all time $t\ge 0$, the
        $L^\infty$-norm of $u(t)$ is bounded by $K$. Therefore, we
        introduce  
	\begin{equation*}
		\beta_- = \inf_{s \in [-K,K]} |f'(s)| \quad \hbox{ and } \quad \beta_+ = \sup_{s \in [-K,K]} |f'(s)|, 
	\end{equation*}
	so that for all $(t,x) \in [0,\infty) \times \T$, 
	\begin{equation}
		\label{Bounds-g-p}
		\beta_- |u(t,x)|^p \le \left| g_p(u(t,x)) \right| \le \beta_+ | u(t,x)|^p.
	\end{equation}
	Therefore, we obtain
	\begin{align*}
		\frac{dE_{p,\lambda}(t)}{dt} 
		& \le 
		- p \lambda \int_\T e^{-p \lambda \varphi(x)} g_p(u(t,x))\, dx
		\\
		& \quad
		+ 
		p \left(
			|\lambda| \| \partial_x \varphi - 1\|_{L^\infty(\T)} \beta_+ - \delta \|u(t)\|_{L^\infty(\T)}^{-\alpha}
			\right)
		\int_\omega e^{-p \lambda \varphi(x)} |u(t,x)|^{p}  \, dx.
	\end{align*}
	In particular, for all $t \ge 0$, if 
	\begin{equation}
		\label{eq:decay-condition}
		|\lambda| \| \partial_x \varphi - 1 \|_{L^\infty(\T)} \beta_+ \le \delta \|u(t) \|_{L^\infty(\T)}^{-\alpha}, 
	\end{equation}
	we have 
	\begin{equation}
		\label{eq:decay-e-lambda}
		\frac{dE_{p,\lambda}(t)}{dt} 
		\le 
		- p \lambda \int_\T e^{-p \lambda \varphi(x)} g_p(u(t,x))\, dx.
	\end{equation}
	%
	%
	As for all $t$, $\|u(t)\|_{L^\infty(\T)} \le K$, we therefore choose 
	\begin{equation*}
		\lambda_* = \hbox{sign\,} (f'(0)) \frac{\delta K^{-\alpha}}{ \| \partial_x \varphi - 1 \|_{L^\infty(\T)} \beta_+}, 
	\end{equation*}
	so that \eqref{eq:decay-condition} is satisfied and \eqref{eq:decay-e-lambda} becomes: 
	\begin{equation*}
		\frac{dE_{p,\lambda_*}(t)}{dt} 
		\le
		 - p |\lambda_*| \beta_- E_{p, \lambda*}(t).
	\end{equation*}
	Therefore, we get that
	\begin{equation*}
		\| e^{- \lambda_* \varphi} u(t) \|_{L^p(\T)} 
		\le
		e^{-|\lambda_*| \beta_- t} \| e^{- \lambda_* \varphi} u_0 \|_{L^p(\T)}, \quad \forall t \ge 0.
	\end{equation*}
	As $\lambda_*$ does not depend on $p$, we can pass to the limit as $p\to \infty$, and obtain:
	\begin{equation*}
		\| e^{- \lambda_* \varphi} u(t) \|_{L^\infty(\T)} 
		\le
		e^{-|\lambda_*| \beta_- t} \| e^{- \lambda_* \varphi} u_0 \|_{L^\infty(\T)}, \quad \forall t \ge 0, 
	\end{equation*}
	and thus, 
	\begin{equation*}
		\| u(t) \|_{L^\infty(\T)} 
		\le
		e^{|\lambda_*| (\sup \varphi - \inf \varphi ) -|\lambda_*| \beta_- t} \| u_0 \|_{L^\infty(\T)}, \quad \forall t \ge 0. 
	\end{equation*}
	This concludes the proof of Lemma \ref{Lem:Transport-decay}, as $|\lambda_*| >0$.
\end{proof}
\begin{remark}
	Note that one can go further and show that the $L^\infty(\T)$ norm of the solutions of \eqref{eq:flux-general}--\eqref{eq:ci} decays in fact faster than an exponential in time. 
	\\
	Indeed, if we set
	\begin{equation*}
		T_1 = \frac{2}{\beta_- } (\sup \varphi- \inf \varphi), 
	\end{equation*}
	and use the explicit choice 
	\begin{equation*}
		\lambda_* = \hbox{sign\,} (f'(0)) \frac{\delta \| u_0 \|_{L^\infty(\T)}^{-\alpha}}{ \| \partial_x \varphi - 1 \|_{L^\infty(\T)} \beta_+}, 
	\end{equation*}
	which is admissible according to the above proof, 
	one in fact gets
	\begin{align*}
		\| u(T_1) \|_{L^\infty(\T)} 
		 & \le 
		e^{- |\lambda_*| (\sup \varphi- \inf \varphi)} \| u_0 \|_{L^\infty(\T)}
		\\
		& \le 
		\exp\left( - c_0 \|u_0\|_{L^\infty(\T)}^{-\alpha} \right) \|u_0\|_{L^\infty(\T)}, 
	\end{align*}
	where $c_0$ is given by 
	\begin{equation*}
		c_0 = \delta \frac{ (\sup \varphi- \inf \varphi) }{ \| \partial_x \varphi - 1 \|_{L^\infty(\T)} \beta_+}.
	\end{equation*}
	Starting from there and using the semi-group property, we introduce a sequence of time, indexed by $n \in \N$, 
	\begin{equation*}
		T_n = n T_1 \quad  \left(=  \frac{2 n}{ \beta_- } (\sup \varphi- \inf \varphi) \right), 
	\end{equation*}
	for which one gets immediately, for all $n \in \N$,
	\begin{equation*}
		\| u(T_{n+1}) \|_{L^\infty(\T)} 
		\le
		\exp\left( - c_0 \|u(T_n)\|_{L^\infty(\T)}^{-\alpha} \right) \|u(T_n)\|_{L^\infty(\T)}. 
	\end{equation*}
	It is then easy to check that the sequence $(\| u(T_n)\|_{L^\infty(\T)})_{n \in \N}$ goes to $0$ faster than any (non-trivial) geometric sequence, which in turn implies that the map $t \mapsto \|u(t)\|_{L^\infty(\T)}$ goes to $0$ faster than any exponential.
	
	It would be interesting to develop a direct proof of Theorem \ref{thm:Transport} based only on a suitable choice of Lyapunov functionals in the spirit of the one used above. 
\end{remark}
\subsection{Proof of Lemma \ref{Lem:Transport-Extinction}}
\label{subsec:Proof-Lem-Extinct}
\begin{proof}
To simplify the presentation,  in the proof of this lemma, 
 $\T$ is identified with an interval centered in $0$, and $\omega$
is identified with an interval of the form $(-A,A)$.

	The proof of Lemma \ref{Lem:Transport-Extinction} is divided
        in two steps. In the first step, we show that if
        $\varepsilon_0>0$ is chosen small enough, then necessarily,
        the solution $u$ of \eqref{eq:flux-general}--\eqref{eq:ci}
        with $u_0 \in L^\infty(\T)$ satisfying $\|u_0\|_{L^\infty(\T)}
        \leq \varepsilon_0$ vanishes in some part of the domain
        $\omega$ after some (small) time. We then show that this
        implies that the solution $u$ vanishes everywhere
        after some time. 
	%
	%
	%
	\smallbreak

	\noindent {\bf $\bullet$ Step 1.} We  introduce the paths 
	\begin{equation*}
		A_-(t) = \sup_{[-K,K]} \{ f'\} t - \frac{A}{2}, \quad A_+(t) = \inf_{[-K,K]}\{f' \} t + \frac{A}{2}.
	\end{equation*}
	We fix $\tau_*$ such for all $t \in [0,\tau_*)$,  
	\[
		-A < A_-(t) < A_+ (t) < A, 
	\]
	that is 
	\begin{equation}
		\label{Def-T-*}	
		\tau_* = \frac{A}{\max \left\{2 \inf_{[-K,K]} |f' | , \ 
 \sup_{[-K,K]} |f' | - \inf_{[-K,K]} | f' | \right\}}.
	\end{equation}
	We then set, for $p \ge 2$ and $t \in [0, \tau_*]$, 
	\begin{equation*}
		\label{eq:def-e-p-loc}
		E_{p,\rm loc}(t)  = \int_{A_-(t)}^{A_+(t)} |u(t,x)|^p \, dx.
	\end{equation*}
	Arguing as in the proof of Proposition \ref{prop:partout}, as
        $[A_-(t), A_+(t)] \subset \omega$ for all $t \in [0, \tau_*]$ and
        $ \inf f' \le f'(u(t,x)) \le \sup f' $ for all $(t,x) \in
        \R_+\times \T$, we get, for all $t \geq 0$, 
	\begin{equation}
		\label{Decay-Energy-local}
		\frac{d}{dt} E_{p, \rm loc} (t) \le - \frac{ p
                  \delta}{\|u(t)\|_{L^\infty(A_-(t),A_+(t))}^\alpha}
                E_{p,\rm loc}(t).  
	\end{equation}
	In fact, to prove this estimate rigorously, we use the definition of admissible weak solutions, i.e. the inequalities \eqref{AdmissibleWeak-Identity} or \eqref{AdmissibleWeak-Identity-alph=1}, choosing $ p > 2$, $\eta(u) = |u|^{p} u$ and $\psi(t,x) = \varphi_\epsilon(t,x)$, where 
	\begin{multline*}
		\varphi_\epsilon(t,x) = \varphi_-^0 \left(\frac{x-
                    A_-(t)}{\epsilon}\right) {\mathbf 1}_{x \in (A_-(t)- \epsilon, A_-(t))} 
		+{\mathbf 1}_{x \in (A_-(t), A_+(t))} 
		\\
		+ \varphi_+^0 \left(\frac{x-
                    A_+(t)}{\epsilon}\right){\mathbf 1}_{x \in (A_+(t), A_+(t)+ \epsilon)},
	\end{multline*}
	with $\varphi_-^0$, $\varphi_+^0$ non-negative smooth cut-off
        functions, taking value $1$ on $\R_+$ and vanishing on
        $(-\infty, -1)$ for $\varphi_-^0$, taking value $1$ on $\R_-$
        and vanishing on $(1,\infty)$ for $\varphi_+^0$. We then pass
        to the limit $\epsilon \to 0$ in the inequalities
        \eqref{AdmissibleWeak-Identity} or
        \eqref{AdmissibleWeak-Identity-alph=1} to show the estimate
        \eqref{Decay-Energy-local}.  
	
	This yields, for all $t \in [0, \tau_*]$, 
	\begin{equation*}
		\label{eq:dynamics-e-p-loc}	
		\frac{d}{dt} \left(\log(\|u \|_{L^p(A_-(t),A_+(t))})
                \right) \le  - \frac{
                  \delta}{\|u(t)\|_{L^\infty(A_-(t),A_+(t))}^\alpha}. 
	\end{equation*}
	Integrating this expression and letting $p \to \infty$, we obtain, similarly as in \eqref{eq:faux-gronwall}, that for all $t \in [0, \tau_*]$, 
	\begin{equation*}
		\|u(t)\|_{L^\infty(A_-(t),A_+(t))} 
		\le \|u_0\|_{L^\infty(-A/2,A/2)} \exp\left(- \delta
                  \int_0^t
                  \frac{ds}{\|u(s)\|_{L^\infty(A_-(s),A_+(s))}^\alpha} \right).  
	\end{equation*}
	Arguing as in the proof of Proposition \ref{prop:partout}, this implies that $\|u(\tau_0) \|_{L^\infty(A_-( \tau_0), A_+( \tau_0))} = 0$ for 
	\begin{equation}	
		\label{def-tau-0}
		 \tau_0 = \frac{\varepsilon_0^\alpha}{\alpha \delta},
	\end{equation}
	 which is smaller than $\tau_*$ for $\varepsilon_0>0$ small enough, i.e. for small enough initial datum (recall \eqref{eq:ci-very-small}).
	\\
	Using the same argument on the solutions $u(\cdot + t_0)$ for all $t_0 \ge 0$, we see that in fact we have obtained
	\begin{equation}
		\label{eq:u-vanishes-somewhere}
		\forall t \ge \tau_0, \quad \|u( t ) \|_{L^\infty(A_-(\tau_0), A_+(\tau_0))} = 0.
	\end{equation}
	We end up this step by emphasizing that $A_-(\tau_0) < A_+(\tau_0)$, so that \eqref{eq:u-vanishes-somewhere} really implies that $u(t,\cdot)$ vanishes on a constant interval for all time $t \ge \tau_0$.
	\smallbreak
	
	%
	\noindent {\bf $\bullet$ Step 2.} In this step, to fix the
        ideas, we assume that $f'(0)>0$, as a completely similar proof
        can be adapted to the case $f'(0) <0$. We then look at the
        evolution of the $L^2$-norm of $u(t)$ on the set
        $(A_-(\tau_0), B(t))$, where $B(t) = A_+(\tau_0) + \beta_- (t
        - \tau_0)_+$, with $\beta_- = \inf_{[-K,K]} f' $ (recall that
        we have assumed $f'(0) >0$). Recall that 
        $u(t,x) = 0$ for all $t \ge \tau_0$ and $x \in [A_-(\tau_0),
        A_+(\tau_0)]$ according to \eqref{eq:u-vanishes-somewhere}. 
Using 
        the definition of admissible weak solutions, we infer:
	\begin{align*}
		\frac{d}{dt} \left( \int_{A_-(\tau_0)}^{B(t)} |u(t,x) |^2 \, dx\right)
		\le 0.
	\end{align*}
        Indeed, this comes from the definition of admissible weak solutions with $\eta(u)=|u|^2$ and $\psi(t,x)=\varphi_\epsilon(t,x)$, where
	\begin{multline*}
		\varphi_\epsilon(t,x) = \varphi_-^0 \left(\frac{x-
                    A_+(\tau_0)}{\epsilon}\right) {\mathbf{1}}_{x \in (A_+(\tau_0)- \epsilon, A_+(\tau_0))} 
		+{\mathbf{1}}_{x \in (A_+(\tau_0), B(t))}
		\\
		+ \varphi_+^0 \left(\frac{x-
                    B(t)}{\epsilon}\right){\mathbf{1}}_{x \in (B(t), B(t)+ \epsilon)},
	\end{multline*}
	with $\varphi_-^0$, $\varphi_+^0$ non-negative smooth cut-off
        functions, taking value $1$ on $\R_+$ and vanishing on
        $(-\infty, -1)$ for $\varphi_-^0$, taking value $1$ on $\R_-$
        and vanishing on $(1,\infty)$ for $\varphi_+^0$.
        \\
	Therefore, for all $t \ge \tau_0$, 
	\[
		\int_{A_-(\tau_0)}^{B(t)} |u(t,x) |^2 \, dx = 0.
	\]
	In particular, waiting a time $T_0> \tau_0$ such that $B(T_0)
        = A_-(\tau_0) + |\T|$,  for all $t \ge T_0$, for almost
        all $x \in \T$, $u(t,x) = 0$. This concludes the proof of
        Lemma~\ref{Lem:Transport-Extinction}. 
\end{proof}
\subsection{Proof of Theorem \ref{thm:Transport}}
\label{subsec:Proof-Thm-Transport}
\begin{proof}
	Theorem \ref{thm:Transport} easily follows from Lemma \ref{Lem:Transport-decay} and \ref{Lem:Transport-Extinction}. Indeed, if one chooses an initial datum $u_0$ satisfying \eqref{eq:ci-small}, the $L^\infty(\T)$-norm of the corresponding solution $u(t)$ of \eqref{eq:flux-general}--\eqref{eq:ci} decays exponentially. Thus, after some time, it becomes smaller than the parameter $\varepsilon_0$ in Lemma \ref{Lem:Transport-Extinction}. It will therefore vanish after some time according to Lemma \ref{Lem:Transport-Extinction}.
\end{proof}

\subsection{A control theoretic interpretation of Theorem \ref{thm:Transport}}
\label{subsec:control-interpretation}
	Let us mention that Theorem \ref{thm:Transport} is closely related to the following control problem: given $\omega$ a non-empty subinterval of $\T$, $T>0$ and $ u_0 \in L^2(\T)$, find a control function $v \in L^2((0,T) \times \T)$ such that the solution $u$ of 
	\begin{equation}
		\label{Controlled}
		\left\{
			\begin{aligned}
			 \d_t u + \d_x\left( f(u) \right) + {\bf 1}_\omega v =0,\quad & (t,x)\in \R_+ \times \T,
			\\
			u_{\mid t=0}=u_0, \quad & x \in \T.
			\end{aligned}
		\right.
	\end{equation}
	satisfies 
	\begin{equation}
		\label{ControlReq}
		u(T) = 0\quad  \hbox{ in } \T.
	\end{equation}
	Assuming \eqref{eq:transport} and defining $K$ by \eqref{eq:def-K}, Theorem \ref{thm:Transport} implies that, for an initial datum $u_0 \in L^\infty(\T)$ satisfying \eqref{eq:ci-small}, choosing the control function $v$ under the feedback form
	\begin{equation}
		\label{Feedback-control}	
		v(t,x) = - \delta \frac{u(t,x)}{| u(t,x)|^\alpha}, \quad (t,x) \in \R_+ \times \T,
	\end{equation}
	for some $\delta >0$, the controlled trajectory $u$ solving \eqref{Controlled}  will satisfy the controllability  requirement \eqref{ControlReq} in some time $T>0$. 
	\\
	Looking more closely at the proof of Lemma \ref{Lem:Transport-Extinction}, we can state the following result:
	
	\begin{proposition}
		\label{prop-control}
		Let $f$ be a smooth flux function satisfying \eqref{eq:transport} and define $K$ by \eqref{eq:def-K}. Let $\omega$ a non-empty subinterval of $\T$.
		\\
		Given $\gamma >0$, there exists a parameter $\delta$ in \eqref{Feedback-control} such that, for any initial datum $u_0 \in L^\infty(\T)$ satisfying \eqref{eq:ci-small}, the corresponding solution $u$ of \eqref{Controlled}--\eqref{Feedback-control} vanishes after the time 
		\begin{equation*}
			T=(1+ \gamma) \frac{|\T|}{\inf_{[-K,K]} |f'|}. 
		\end{equation*}
	\end{proposition}
	\begin{proof}
		We do the same identifications as in the proof of Lemma \ref{Lem:Transport-Extinction}. Indeed, choosing $\gamma >0$ smaller if necessary, one
                can assume $ \gamma |\T| / \inf |f'| < \tau_*$, with
                $\tau_*$ as in \eqref{Def-T-*}. Thus, taking
	\[
		\delta = \frac{K^\alpha}{\alpha} \frac{\inf |f'|}{\gamma |\T|},   
	\]
	for solutions $u_0 \in L^\infty(\T)$ satisfying
        \eqref{eq:ci-small}, the time $\tau_0$ in \eqref{def-tau-0} is
        smaller than $\tau_*$ in \eqref{Def-T-*} and than $\gamma |\T|/
        |\inf f'|$, so that the proof of Lemma
        \ref{Lem:Transport-Extinction} easily yields that $u$ vanishes
        after the time $T_0 = (1+ \gamma) |\T|/\inf |f'|$. 
	\end{proof}
	Note that the time of extinction given by Proposition \ref{prop-control} can be made arbitrarily close to the critical time expected to control \eqref{Controlled}, given by $|\T| / \inf_{[-K,K]} |f'|$. In this sense, we have produced a non-linear feedback operator which controls \eqref{Controlled} in almost sharp time.


\section{Proof of Theorem \ref{thm:Burgers-Main}: The degenerate case}
\label{sec:burgers}

The goal of this section is to discuss the case of a flux satisfying
$f'(0) = 0$, with $f''(0)\not =0$, and prove
Theorem~\ref{thm:Burgers-Main}. As said in the  
introduction, we first prove Theorem~\ref{thm:Burgers-Main} in
the case of a constant initial datum and a strictly convex flux
satisfying \eqref{Cond-Flux}. We then deduce the other instances
of Theorem \ref{thm:Burgers-Main} by using symmetry arguments and
comparison arguments. 
\subsection{Computation and estimates for the solution $u$ of (\ref{eq:sol-u-K}) for a strictly convex flux $f$ with $f'(0) = 0$}
We first assume that $f$ satisfies, for some $K>0$: 
\begin{equation}
		\label{Cond-Flux-convexe}
		f'(0) = 0 
		\quad \hbox{ and } \quad 
		\exists K >0, \quad \inf_{s \in [-K,K]} f''(s) >0, 
\end{equation}
and the damping is given by $a(x)=\delta \mathbf{1}_{(0,A)}$.

Let $u$ be the solution of
\begin{equation}
  \label{eq:sol-u-K}
  \d_t u+\d_x f(u)+\delta {\mathbf{1}}_{(0,A)} \frac{u}{|u|^\alpha} =0 \text{ in } \R_+ \times \T, \quad u_{\mid t=0}= K \text{ in } \, \T,
\end{equation}
where $K$ is the positive constant in \eqref{Cond-Flux-convexe}, $\delta >0$ and $(0,A) \subset \T$. 

Obviously, as $K > 0$, the
solution $u$ will stay non-negative for all times. 
%
\\
We  develop a precise analysis of the characteristics curves of the
solution, illustrated in Figure \ref{fig:evo_burgers}. We depict in
Figure \ref{fig:evo_burgers} both the evolution of the solution and
its characteristic curves for the Burgers equation. The simulations
are made following the process described in Section~\ref{sec:num}. The
numerical parameters are $\alpha=1$, $u_0(x)=K=1.25$, $A=0.25$,
$\delta =1$, $\delta x = \delta t = 5\cdot 10^{-5}$ and the final time
is $T_f=10$. The dashed line indicates the location of the support of
$a$. It appears that the solution becomes zero inside the support of
$a$, with corresponding characteristics becoming vertical straight
lines.  
\begin{figure}[h]
\begin{tabular}{cc}
\input{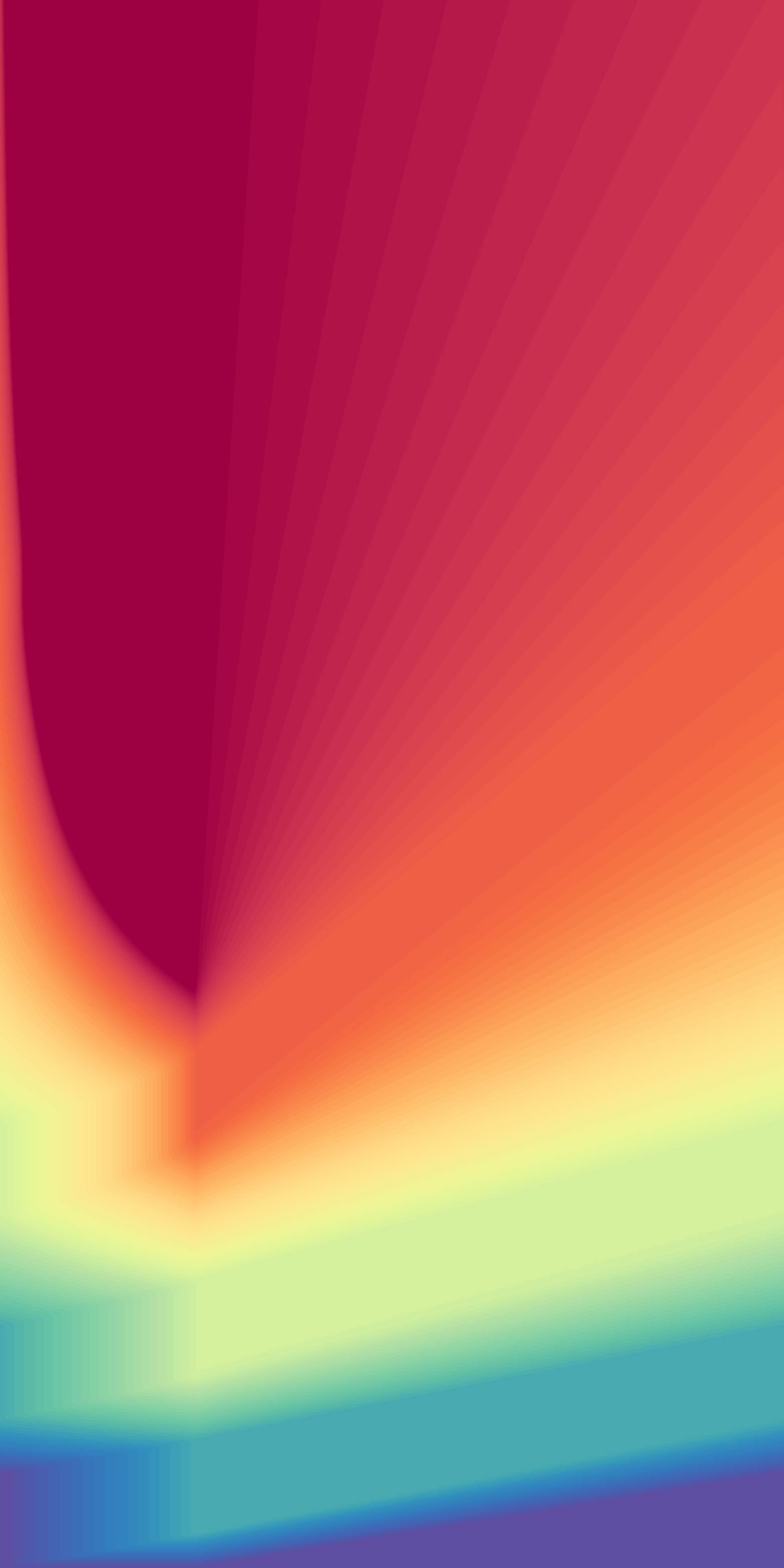}
  &
\input{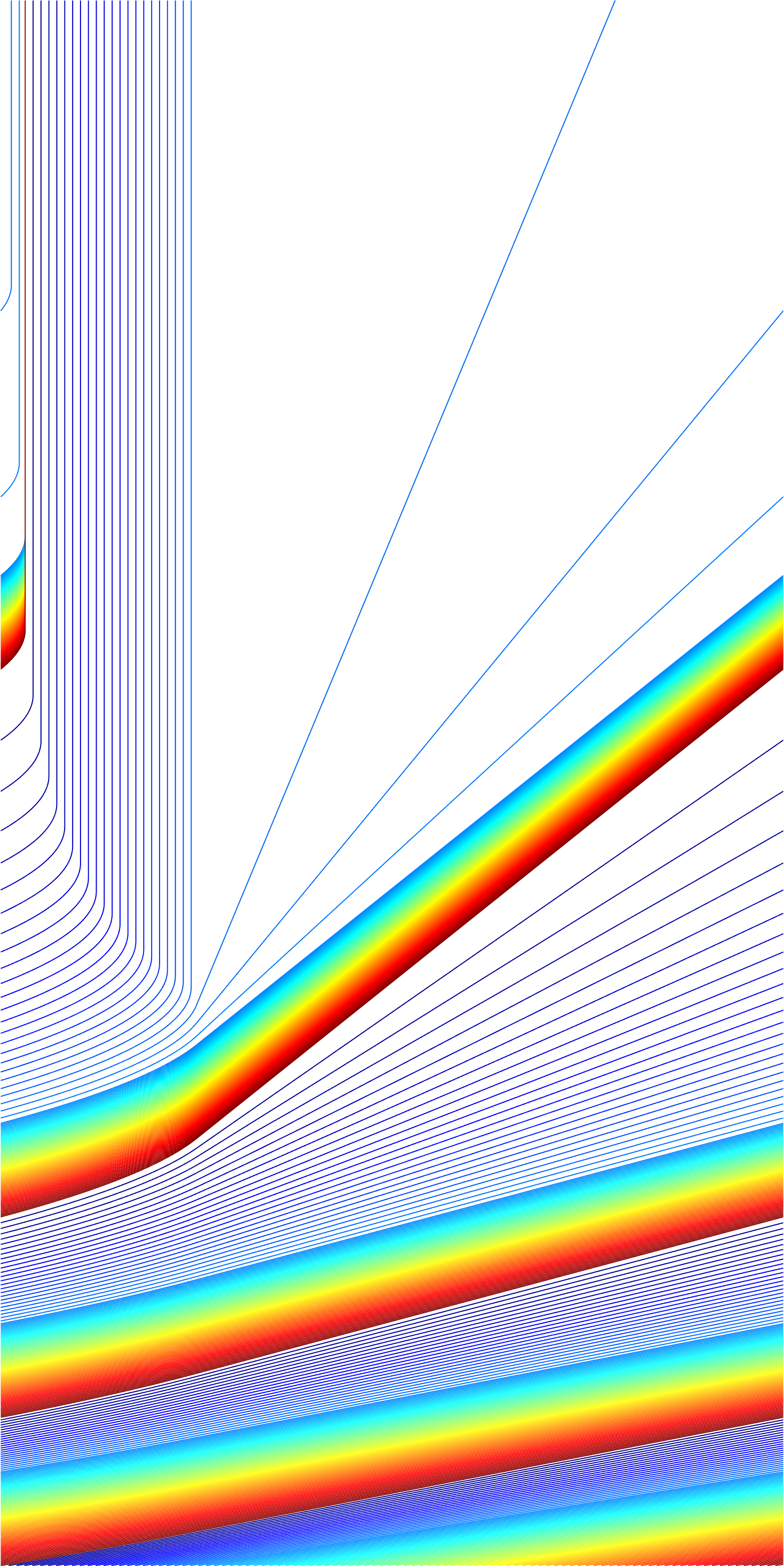}\\
\end{tabular}   
  \caption{Evolution of the solution (left) and of the characteristic curves (right) for Burgers flux}
  \label{fig:evo_burgers}
\end{figure}
Our goal in the following is to prove rigorously that this observed behavior indeed coincides with the theoretical one.

\subsubsection{Formal computation of the solution $u$ of \eqref{eq:sol-u-K} for a strictly convex flux $f$ with $f'(0) = 0$.}\label{subsec:comput}
We shall compute $u$ using characteristics as if the solution
were regular. We will fully justify this assumption later (in Subsection \ref{subsubsec:reg-u}, Lemma \ref{Lem-Reg-u0}).
\\
For $\alpha < 1$, when $u$ is smooth, for $(t_0, x_0) \in \R_+ \times \T$, the characteristics are given by 
\begin{equation}
	\label{eq:characteristics}
	\left\{
		\begin{array}{l}
			\displaystyle \frac{dX}{dt}(t,t_0, x_0) = f'(u(t,X(t,t_0, x_0))), \quad t \ge 0, 
			\smallskip
			\\
			X(t_0,t_0,x_0) = x_0,
		\end{array}
	\right.
\end{equation}
and the solution $u$ along the characteristics satisfies:
\begin{equation}
	\label{eq:Explicit-u}
	\frac{d}{dt} \left(u(t,X(t,t_0,x_0)) \right) = - a(X(t,t_0,x_0)) \frac{u(t,X(t,t_0,x_0))}{|u(t,X(t,t_0,x_0))|^{\alpha}}.
\end{equation}
The solutions of \eqref{eq:characteristics}--\eqref{eq:Explicit-u} are then solutions of 
\begin{equation}
	\label{Characterics-AlphaLe1}
		\frac{d}{dt} 
				\begin{pmatrix} 
				X(t,t_0, x_0) \\ u(t,X(t,t_0,x_0)) 
				\end{pmatrix}
			= F_\alpha(X(t,t_0,x_0), u(t,X(t,t_0,x_0))), 
			\quad t \geq 0, 
\end{equation}
where
\begin{equation}
	\label{Def-F-alpha}
	F_\alpha(X,u) = \begin{pmatrix}f'(u) \\ -a (X)
          u/|u|^\alpha\end{pmatrix}. 
\end{equation}
Of course, when $\alpha = 1$, similar computations can be performed as
long as the solution $u$ stays positive, but the corresponding
definition of $F_\alpha$ for $\alpha = 1$ should be made more specific
when $u$ vanishes. We thus introduce   
\begin{equation}
	\label{Def-F-1}
	F_1(X,u) = \begin{pmatrix}f'(u) \\ -a (X) \Sign (u) \end{pmatrix}, 
\end{equation}
where $\Sign$ is defined in \eqref{Def-Sign-u}, and the corresponding counterpart of \eqref{Characterics-AlphaLe1} should then read as 
\begin{equation}
	\label{Characterics-Alpha=1}
		\frac{d}{dt} 
				\begin{pmatrix}
				X(t,t_0, x_0) \\ u(t,X(t,t_0,x_0)) 
				\end{pmatrix}
			\in F_1(X(t,t_0,x_0), u(t,X(t,t_0,x_0))), 
			\quad t \geq 0.
\end{equation}
\medskip
In the computations given afterward, we will use the fact that as
$u_0(x) = K\geq 0$, the solution $u$ of \eqref{eq:sol-u-K} stays
non-negative, so that we can in fact write $u/|u|^\alpha = u^{1-
  \alpha}$, which will make the various expressions slightly easier. 

\noindent {\it On the interval $(A,1)$.} As $a$ vanishes on $(A,1)$,
\eqref{eq:Explicit-u} implies that the solution $u$ stays constant
along the characteristics in $(A,1)$. Therefore, if we choose $x_0 =
A$, we get  
\[
	u(t,X(t,t_0, A)) = u(t_0, A), \quad \hbox{ as long as } t \mapsto X(t,t_0,A) \hbox{ stays smaller than } 1, 
\]
and therefore
\[
	X(t,t_0, A) = A + (t-t_0) f'(u(t_0, A))\quad \hbox{ for } t \in \left[t_0, t_0 + \frac{1-A}{f'(u(t_0, A))} \right].
\]
Thus, we can write, for all $t \ge 0$, 
\begin{equation}
	\label{eq:interaction-t0-t1}
	u\left(t + \frac{1-A}{f'(u(t,A))}, 0 \right) 
	= 
	u\left(t + \frac{1-A}{f'(u(t,A))}, 1 \right) 
	= 
	u(t, A).
\end{equation}
Let us finally note that easy computations show that
\begin{equation*}
	\forall t \in \left[0, \frac{1-A}{f'(K)} \right], 
	\quad
	u(t,0) = u(t,1) = K.
\end{equation*}
\\
{\it On the interval $(0,A)$.} We deal with this case as before. But
now, as long as the characteristic $t \mapsto X(t,t_0, 0)$ defined for
$t\ge t_0$ stays in $[0,A]$, we have 
\begin{equation*}
	u(t,X(t,t_0,0)) = \left(u(t_0, 0)^\alpha - \delta \alpha (t-t_0) \right)_+^{1/\alpha}.
\end{equation*}
Therefore, the characteristic $X(t,t_0,0)$ reaches $x = A$ for the
first solution $t \ge t_0$ (if any) of  
\begin{equation*}
	\int_0^{t-t_0} f'\left( \left(u(t_0, 0)^\alpha - \delta \alpha \tau \right)_+^{1/\alpha}\right) \, d\tau = A,
\end{equation*}
for which we have 
\begin{equation*}
	u(t, A) = \left(u(t_0, 0)^\alpha - \delta \alpha (t-t_0) \right)_+^{1/\alpha}.
\end{equation*}
For $t \ge 0$, such that
\begin{equation*}
	\int_0^{t} f'\left( \left(K^\alpha - \delta \alpha \tau \right)_+^{1/\alpha}\right) \, d\tau < A,
\end{equation*}
we get 
\begin{equation*}
	u(t,A) = \left(K^\alpha - \delta \alpha t \right)_+^{1/\alpha}.
\end{equation*}

\subsubsection{Justification of the above formulae}\label{Subsec-Justifications}
When the solution $u$ is smooth and strictly positive, all the above computations are fully justified. 


Besides, for $\alpha \in (0,1]$, we only have \emph{a priori}
existence results for solutions of \eqref{Characterics-AlphaLe1} or
\eqref{Characterics-Alpha=1} due to Cauchy--Peano Theorem. Due to
\cite[Chapter 2 Section 10 Theorem 1]{Filippov}, we also have forward
uniqueness of the solutions of \eqref{Characterics-AlphaLe1} as long as $X$
stays in $(0,A)$ or as long as $X$ stays in $(A,1)$, as the function
$F_\alpha$ satisfies the following one-sided Lipschitz condition:
there exists $C >0$ such that for all $(X_1,u_1) \in \R^2$, $(X_2,
u_2) \in \R^2$ with $|u_1|, |u_2| \leq K$, and $(X_1, X_2) \in (0,A)^2
\cup (A,1)^2$,  
\begin{equation}
	\label{1-sided-Lips}
	((X_1,u_1) - (X_2, u_2)) \cdot (F_\alpha(X_1,u_1) - F_\alpha(X_2,u_2) ) \leq C |(X_1,u_1) - (X_2, u_2)|^2.
\end{equation}
The uniqueness across the set $\{X = 0\}$ in our setting will follow
from the fact that the solution $u$ of \eqref{eq:sol-u-K} with
constant initial datum $u_0(x) = K>0$ stays non-negative for all times,
and strictly positive at $x = 0$ for all times (see
Section~\ref{subsubsec:reg-u}), so the characteristics $t \mapsto 
X(t,t_0,x_0)$, when meeting  $\{X = 0\}$, will simply follow the
dynamics in $(0,A)$. This argument can be invoked similarly when
characteristics meet the set $\{X = A\}$ while $u$ stays
positive. However, as we will see, there will be some time at which
$u(t,A)$ vanishes. There, uniqueness should also hold across the set
$\{X = A\}$ simply by continuity of  $F_\alpha(x, 0)$ across $\{X =
A\}$, at least for $\alpha \in (0,1)$. This is however less clear to
prove, especially when turning to the case $\alpha = 1$. 

Thus, to properly justify the above computations, we  construct explicitly the solution $u$ of \eqref{eq:sol-u-K}  using the characteristics formulae above. In turn, this will guarantee the characteristics formulae given above. Note that, strictly speaking, our arguments construct \emph{a} solution of \eqref{eq:sol-u-K}, but by uniqueness of the admissible solution of \eqref{eq:sol-u-K}, see Proposition \ref{prop:cauchy}, this solution is \emph{the} solution of \eqref{eq:sol-u-K} with initial datum $u_0(x) = K$.

\subsubsection{Regularity of the solution $u$.}\label{subsubsec:reg-u}
The goal of this section is to prove the following regularity result on the solution $u$ of \eqref{eq:sol-u-K} with $u_0(x) = K$.
\begin{lemma}
	\label{Lem-Reg-u0}
	Let $f$ and $K$ as in  \eqref{Cond-Flux-convexe},  $\delta >0$ and $(0,A) \subset \T$.
	
%

	Then the solution $u$ of \eqref{eq:sol-u-K} satisfies the following regularity properties: $u\in \mathscr{C}^0([0, \infty) \times \T)$,  $u$ is  piecewise $\mathscr{C}^1([0,\infty) \times \T)$, and we have the more precise result: 
	\begin{itemize}
		\item The set $\mathscr{Z} = \{(t,x) \in [0,\infty) \times \T, \, \text{s.t.} \,u(t,x) = 0\}$, if not empty, is a closed set of $[0,\infty) \times (0,A]$ whose boundary is globally Lipschitz and piecewise $\mathscr{C}^1$. 
		\item For all bounded open subset $\Omega$ such that
                  $\overline\Omega \subset ([0,\infty) \times \T
                  \setminus \mathscr{Z})$, there exists a finite
                  number of smooth ($\mathscr{C}^1$) curves
                  $\mathcal{C}_i$, which may intersect only
                  transversally, such that $u$ is $\mathscr{C}^1$ in
                  the adherence of each of the connected component of
                  $\Omega \setminus \mathcal{C}_i$. 
		\item For all bounded open subset $\Omega \subset
                  ([0,\infty) \times \T)$,  there exists a finite
                  number of   curves
                  $\mathcal{C}_i$ (globally Lipschitz and piecewise
                  $\mathscr{C}^1$), which may intersect only
                  transversally, 
                  such that $u$ is $\mathscr{C}^1$ in each of the
                  connected component of $\Omega \setminus
                  \mathcal{C}_i$. 
                  \item For all $t \geq 0$, $x \mapsto u(t,x)$ is non-decreasing on $[A, 1]$.
	\end{itemize}
\end{lemma}
The proof of Lemma \ref{Lem-Reg-u0} relies on the explicit
construction of $u$ using characteristics formula. As we will see in
the proof, some curve of strong $\mathscr{C}^1$ singularity (meaning
that at each point of the curve, the function $u$ on each side on the
curve cannot be both extended as a $\mathscr{C}^1$ function up to the
curve) 
may appear when the characteristic entering the zone in which the damping is active corresponds to a small value of $u$. In this case, characteristics may become vertical and merge after some time. This does not violate the forward uniqueness of the characteristics in the sense of Filippov. Still, we emphasize that when the characteristics merge, we cannot use them backward in time. This is in fact completely similar to the phenomenon which appears when solving the ODE \eqref{Model-ODE}.

To be more precise, we  introduce $\varepsilon \in (0,K]$ as the
solution, if it exists, of  
\begin{equation}
	\label{def-varepsilon}
		\int_0^\infty f'( (\varepsilon^\alpha - \delta \alpha \tau)_+^{1/\alpha} )\ d \tau = A,
\end{equation}
and $\To$ as 
\begin{equation}
	\label{Def-To}
	 \To = \inf \{ t \in [0,\infty), \, u(t,0) \leq \varepsilon\}.
\end{equation}
The role of $\varepsilon$ will appear clearly in the proof
below. Loosely speaking, when $u(t,0) \leq \varepsilon$, the
characteristic issued from $(t,0)$ will never reach the set $\{x =
A\}$: in other words, the  characteristic issued from $(t,0)$ is of
too low energy to overpass the damping set. 

The curves $\mathcal{C}$ on which $\mathscr{C}^1$ singularities may appear will be constructed with the solution $u$ itself. In fact, these curves will simply be 
\begin{equation}
	\label{SingularCurves}
	\left\{
		\begin{array}{l}
	\displaystyle \mathcal{C}_0: t \mapsto (t,X(t,0,0)),
	\\
	\displaystyle \mathcal{C}_1: t \mapsto (t, X(t,0,A)),
	\\ 
	\displaystyle \mathcal{C}_2: t \mapsto (t,0), 
	\\
	\displaystyle \mathcal{C}_3: t \mapsto (t,A), 
		\end{array}
	\right.
\end{equation}
to which, if $\To < \infty$, one should add the boundary of the set $\mathscr{Z}$ (if $\mathscr{Z} \neq \emptyset$), that will be shown to be delimited by
\begin{multline}
	\label{Eq-mathcalC}
	\mathcal{C} : t_0 \in (\To,\infty) \mapsto (t(t_0),x(t_0)), 
	\\
	\text{ where } t(t_0) = t_0 + \frac{(u(t_0,0))^\alpha}{\delta \alpha}, \text{ and } x(t_0) = \int_0^\infty f'((u(t_0,0)^\alpha - \delta \alpha \tau)_+^{1/\alpha}) \, d\tau,
\end{multline}
and a part of $\mathcal{C}_3$, namely $\{(t,A), \hbox{ for } t \geq t_*\}$ for  
\begin{equation}
	\label{def-t-*}
	t_* = \To + \frac{\varepsilon^\alpha}{\delta \alpha}.
\end{equation}
The various discontinuity curves, $\mathscr{Z}$ region and times $T_*$ and $t_*$ are displayed in Figure \ref{fig:discont_curves}.
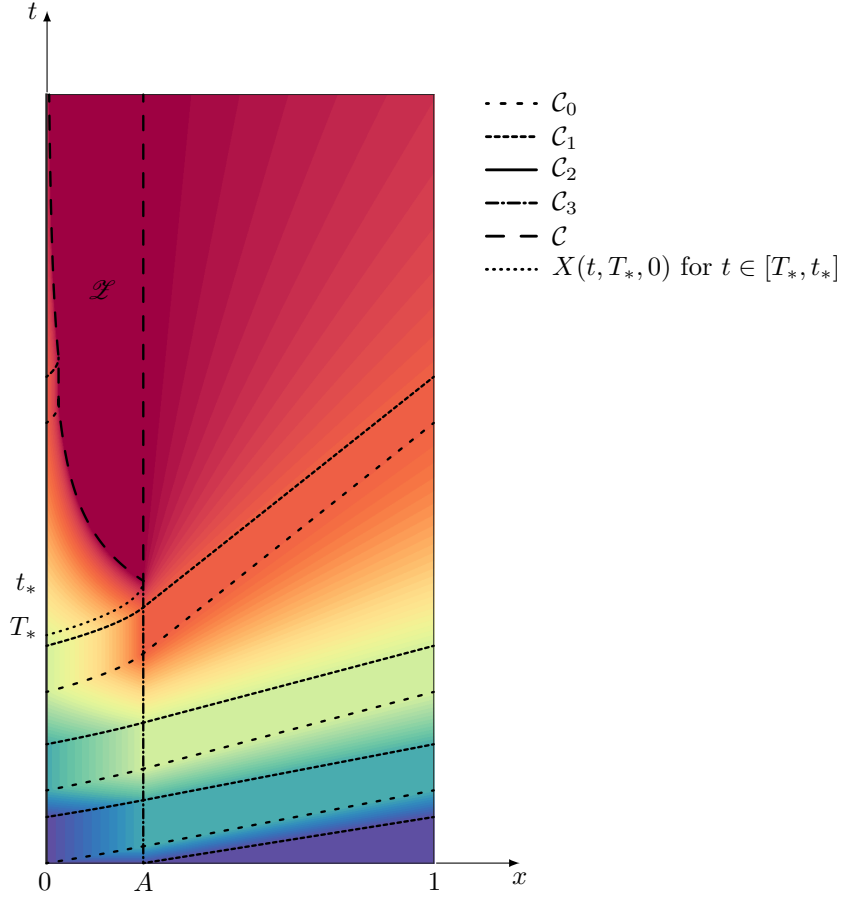
\begin{figure}[htbp]
  \centering
\begin{tikzpicture}
  \node at (0,0){\includegraphics[width=.5\textwidth]{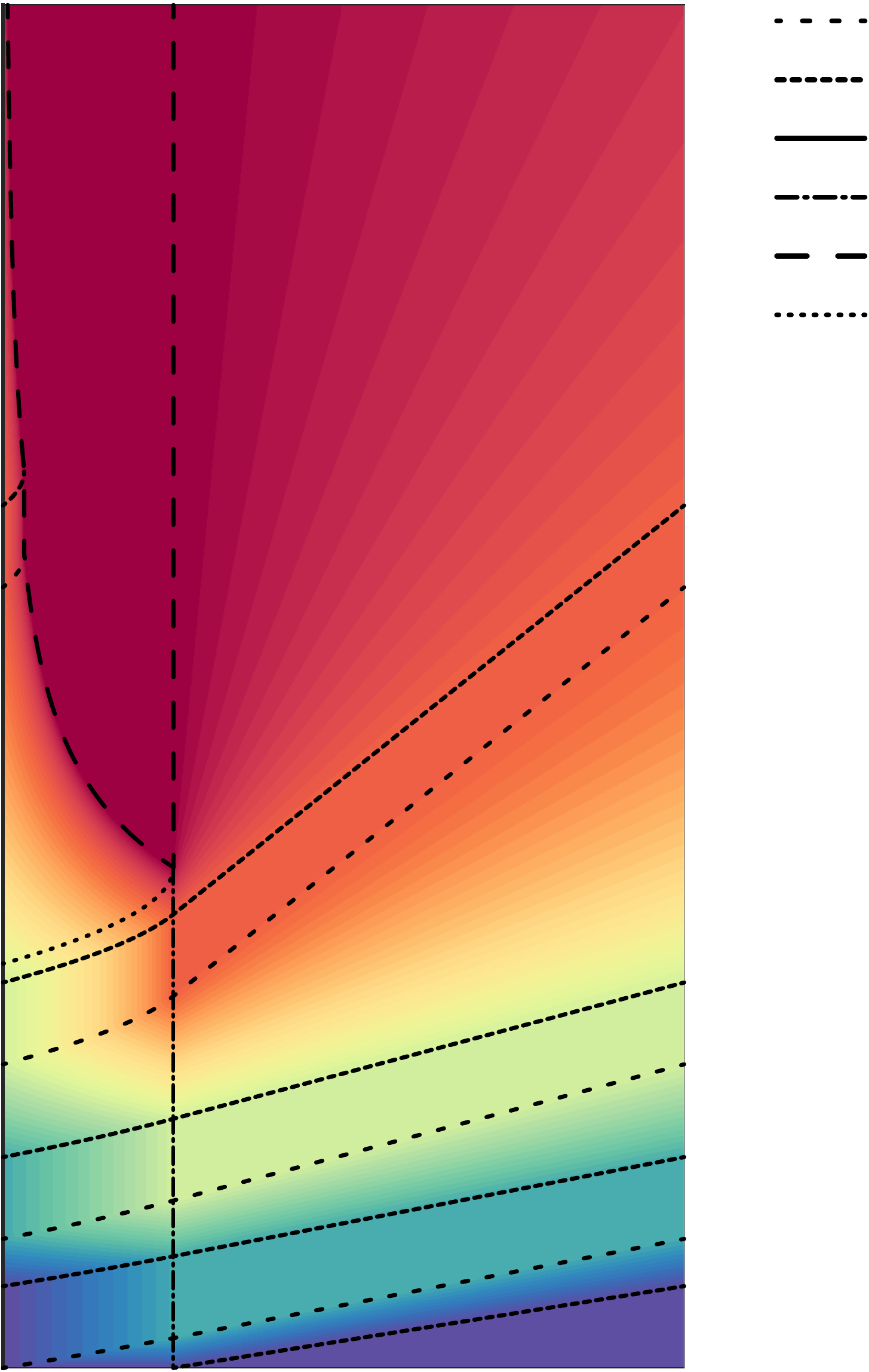}};

      \node at (-2.5,2.5) {$\mathscr{Z}$};
      \node[below] at (-3.25,-5.1) {$0$};
      \node[below] at (-1.95,-5.1) {$A$};
      \node[below] at (1.88,-5.1) {$1$};

      \draw[>=latex,->] (1.9,-5.1) -- (3,-5.1);
      \node[below] at (3,-5.1) {$x$};

      \draw[>=latex,->] (-3.22,5.12) -- (-3.22,6.2);
      \node[left] at (-3.22,6.2) {$t$};
      
      \node[right] at (3.3,4.97) {$\mathcal{C}_0$};
      \node[right] at (3.3,4.52) {$\mathcal{C}_1$};
      \node[right] at (3.3,4.1) {$\mathcal{C}_2$};
      \node[right] at (3.3,3.65) {$\mathcal{C}_3$};
      \node[right] at (3.3,3.21) {$\mathcal{C}$};
      \node[right] at (3.3,2.78) {$X(t,T_*,0)$ for $t\in[T_*,t_*]$};

      \node[left] at (-3.22,-1.98) {$T_*$};
      \node[left] at (-3.22,-1.4) {$t_*$};
\end{tikzpicture}
  
  \caption{Discontinuity curves, $\mathscr{Z}$ region and times $T_*$ and $t_*$.}
  \label{fig:discont_curves}
\end{figure}
\begin{proof} 

\noindent{\bf $\bullet$ Preliminary computations: Existence and uniqueness of $\varepsilon$ in \eqref{def-varepsilon}.}
We first emphasize that condition \eqref{def-varepsilon} is satisfied by at most one parameter $\varepsilon >0$ as the map 
\begin{equation}
	\label{Def-g-v}
	g : v \mapsto \int_0^\infty f'((v^\alpha - \delta \alpha \tau)_+^{1/\alpha}) \, d\tau
\end{equation}
is well-defined and continuous on $[0,K]$, $g(0) = 0$, and $g$ is strictly increasing. Indeed, if $0\leq v_1 < v_2 \leq K$, using the fact that $f'$ is strictly increasing on $[0,K]$ as $f$ is assumed to be strictly convex,
\begin{align*}
	g(v_1) = \int_0^\infty f'((v_1^\alpha - \delta \alpha \tau)_+^{1/\alpha}) \, d\tau
	& =  \int_0^{v_1^\alpha/(\delta \alpha)} f'((v_1^\alpha - \delta \alpha \tau)_+^{1/\alpha}) \, d\tau
	\\
	& < \int_0^{v_1^\alpha/(\delta \alpha)} f'((v_2^\alpha - \delta \alpha \tau)_+^{1/\alpha}) \, d\tau \leq  g(v_2).
\end{align*}
Thus, if $g(K) \geq A$, there exists a unique $\varepsilon \in [0,K]$ satisfying \eqref{def-varepsilon}. (In case $g(K) < A$, there is no $\varepsilon \in [0,K]$ satisfying \eqref{def-varepsilon}.) Note in particular that, when there exists $\varepsilon \in (0,K]$ satisfying \eqref{def-varepsilon}, for all $v \in (\varepsilon,K]$, $g(v) \geq A$, so that 
\begin{equation}
	\label{Bounds-eps-K}
	\forall v \in (\varepsilon, K], \quad A \leq f'(v) \frac{v^\alpha}{\delta \alpha} \leq f'(K) \frac{K^\alpha}{\delta \alpha}.
	\medskip
\end{equation}

\noindent{\bf $\bullet$ Construction of $u$.}
In our construction below, we distinguish the cases $t < \To$ and $t
\geq \To$. In particular, according to the definition \eqref{Def-To}
of $\To$, for all $t \in [0, \To)$, $u(t, 0) > \varepsilon$. 

We restrict ourselves to the case $\To >0$, since the case $\To = 0$ can
be  easily adapted from the case $\To > 0$ and $t \geq \To$. Hence we
assume that $\To>0$, and deal separately with the cases $t \leq \To$
and $t \geq \To$. 
%


We will not point out below, along the construction of $u$, that the
curves  delimiting the $\mathscr{C}^1$ singularities are exactly the
ones in \eqref{SingularCurves}--\eqref{Eq-mathcalC}, but it will
appear clearly from the construction of $u$.

As a matter of fact, we construct a sequence of time-space domains which
eventually cover $[0,\infty) \times \T$, and a function $u$ on each of
these time-space domains, such that $u$ is a globally $\mathscr{C}^0$
function there, and is a piecewise $\mathscr{C}^1$ function, where the
curves of $\mathscr{C}^1$ discontinuities are given by
\eqref{SingularCurves}--\eqref{Eq-mathcalC}.  
Besides, apart from these curves of singularities, the solution $u$ is
constructed to satisfy the characteristics equations
\eqref{Characterics-AlphaLe1} in the case $\alpha \in (0, 1)$, or
\eqref{Characterics-Alpha=1} in the case $\alpha = 1$ away from the
set $\mathscr{Z}$. The regularity of the curves of $\mathscr{C}^1$
discontinuities then allows to check easily that the solution $u$
constructed this way solves \eqref{eq:sol-u-K} in the sense of
Definition~\ref{Def-F-alpha} for $\alpha \in (0,1)$ or of
Definition~\ref{Def-F-1} when $\alpha = 1$.  

In the proof below, with a slight abuse of terminology, we call
``smooth functions'' functions which are $\mathscr{C}^1$. 
\medskip
\\
{\bf Case $t \leq \To$.}
As the velocities involved for the solution $u$ should belong to $[0,f'(K)]$, using the light cone of the equation, $u$ should be fully determined in 
\[
	\mathscr{T}_0 = \{(t,x)\, \hbox{ with } t \in [0,A/f'(K)] \hbox{ and } x \in [t f'(K),A] \}
\]
by $u(0, \cdot)_{\mid [0,A]}$. For $(t,x) \in \mathscr{T}_0$, we shall therefore look for $x_0$ such that $X(t,0,x_0) = x$, that is 
\[
	x_0 + \int_0^t f'\left((K^\alpha - \delta \alpha \tau)_+^{1/\alpha} \right) \, d \tau = x.
\]
For fixed $t \in [0,A/f'(K)] $, it is easily seen that the map 
\[
[0,A]\ni x_0\mapsto 	k_t(x_0) = x_0 + \int_0^t f'\left((K^\alpha - \delta \alpha \tau)_+^{1/\alpha} \right) \, d \tau
\]
is smooth, and strictly increasing, with image containing $[t f'(K),A]$. Therefore, for all $(t,x) \in \mathscr{T}_0$, there exists a unique $x_0 \in [0,A]$ such that 
$X(t,0,x_0) = x$, so that we set $u(t,x) = (K^\alpha - \delta \alpha
t)_+^{1/\alpha}$. As $t \leq A/f'(K)$, the bound \eqref{Bounds-eps-K} allows to guarantee that for all $(t,x) \in \mathscr{T}_0$, $K^\alpha - \delta \alpha t$ is strictly positive, so that $u \in \mathscr{C}^1(\mathscr{T}_0)$
and $t \mapsto u(t,A)$ is smooth on
$[0,A/f'(K)]$, non-increasing, and strictly positive.
\\
We then consider the triangle 
\[
	\mathscr{T}_1 = \{(t,x) \, \hbox{ with } t \in [0,1/f'(K)]
        \hbox{ and } x \in [t f'(K),1] \} ,
\]
and construct $u$ in this set. Of course, as we have
already shown that $u$ is smooth in the set $\mathscr{T}_0$, we
 only focus on the set $\mathscr{T}_1 \setminus
\mathscr{T}_0$. We  determine $u$ in $\mathscr{T}_1 \setminus
\mathscr{T}_0$ from  
$u(\cdot,A)_{\mid [0,A/f'(K)]} $  and $u(0, \cdot)_{\mid
  [A,1]}$. Introducing $X(t,0,A) = A + t f'(u(0,A)) = A + tf'(K)$, for
$(t,x) \in \mathscr{T}_1 \setminus \mathscr{T}_0$,  we set $u(t,x) = K$ if $x \in
[X(t,0,A), 1]$, while if $x \in [ \max \{ t f'(K), A
\}, X(t,0,A)]$, we find a time $t_0 \in [0,A/f'(K)]$ such that
$X(t, t_0, A) = x$, and set $u(t,x) = u(t_0,A)$. Indeed, this can be
achieved since for $t \in 
[0, 1/f'(K)]$ and $x \in [ \max \{ t f'(K), A \}, X(t,0,A)]$, finding
$t_0 \in [0,\min\{t, A/f'(K)\}]$ such that $X(t,t_0, A) = x$ amounts to solving the
equation 
\[
	A + (t-t_0) f'(u(t_0,A)) = x.
\]
But the map 
\begin{equation}
	\label{Def-h-t}
	h_t :  t_0 \mapsto A + (t-t_0) f'(u(t_0,A))
\end{equation}
satisfies: 
\begin{itemize}
	\item $h_t$ is strictly decreasing on the interval $[0, \min\{t, A/f'(K)\} ]$. This is a consequence of the fact that $t_0 \mapsto u(t_0, A)$ is non-increasing and strictly positive on $[0,A/f'(K)]$ and that $f$ is strictly convex with $f'(0) =0$. Besides, for all $t_0 \in [0, \min\{t, A/f'(K)\} ]$,
	\[
		h_t'(t_0) \leq - f'(u(A/f'(K),A)) < 0.
	\]
	\item $h_t(0) = A +t f'(K) = X(t,0,A)$.
	\item If $t \leq A/f'(K)$, $h_t(t) = A$, and if $t \geq A/f'(K)$,
	\begin{align*}
		h_t\left( \frac{A}{f'(K)} \right) &= A+ \left(t - \frac{A}{f'(K)} \right) f'\left( u\Big(\frac{A}{f'(K)},A\Big) \right)
		\\
		&	 \leq A+ \left(t - \frac{A}{f'(K)} \right) f'(K) \leq t f'(K).
	\end{align*}
	\item $h_t$ is smooth.
\end{itemize} 
 It then follows that $h_t$ is a diffeomorphism from the interval $[0,
 \min\{t, A/f'(K)\} ]$ to its image, which contains $[\max\{ t
 f'(K),A\},X(t,0,A)]$, so that we can write, for $t \in [0, 1/f'(K)]$
 and $x$ in the interval $[ \max \{ t f'(K), A \}, X(t,0,A)]$, $u(t,x)
 = u(h_t^{-1}(x), A)$.  
 
 These formula easily show that the function $u$ is smooth for $(t,x)
 \in \mathscr{T}_1$ with $x \in [\max\{ t f'(K),A\},X(t,0,A)]$, and it
 is clear that $u$ is smooth for $(t,x) \in \mathscr{T}_1$ with $x \in
 [X(t,0,A),1]$. However, though it is clear that $u$ is continuous
 along the curve $t\mapsto (t,X(t,0,A))$, $u$ may not be locally
 $\mathscr{C}^1$ in a neighborhood of this curve, even if it can be
 extended as $\mathscr{C}^1$ functions up to this curve from each
 side. It follows that $u$ is piecewise
 $\mathscr{C}^1(\mathscr{T}_1)$, and that $t \mapsto u(t,1)$ is a
 piecewise $\mathscr{C}^1$ non-increasing function on $[0,1/f'(K)]$,
 which remains strictly positive on the interval $[0,1/f'(K)]$. 

 It is also easy to check that for all $t \in [0, A/f'(K)]$, $x\mapsto
 u(t,x)$ is non-decreasing on $[A,1]$. 

In fact, this is the starting point of an iterative argument showing that for all $n \in \N$ with $T_{n} = n/f'(K)  < \To$, one can construct a solution $u$ of \eqref{eq:sol-u-K} in 
\[
	\mathscr{T}_{2n+1} = \{(t,x) \in [0,(n+1)/f'(K)] \times [0,1] \hbox{ with } x \in [\max\{ tf'(K) -n ,0\}, 1]\},
\]
such that:
\begin{itemize}
	\item  $u$ is globally $\mathscr{C}^0$ and piecewise $\mathscr{C}^1$ on $\mathscr{T}_{2n+1}$
	\item $t\mapsto u(t,A)$ is strictly positive and 
          non-increasing on $[0 ,(n+A )/f'(K)]$. 
	\item $t\mapsto u(t,1)$ is strictly positive non-increasing on $[0,(n+1)/f'(K)]$.
	\item For all $t \in [0,T_n+A/f'(K)]$, $x \mapsto u(t,x)$ is non-decreasing on $[A, 1]$.
	
\end{itemize}
Indeed, these properties are already proved for $n =0$. Let us show that if they hold for $n \in \N$, they also hold for $n+1$ provided $T_{n+1} < \To$. 
\\
Using the above properties for $n$, we first consider the equation in the trapezoid 
\[
	\mathscr{T}_{2n+2} = \{(t,x) \in [0,(n+1+A)/f'(K)] \times [0,A] \hbox{ with } x \in [\max\{t f'(K)- (n+1),0\}, A]\},
\]
 and use the boundary condition $u(\cdot,0)_{\mid [0,(n+1)/f'(K)]}$
 and $u(0, \cdot)_{\mid [0,A]}$ to construct $u$ in $\mathscr{T}_{2n+2}$. In order to do this, we set, for $t \geq 0$, 
 \[
 	x_0(t) = \min\left\{ \int_0^t f'((K^\alpha - \delta \alpha \tau)_+^{1/\alpha})\, d\tau, A\right\}.
 \]
  
 Now, let us fix $t \in [0, (n+1+A)/f'(K)]$ and $x \in [0,A]$. If $x > x_0(t)$, we define $u(t,x) = (K^\alpha - \delta \alpha t)_+^{1/\alpha}$. If $x < x_0(t)$, we look for $t_0 \in [0,t]$ such that $X(t,t_0,0) = x$, i.e. such that 
 \[
 	\int_0^{t- t_0} f'\left((u(t_0,0)^\alpha - \delta \alpha \tau)_+^{1/\alpha} \right) \, d\tau = x. 
 \]
 We thus define, for $t _0 \in [0,t]$,
 \[
 	g_t(t_0) =  \int_0^{t- t_0} f'\left((u(t_0,0)^\alpha - \delta \alpha \tau)_+^{1/\alpha} \right) \, d\tau.
 \]
It is not difficult to check that, for $t \leq (n+1+A)/f'(K)$, $g_t$ enjoys the following properties: 
\begin{itemize}
	\item $g_t$ is strictly decreasing on $[0, \min\{ t, T_n\}
          ]$. This is a consequence of the fact that $t \mapsto u(t,
          1)$ ($= u(t,0)$) is non-increasing and strictly larger than
          $\varepsilon$ on $[0,T_n]$. Besides, $g_t$ is piecewise
          $\mathscr{C}^1$ on $ [0, \min\{ t, T_n\} ]$ and for all
          $t_0$ such that $g_t(t_0) \le A$ and for which $g_t$ is
          differentiable, 
	\begin{align*}
		g_t'(t_0) & \leq -f'((u(t_0,0)^\alpha - \delta \alpha (t-t_0))_+^{1/\alpha}) 
		\\
		& \leq - f'((u(\min\{ t, T_n\},0)^\alpha - \delta \alpha  (\min\{ t, T_n\}-t_0))_+^{1/\alpha}).
	\end{align*}
	Note that, if $f'(((u(\min\{ t, T_n\},0)^\alpha - \delta
        \alpha  (\min\{ t, T_n\}-t_0))_+^{1/\alpha}) = 0$ and
        $g_t(t_0) \leq A$, then $g(u(\min\{ t, T_n\},0)) \leq A$,
        which is not compatible with $u(\min\{t,T_n\},0) >
        \varepsilon$. Thus there exists $\gamma >0$ such that for all
        $t_0$ such that $g_t(t_0) \leq A$ and for which $g_t$ is
        differentiable, $g_t'(t_0) \leq - \gamma$. 
	\item $g_t(t) = 0$.
	\item $g_t(0) = x_0(t)$ if $\int_0^t f'((K^\alpha - \delta \alpha \tau)_+^{1/\alpha})\, d\tau \leq A$, and is larger than $A$ otherwise.
\end{itemize} 
 One then easily shows that for $t \leq (n+1+A)/f'(K)$, the map $g_t$
 is a piecewise $\mathscr{C}^1$ function from $[0, \min\{t,T_n\}]$ to
 its image, which contains $[0,x_0(t)]$. Besides, there exists a
 unique $t_0(t,A)$ such that $g_t$ is a piecewise $\mathscr{C}^1$
 diffeomorphism from $[t_0(t,A), \min\{t,T_n\}]$ to $[0,x_0(t)]$. We
 can then define $u$ for $t \leq (n+1+A)/f'(K)$ and $[0,x_0(t)]$ by  
 \[
 	u(t,x) = (u(g_t^{-1}(x),0)^\alpha - \delta \alpha (t - g_t^{-1}(x)))_+^{1/\alpha},
\]
and, for $x \in x_0(t)$, by $u(t,x) = (K^\alpha - \delta \alpha t)_+^{1/\alpha}$.
 
As $g_t$ depends smoothly on the time parameter $t$, this defines $u$ as a globally $\mathscr{C}^0$ and piecewise $\mathscr{C}^1$ function in $\mathscr{T}_{2n+2}$ .

We then check that $t \mapsto u(t,A)$ is strictly positive, piecewise $\mathscr{C}^1$ and non-increasing on $[0, (n+1+A)/f'(K)]$. It is obviously $\mathscr{C}^1$ and decreasing in $\{t, \, x_0(t) < A\}$, as $u(t,A) = (K^\alpha - \delta \alpha t)_+^{1/\alpha}$, the positivity coming from the fact that $K > \varepsilon$ and $t \leq T_{n+1} \leq \To$. For $\{t, x_0(t) = A\}$, which is an interval of the form $[t_A,(n+1+A)/f'(K)]$, $u(t,A)$ is given by the formula 
\[
 	u(t,A) = (u(g_t^{-1}(A),0)^\alpha - \delta \alpha (t - g_t^{-1}(A)))_+^{1/\alpha}. 
\]

Now, one can check that for $t^a, t^b \in [t_A,(n+1+A)/f'(K)]$ such that $t^a < t^b $, defining $t_0^a$ and $t_0^b$ by the formula
\[
	g_{t^a}(t_0^a) = A = g_{t^b} (t_0^b), \quad \text{ i.e. }\quad  t_0^a = g_{t^a}^{-1} (A), \, t_0^b = g_{t^b}^{-1} (A),
\]
the decay of $t_0 \mapsto u(t_0, 0)$ on $[0, T_n]$ implies that
\[
	t_0^a \leq t_0^b, \quad \text{ and } \quad t^a - t_0^a \geq t^b - t_0^b.
\]
Consequently, since $t_0 \mapsto u(t_0, 0)$ is non-increasing on $[0,T_{n+1}]$, we immediately have that $t \mapsto u(t,A)$ is non-increasing on $[0, (n+1+A)/f'(K)]$. It is also obviously piecewise $\mathscr{C}^1$ on $[0, (n+1+A)/f'(K)]$. The fact that $u(t,A)$ is strictly positive comes from the fact that $t\mapsto u(t,0)$ stays strictly larger than $\varepsilon$ for $t \leq T_n$.
We then construct the solution $u$ in the set $\mathscr{T}_{2n+3}
\setminus \mathscr{T}_{2n+2}$, using the information given by
$u(\cdot, A)_{\mid [0, (n+1+A)/f'(K)]}$ and $u(0, \cdot)_{\mid
  [A,1]}$. Again, as when working in $\mathscr{T}_1$, we construct the
solution $u$ using characteristics and setting, for $(t,x) \in
\mathscr{T}_{2n+3}$ with $x \in [A, \min\{A+ t f'(K),1\} ]$,  
\[
	u(t,x) = u(h_t^{-1}(x), A).
\]
The other case, corresponding to $x \in [\min\{A+ t f'(K),1\} ,1]$, lies in fact in $\mathscr{T}_1$, so regularity issues have been dealt with before. We only need to check that for all $t \in [0, T_{n+1}+A/f'(K)]$, $x \mapsto u(t,x)$ is non-decreasing on $[A,1]$: this is obvious if $x \in [\min\{A+ t f'(K),1\}, 1]$ as $u(t,x) = K$ there; when $x \in [A, \min\{A+ t f'(K),1\} ]$, the above formula and the facts that $t \mapsto u(t,A)$ is non-increasing and that $h_t^{-1}$ is strictly decreasing imply that $x \mapsto u(t,x)$ is non-increasing on $[A, \min\{A+ t f'(K),1\} ]$.
Therefore, all the items in the above property also hold for $n +1$. 

We can thus iterate these arguments while $T_n = n / f'(K) < \To$. If
$\To = \infty$, this concludes Lemma \ref{Lem-Reg-u0}. If $\To <
\infty$, we perform a similar iteration in the trapeze  
\[
	\mathscr{T} = \{(t,x) \in [0,\To+1/f'(K)] \times [0,1] \hbox{ with } x \in [\max\{ (t-\To) f'(K) ,0\}, 1]\}, 
\]
constructing a function $u$ in $\mathscr{T}$ such that:
\begin{itemize}
	\item $u$ is a solution of \eqref{eq:sol-u-K}, is piecewise $\mathscr{C}^1$ on $\mathscr{T}$ and globally $\mathscr{C}^0$ on $\mathscr{T}$.
	\item $t\mapsto u(t,A)$ is strictly positive and 
          non-increasing on $[0 ,\To+A )/f'(K)]$. 
	\item $t\mapsto u(t,0)$ is strictly positive non-increasing on $[0,\To+1/f'(K)]$.
	\item $u(\To, 0) = \varepsilon$.
	\item For all $t \in [0,T_*+A/f'(K)]$, $x \mapsto u(t,x)$ is non-decreasing on $[A, 1]$.

\medskip
\end{itemize}
\noindent{\bf Case $t \geq \To$.} The difficulty when $t \geq \To$ is to show that the solution $u$ remains continuous in time-space. In order to do this, we introduce $(x_*(t),u_*(t))$ given by 
\begin{equation}
	\label{Def-x-*-u-*}
	\left\{
		\begin{array}{l}
			\displaystyle \frac{d}{dt} 
				\begin{pmatrix} x_* \\ u_* \end{pmatrix}
				= 
				F_\alpha
				\begin{pmatrix} x_* \\ u_* \end{pmatrix}, 
				\quad t \geq \To, 
				\text{ if } \alpha \in (0,1), 
		\\
				\text{ or }
		\\
				\displaystyle \frac{d}{dt} 
				\begin{pmatrix} x_* \\ u_* \end{pmatrix}
				\in F_1
				\begin{pmatrix} x_* \\ u_* \end{pmatrix}, 
				\quad t \geq \To, 
				\text{ if } \alpha = 1, 
		\end{array}
	\right.
	\qquad  
	\begin{pmatrix} x_*(\To) \\ u_*(\To) 
	\end{pmatrix}= 
	\begin{pmatrix} 0 \\ \varepsilon 
	\end{pmatrix}, 
\end{equation}
where $F_\alpha$ is defined in
\eqref{Def-F-alpha}--\eqref{Def-F-1}. Here,  $x_*$ corresponds to
the characteristics $X(t,\To,0)$  $t \geq \To$ and $x_*(t) \leq
A$. Note that, due to the choice of $\varepsilon$ in
\eqref{def-varepsilon}, there exists a time $t_*$ ($=  \To +
\varepsilon^\alpha/(\delta \alpha)$), such that the solution
$(x_*,u_*)$ satisfies 
\begin{equation}
	\label{prop-t-*}
\quad 
	\forall t \in [\To, t_*), \quad x_*(t) \in [0,A),
	\quad \hbox{ and } 	\quad\forall t \geq t_*, \quad (x_*(t) , u_*(t)) = (A, 0), 
\end{equation}
thus guaranteeing the uniqueness of the solution of \eqref{Def-x-*-u-*}.

We then introduce the following sets:
\begin{align*}
	& \mathscr{R}_0 = \{ (t,x) \in [0,\infty) \times [0,A], \text{ such that, if } t \geq \To, \, x \in [x_*(t),A] \}, 
	\\
	& \mathscr{R}_1 = \{ (t,x) \in [0, \infty) \times [A,1] \}, 
	\\
	& \mathscr{R}_2 = \{ (t,x), \text{ with } t \geq \To, \, x \in [0,x_*(t)] \},  
\end{align*}
and, similarly as before, for $n \geq 0$, 
\begin{align*}
	& \mathscr{T}^*_{2n} = \{(t,x) \in [0,\To+(n+A)/f'(K)] \times [0,1] \hbox{ with } x \in [ ((t-\To) f'(K) -n)_+, A]\}, 
	\\
	& \mathscr{T}^*_{2n+1} = \{(t,x) \in [0,\To+(n+1)/f'(K)] \times [0,1] \hbox{ with } x \in [ ((t-\To)f'(K) -n)_+, 1]\}.
\end{align*}
As before, we construct iteratively a solution $u$ of \eqref{eq:sol-u-K} in $\mathscr{T}_{2n+1}^*\setminus \mathscr{R}_2$  for all $n \in \N$ such that:
\begin{itemize}
	\item $u$ is piecewise $\mathscr{C}^1$ and globally $\mathscr{C}^0$ in $\mathscr{T}_{2n+1}^* \setminus \mathscr{R}_2$.
	\item $t\mapsto u(t,1)$ is strictly positive non-increasing on $[0,\To+(n+1)/f'(K)]$ and piecewise $\mathscr{C}^1$.
	\item for all $t \geq \To$ $u(t,x_*(t)) = u_*(t)$, where $(x_*,u_*)$ is the solution of \eqref{Def-x-*-u-*}.
	\item For all $t \in [0,\To+(n+A)/f'(K)]$, $x \mapsto u(t,x)$ is non-decreasing on $[A, 1]$.
\end{itemize}
Of course, the previous paragraph shows that this is true for $n = 0$. Let us then assume that these properties are true for some $n \in \N$ and show that they are then true for $n+1$. 

Similarly as before, we work first on $\mathscr{T}^*_{2n+2} \cap \mathscr{R}_0$. The construction of the function $u$ in the set $\mathscr{T}_{2n+2}^* \cap \mathscr{R}_0$ can then be handled as for $\mathscr{T}_{2n+2}$, and following the same arguments, we easily get that the function $u$ there is piecewise $\mathscr{C}^1$ in $\mathscr{T}_{2n+2}^* \cap \mathscr{R}_0$, and that $t \mapsto u(t,A)$ is a non-increasing function on $[0,\To+ (n+1+A)/f'(K)]$, strictly positive while $t \leq t_*$, and vanishing for $t \geq t_*$.

The construction of the solution $u$ in the set $\mathscr{T}_{2n+3}^* \cap \mathscr{R}_1$ can then be done similarly as the one corresponding to $\mathscr{T}_{2n+3}$, and following the same lines, we get that the solution $u$ is piecewise $\mathscr{C}^1$ in $\mathscr{T}_{2n+3}^* \cap \mathscr{R}_1$. We nevertheless make the proof slightly more precise: for $(t,x) \in \mathscr{T}_{2n+3}^* \cap \mathscr{R}_1 \setminus \mathscr{T}_{1}$, $u(t,x)$ is constructed by 
\[
	u(t,x) = u(h_t^{-1} (x),A), 
\]
where $h_t$ is given by \eqref{Def-h-t}. 
In particular, to establish the strict positivity of $u(t,1)$ for $t \leq \To+(n+2)/f'(K)$, we just remark that $t_0 = h_t^{-1}(x)$ is equivalent to $A+ (t- t_0) f'(u(t_0,A)) = 1$, so that $u(t,x) = u(t_0,A)$ cannot be zero.

The continuity of the function $u$ constructed above in $\mathscr{T}_{2n+3}^* \setminus \mathscr{R}_2$ follows easily from the continuity of $u$ across the interfaces of $\mathscr{T}_{2n+2}^* \cap \mathscr{R}_0$, and $\mathscr{T}_{2n+3}^* \cap \mathscr{R}_1$.

The fact that for all $t \in [0, T_* + (n+1+A)/f'(K)]$, $x \mapsto u(t,x)$ is non-decreasing in $[A,1]$ follows as before.
\medskip

Our goal now is to construct $u$ in $\mathscr{R}_2$ as a piecewise $\mathscr{C}^1$ function such that $u$ is globally continuous on $[0, \infty) \times \T$ and solves \eqref{eq:sol-u-K} in $[0,\infty) \times \T$.

As $t \mapsto u(t,0)$ is strictly positive and non-increasing on $[0,\infty)$, we can show, similarly as before that, for $(t,x) \in \mathscr{R}_2$, the map
\[
 	g_t(t_0) =  \int_0^{t- t_0} f'\left((u(t_0,0)^\alpha - \delta \alpha \tau)_+^{1/\alpha} \right) \, d\tau
 \]
 has the following properties:
\begin{itemize}
	\item $g_t$ is decreasing on $[\To, t]$, takes values in $[0,x_*(t)]$, and is surjective on $[0, x_*(t)]$,
	\item For all $(t_0^a, t_0^b) \in [\To, t]^2$, $g_t(t_0^a) =g_t(t_0^b)$ implies $t_0^a = t_0^b$ or $(u(t_0^a,0)^\alpha - \delta \alpha (t - t_0^a))_+ = (u(t_0^b,0)^\alpha - \delta \alpha (t - t_0^b))_+  $. 
\end{itemize}
It thus allows to set, for $(t,x) \in \mathscr{R}_2$, 
\begin{equation}
	\label{u-T-*-R-2n+2}
	u(t,x) = (u(g_{t}^{-1}(x), 0)^\alpha - \delta \alpha (t- g_{t}^{-1}(x)))_+^{1/\alpha},
\end{equation}
where $g_t^{-1} (x)$ denotes any $t_0 \in [\To, t]$ such that $g_t(t_0
) = x$. Note that, according to the second item above, the definition
above does not depend on the choice of $t_0 $ such that $g_t(t_0) =
x$. Besides, $u$ defined this way is continuous on $\mathscr{R}_2$ as
one can easily check, and if $t \geq \To$, $u(t,x_*(t)) = u_*(t)$,
where $(x_*, u_*)$ is the solution of \eqref{Def-x-*-u-*}.

We then remark that, for $t \geq \To$, and $\To \leq t_0^a < t_0^b \leq t$, 
\begin{align*}
	g_t(t_0^a) & =  \int_0^{t- t_0^a} f'\left((u(t_0^a,0)^\alpha - \delta \alpha \tau)_+^{1/\alpha} \right) \, d\tau
	\\
	& \geq  \int_0^{t- t_0^b} f'\left((u(t_0^a,0)^\alpha - \delta \alpha \tau)_+^{1/\alpha} \right) \, d\tau
	\\
	& \geq  \int_0^{t- t_0^b} f'\left((u(t_0^b,0)^\alpha - \delta \alpha \tau)_+^{1/\alpha} \right) \, d\tau = g_t(t_0^b). 
\end{align*}
In particular, analyzing the case of equality in the above estimates, we easily get that for $(t,x) \in \mathscr{R}_2$ such that $u(t,x) \neq 0$, $t_0(t,x) = g_t^{-1}(x) $ is uniquely defined and 
\begin{equation}
	\label{Est-t-0-t-x}
	t - t_0(t,x) \leq \frac{u(t_0(t,x), 0)^\alpha}{\delta \alpha} \leq \frac{\varepsilon^\alpha}{\delta \alpha} .
\end{equation}
Besides, $g_t$ is piecewise $\mathscr{C}^1$ locally around $t_0(t,x)$ and
\begin{multline*}
	g_t'(t_0) = - f'\left((u(t_0,0)^\alpha - \delta \alpha (t- t_0)^{1/\alpha} \right)
	\\
	+ u(t_0,0)^{\alpha - 1} \partial_t u(t_0,0) \int_0^{t- t_0} f'\left((u(t_0,0)^\alpha - \delta \alpha \tau)^{1/\alpha} \right) (u(t_0,0)^\alpha - \delta \alpha \tau)^{1/\alpha - 1}  \, d\tau.
\end{multline*}
As $u(t,x) \neq 0$ and $t_0 \mapsto u(t_0,0)$ decays, this implies that 
\[
	g_t'(t_0(t,x)) \leq - f' \left((u(t_0,0)^\alpha - \delta \alpha (t- t_0(t,x)))^{1/\alpha} \right) = - f'(u(t,x)) <  0, 
\]
and the bound is uniform in a neighborhood of $t$ and $x$. As $t(t_0,x)$ is defined by $g_t(t_0(t,x)) = x$, we see that this implies that $t_0$ is piecewise $\mathscr{C}^1$ for $(t,x) \in \mathscr{R}_2$ such that $u(t,x) \neq 0$, so that the definition \eqref{u-T-*-R-2n+2} shows that $u$ is piecewise $\mathscr{C}^1$ for $(t,x) \in \mathscr{R}_2$ such that $u(t,x) \neq 0$.

It is thus also interesting to determine the area $\mathscr{Z} = \{ (t,x) , \text{ s.t. } u(t,x) = 0\} \cap \mathscr{R}_2$. 

For $(t,x) \in \mathscr{R}_2$, it is clear, from the relation
$g_t(t_0(t,x)) = x$ and from the fact that $t_0 \mapsto u(t_0, 0)$
decays, that for $t^a < t^b$ such that $(t^a,x)$ and $(t^b,x)$ are in
$\mathscr{R}_2$,  
\[
	t_0(t^a,x) \leq t_0(t^b,x), \quad \text{ and } \quad 
	t^a - t_0(t^a,x) \leq t^b - t_0(t^b,x),
\]
so that we easily derive from the decay of $t_0 \mapsto u(t_0,0)$ that 
\[
	u(t^a,x) \leq u(t^b,x).
\]
It follows that, for all $x \in [0,A]$, the map $t \mapsto u(t,x)$ decays while $(t,x) \in \mathscr{R}_2$. 

We can in particular define, for $x \in (0,A]$, the time $t_*(x) =
\inf\{ t, \, s.t. \, u(t,x) = 0\}$ (which may be
infinite). If $t_*(x)$ is finite, for $t < t_*(x)$, $u(t,x) \neq 0$,
so $t_0(t,x)$ is well-defined and unique, and is an increasing
function of time which is bounded by $t_*(x)$. Thus, the limit of
$t_0(t,x)$ as $t \to t_*^-$ exists, and we call it $t_{0,*}(x)$. We
then easily get that  
\[
	u(t_{0,*}(x), 0)^\alpha - \delta \alpha (t_*(x) - t_{0,*}(x)) = 0, 
\]
so that 
\[
	x = g_{t_*(x)} (t_{0,*}(x)) = g\left(u(t_{0,*}(x), 0)\right), 
\]
where $g$ is defined in \eqref{Def-g-v}.
Let us now note that there exists only one $t_0$ such that 
\begin{equation}
	\label{t-0-vs-x}
	u(t_0,0)^\alpha - \delta \alpha (t_*(x) - t_0) = 0 \quad \text{ and } \quad x = g(u(t_0,0)).
\end{equation}
Indeed, the second equation determines uniquely $u(t_0,0)$, and the first equation then determines uniquely $t_0$. Thus, if $t_0$ satisfies \eqref{t-0-vs-x}, $t_0 = t_{0,*}(x)$ and $t_*(x) = t_0 + u(t_0,0)^\alpha/(\delta \alpha)$. 

This suggests to study the parametric equation
\begin{multline}
	\label{Eq-mathcalCBis}
	\mathcal{C} : t_0 \in (\To,\infty) \mapsto (t(t_0),x(t_0)), 
	\\
	\text{ where } t(t_0) = t_0 + \frac{(u(t_0,0))^\alpha}{\delta \alpha}, \text{ and } x(t_0) = g(u(t_0,0)).
\end{multline}
(This definition of course coincides with the definition of $\mathcal{C}$ in \eqref{Eq-mathcalC}, recall the definition of $g$ in \eqref{Def-g-v}.)
It is clear that by construction $u(t(t_0),x(t_0)) = 0$ for all $t_0 \geq \To$. Besides, for $(t,x) \in \mathscr{R}_2$, if there exists $T_0 \geq \To$ such that $x = g(u(t_0,0))$, then 
\begin{itemize}
	\item if $t \geq \inf \{t(t_0), \, \text{ for } \, t_0\, \,  \text{ s.t. } \, x = g(u(t_0,0)) \}$, then $u(t,x) = 0$; 
	\item if $t \geq \inf \{t(t_0), \, \text{ for } \, t_0\, \,  \text{ s.t. } \, x = g(u(t_0,0)) \}$, then $u(t,x) > 0$. 
\end{itemize}
These statements follow immediately from the decay of $t \mapsto u(t,x)$ and the non-negativity of $u$. It follows that 
\begin{equation}
	\label{Def-Zero-Set}
	\partial \mathscr{Z} = \mathcal{C} \cup \{(t,A), \text{ s.t. } t \geq t_* \}.
\end{equation}
We remark that, as $t_0 \mapsto u(t_0,0)$ is piecewise $\mathscr{C}^1$ and strictly positive, except at singularities (which are in finite number in any bounded interval), we have 
\[
	\frac{d}{dt_0} \left(\! t_0 + \frac{(u(t_0,0))^\alpha}{\delta \alpha}, g(u(t_0,0)) \! \right)
	= 
	\left(\! 1 + \frac{u(t_0,0)^{\alpha -1} \partial_t u(t_0,0)}{\delta} , 
	g'(u(t_0,0)) \partial_t u(t_0,0) \!\right), 
\]
so that 
\[
	\left| \frac{d}{dt_0} \left(t_0 + \frac{(u(t_0,0))^\alpha}{\delta \alpha}, g(u(t_0,0)) \right)
	\right| \neq 0.
\]
This proves that the tangent, hence the normal, of the curve $\mathcal{C}$  is well-defined except at a locally finite number of points. 
This indicates that $\mathcal{C}$ is a piecewise $\mathscr{C}^1$ parametric curve with finite limits at singularity points. It is thus a globally Lipschitz and piecewise $\mathscr{C}^1$ parametric curve.
\medskip

We have thus proved the regularity properties stated in Lemma \ref{Lem-Reg-u0}. In particular, in all connected component of $([0,\infty) \times \T) \setminus (\cup_{i=1}^4 \mathcal{C}_i \cup \mathcal{C})$, the function $u$ is $\mathscr{C}^1$ and satisfies, by construction 
\[
	\partial_t u + \partial_x (f(u)) + h(t,x) = 0,  
\]
where $h(t,x) = \delta {\mathbf 1}_{(0,A)} u(t,x)/|u(t,x)|^\alpha$ if
$u(t,x) >0$ and $ h(t,x) = 0$ if $u(t,x) = 0$. Since $u$ is also
continuous in the whole set $[0,\infty) \times \T$, a straightforward
application of the integration by parts formula shows that the
function $u$ we constructed above is an admissible solution of
\eqref{eq:sol-u-K} in $[0,\infty) \times \T$. 
\end{proof}

\subsubsection{Dynamics of the solution $u$ in the case of a strictly convex flux.} 

\begin{lemma}
	\label{Lem-To-Finite}
		Let $f$ and $K$ as in  \eqref{Cond-Flux-convexe}, $\delta >0$ and $\omega= (0,A) \subset \T$. Then the solution $u$ of \eqref{eq:sol-u-K} satisfies:
	\begin{itemize}
		\item[$(i)$] $\To$ defined in \eqref{Def-To} is finite.
		\item[$(ii)$] For all $t \ge t_*$ (defined in \eqref{def-t-*}), there exists an open subinterval $\omega(t) \subset \omega$ such that $u(t)_{\mid \omega(t)} = 0$, and 
	\begin{equation*}
		|\omega \setminus \omega(t) | \le \frac{C}{t^{1+\alpha}}, 
		\quad 
		\| u(t) \|_{L^\infty(\T)} \le \frac{C}{t}.
	\end{equation*} 
	\end{itemize}
\end{lemma}
	
\begin{proof}[Proof of item $(i)$ of Lemma~\ref{Lem-To-Finite}] The
proof of the first item of Lemma \ref{Lem-To-Finite} follows from the
analysis of the solution $u$ of \eqref{eq:sol-u-K} along the
characteristics $t\mapsto X(t, 0, 0)$, in particular when it crosses
the set $\{x = 0\}$. 

We thus introduce four sequences $(u_n)_{n \in \N}$, $(t_n)_{n \in \N}$, $(v_n)_{n \in \N}$, $(\tau_n)_{n \in \N}$, initialized by $u_0 = K$ and $t_0= 0$, and defined iteratively for $n$ such that $u_n > \varepsilon $ and $v_n > 0$ as follows: 
\begin{itemize}
	\item $\tau_{n}> t_n$ is the unique solution of 
	\[
		 \int_0^{\tau_{n} - t_n} f'( (u_n^\alpha - \delta \alpha \tau)_+^{1/\alpha} ) \ d \tau =A.
	\]
	\item $v_{n} = (u_n^\alpha - \delta \alpha (\tau_{n} -t_n))_+^{1/\alpha}$.
	\item $t_{n+1} = \tau_{n} + (1-A)/f'(v_{n})$.
	
	\item $u_{n+1} = v_{n}$.
\end{itemize}
These choices are made so that for all $n \in \N$ such that $u_n > \varepsilon$, $u(t_n,0) = u_n$, and $u(\tau_{n},A) = v_{n}$. 
%
\\
We also set $n_0$ the first integer (if any) for which $u_n <
\varepsilon$ or $v_n = 0$. This index $n_0$, if any, is in fact such
that $u_{n_0} \le \varepsilon$ and $v_{n_0} = 0$. Indeed, if $u_{n} >
\varepsilon$, one easily checks from the definition of $\varepsilon$
that $v_{n} >0$.  
\\
Note that we easily deduce from the formula of $v_n$ that 
\begin{equation*}
	\forall n \in \{0, \cdots, n_0 -1\}, \quad \delta \alpha (\tau_n - t_n) \le u_n^\alpha.
\end{equation*}
In order to study these sequences, it will be convenient to have a
good estimate on $\tau_{n} - t_n$ in terms of $u_n$ only.  We thus define
\begin{equation*}
	\beta_- = \inf_{[0,K]} \frac{f'(s)}{s}, 
	\quad 
	\beta_+ = \sup_{[0,K]} \frac{f'(s)}{s}, 
\end{equation*}
which are finite as $f'(0) = 0$ and which are both strictly positive
as $f$ is strictly convex. 
\\
%
We then recall that $1/\alpha \ge 1$. Thanks to the convexity of the
function $s \mapsto s^{1/\alpha}$ at the point $u^\alpha$, its graph
on $[0, u^\alpha]$ is below the chord initiated from the origin,
i.e. $s\mapsto s u^{1- \alpha}$, and above its tangent $s \mapsto (1-
1/\alpha) u + s u^{1-\alpha}/\alpha$, which yields the following
estimates: for all $u>0$, $\tau \ge 0$, 
\begin{equation*}
	(u - \delta \tau u^{1- \alpha})_+
	\le
	 (u^\alpha - \delta \alpha \tau)_+^{1/\alpha} 
	 \le
	 (u - \delta \alpha \tau u^{1- \alpha})_+. 
\end{equation*}
Combining the above two  estimates, we infer
\[
	 \beta_-  \int_0^{\tau_{n} - t_n}  (u_n - \delta \tau u_n^{1- \alpha})_+ \, d\tau	 \le 
	 A 
	 \le \beta_+ \int_0^{\tau_{n} - t_n}  (u_n - \delta \alpha \tau u_n^{1- \alpha})_+ \, d\tau, 
\]
which implies in particular that 
\[
	A \ge 
	\left\{
		\begin{array}{ll}
			\displaystyle 
			\frac{\beta_-}{2  \delta} u_n^{1+\alpha} & \hbox{ if } u_n^\alpha \le \delta (\tau_{n} - t_n), 
			\smallskip\\
			\displaystyle 
			\beta_- u_n (\tau_{n} - t_n) - \frac{\delta \beta_-}{2} u_n^{1- \alpha} (\tau_{n} - t_n)^2\, & \hbox{ if } u_n^\alpha \ge \delta (\tau_{n} - t_n),
		\end{array}
	\right. 
\]
and
\[
	A \le 
	\left\{
		\begin{array}{ll}
			\displaystyle 
			\frac{\beta_+}{2 \alpha \delta} u_n^{1+\alpha} & \hbox{ if } u_n^\alpha \le \alpha \delta (\tau_{n} - t_n), 
			\smallskip\\
			\displaystyle 
			\beta_+ u_n (\tau_{n} - t_n) - \frac{\alpha \delta \beta_+}{2} u_n^{1- \alpha} (\tau_{n} - t_n)^2\, & \hbox{ if } u_n^\alpha \ge \alpha \delta (\tau_{n} - t_n).
		\end{array}
	\right. 
\]
We claim that we can deduce from this a lower bound on $\tau_n - t_n$ of the form 
\begin{equation*}
	\forall n \in \{0, \cdots, n_0 - 1\}, 
	\, 
	\alpha \delta (\tau_n - t_n) \ge u_n^\alpha \min\left\{\alpha,  1 - \sqrt{\left(1 - \frac{2 A \alpha \delta}{ \beta_+ K^{1+\alpha}}\right)_+}\right\}.
\end{equation*}
Indeed, this is obvious when $\tau_n - t_n \ge
u_n^\alpha/\delta$. Otherwise, when $\tau_n - t_n \le
u_n^\alpha/\delta$, we have $u_n^\alpha \ge \alpha \delta (\tau_{n} -
t_n)$ and thus one should have  
\[
	\frac{A \alpha \delta}{2 \beta_+ u_n^{1+\alpha}} \le
        \frac{\alpha \delta}{2} \frac{\tau_n - t_n}{u_n^\alpha} -
        \left(  \frac{\alpha \delta}{2} \frac{\tau_{n} -
            t_n}{u_n^\alpha}\right)^2. 
\]
But $u_n \le K$ as the sequence $u_n$ is non-increasing, so that 
\[
	\frac{A \alpha \delta }{2 \beta_+ K^{1+\alpha}} \le \left( \frac{\alpha \delta}{2} \frac{\tau_n - t_n}{u_n^\alpha}\right) -  \left(  \frac{\alpha \delta}{2} \frac{\tau_{n} - t_n}{u_n^\alpha}\right)^2.
\]
Consequently, if $A \alpha \delta/(2 \beta_+ K^{1+\alpha}) > 1/4$, this cannot happen. Besides, if $$A \alpha \delta/(2 \beta_+ K^{1+\alpha}) \le 1/4,$$ we obtain immediately that
\begin{equation*}
	\frac{\tau_n - t_n}{u_n^\alpha} \ge \frac{1}{\alpha \delta} \left( 1 - \sqrt{1 - \frac{2 A \alpha \delta}{ \beta_+ K^{1+\alpha}}}\right). 
\end{equation*}
It follows from this estimate and the definition of $v_n$ that there exists $c_0 < 1$ such that for all $n \in \{0, \cdots, n_0 - 1\}$, 
\[
	u_{n+1} = v_n \le c_0 u_n.
\]
Therefore,
\begin{equation*}
	\forall n \in \{0, \cdots, n_0 - 1\}, \quad u_{n} \le c_0^n K,
        \hbox{ and } v_n \le c_0^{n-1} K,
\end{equation*}
and there indeed exists $n_0 \in \N$ such that $u_{n_0} \le
\varepsilon$. One can also estimate $t_{n_0}$ and show that it is
finite, and thus $\To$ is  finite.
\end{proof}
\medskip

\begin{proof}[Proof of item $(ii)$ of Lemma~\ref{Lem-To-Finite}] In order to prove item $(ii)$ of Lemma \ref{Lem-To-Finite}, we look at the solution $u$ in $[t_*, \infty) \times [A,1]$, where $t_*$ is given by \eqref{def-t-*}. 

	It follows from Lemma \ref{Lem-Reg-u0} that $u$ is piecewise
        $\mathscr{C}^1$ and continuous in $[t_*, \infty) \times
        (A,1)$, non-negative there, and $u(t, A) = 0$ for all $t \geq
        t_*$. Consequently, for all $t \geq t_*$,  
	\[
		u(t, 1) = u(t_*, x_0), 
	\]
	where $x_0$ is  the unique solution of 
	\[
		x_0 + (t - t_*) f'( u(t_*,x_0)) = 1.
	\]
	In particular, one should have 
	\[
		\beta_- u(t_*,x_0)\le f'(u(t_*, x_0)) \le \frac{1}{t-t_*}.
	\]
Therefore, 
\begin{equation}
	\label{Decay-u-t-1}
	\forall t \ge t_*, \quad u(t,1) \le \frac{1}{\beta_- (t - t_*)}.
\end{equation}
From Lemma \ref{Lem-Reg-u0} and the fact that for all $t \geq 0$, $x \mapsto u(t,x)$ is non-decreasing on $[A,1]$ and non-negative, we get
\begin{equation*}
	\forall t \ge t_*, \quad \|u(t) \|_{L^\infty(A,1)} \le \frac{1}{\beta_- (t - t_*)}.
\end{equation*}

Now, the set $\mathscr{Z}$ on which $u = 0$ is delimited by the curve $\{ (t, A), \hbox{ for } t \geq t_*\}$ and the curve $\mathcal{C}$ given by \eqref{Eq-mathcalC}, that we now estimate: for $t_0 \geq T_*$, 
\[
 	t_0 + \frac{(u(t_0, 0))^\alpha}{\delta \alpha } \leq t_0 + \frac{1}{\beta_-^\alpha (t - t_*)^\alpha},  
\]
and
\begin{align*}
	\int_0^\infty f'((u(t_0,0)^\alpha - \delta \alpha \tau)_+^{1/\alpha}) \, d\tau 
	& 
	= \int_0^{u(t_0,0)^\alpha/(\delta \alpha)} f'((u(t_0,0)^\alpha - \delta \alpha \tau)_+^{1/\alpha}) \, d\tau 
	\\
	& \leq 
	\beta_+ \int_0^{u(t_0,0)^\alpha/(\delta \alpha)} (u(t_0,0)^\alpha - \delta \alpha \tau u(t_0,0)^{1- \alpha} )\, d\tau 
	\\
	& \leq \frac{\beta_+}{2 \delta \alpha} (u(t_0,0))^{1+ \alpha}  \leq \frac{\beta_+}{2 \delta \alpha \beta_-^{1+\alpha} } \frac{1}{(t-t_*)^{1+ \alpha}}.
\end{align*}
We then easily deduce that for $t \geq t_*$, there exists an open subinterval $\omega(t) \subset \omega$ such that $u(t)_{\mid \omega(t)} = 0$, and, for some constant $C$, for all time $t \geq t_*$,
	\begin{equation*}
		|\omega \setminus \omega(t) | \le \frac{C}{t^{1+\alpha}}. 
	\end{equation*} 

Besides, for $t \ge t_*$, we easily get from \eqref{u-T-*-R-2n+2} and
\eqref{Est-t-0-t-x} that  
\[
	\forall x \in [0,A], \quad 0 \leq u(t,x) \le \max_{t_0 \in [t - \varepsilon^\alpha/(\delta \alpha), t]} u(t,0) \leq \frac{1}{\beta_- (t-t_* - \varepsilon^\alpha/(\delta \alpha)) }.
\]
The item $(ii)$ of Lemma \ref{Lem-To-Finite} easily follows. 
\end{proof}
\subsection{Concave flux and positive constant initial datum}\label{Subsec-Concave-Flux}
Let $K>0$ and $f$ be a strictly concave flux, and consider the solution $u$ of 
\begin{equation}
 	\label{eq:flux-concave}
\d_t u + \d_x f(u) +a(x)\frac{u}{|u|^\alpha}=0,\quad (t,x)\in\R_+\times \T,
			\quad
			u_{\mid t=0}=K.
\end{equation}
Setting 
\[
	v(t,x) = u(t,-x), \quad (t,x) \in \R_+ \times \T, \quad \tilde a(x) = a(1-x), \quad x \in \T,
\]
one easily checks that $v$ formally solves 
\begin{equation}
	\label{eq:flux-convexe-v}
\d_t v + \d_x g(v) +\tilde a(x)\frac{v}{|v|^\alpha}=0,\quad (t,x)\in\R_+\times \T,
			\quad
			v_{\mid t=0}=K.
\end{equation}
with $g = -f$, which satisfies $g'(0) = 0$ and $\inf_{[-K,K]} g''
>0$. The fact that this transformation maps an admissible solution $u$
of \eqref{eq:flux-concave} to an admissible solution $v$ of
\eqref{eq:flux-convexe-v} is easy to check.  

Therefore, the counterpart of Lemma \ref{Lem-To-Finite} item $(ii)$ also holds for solutions of \eqref{eq:sol-u-K} when $f$ only satisfies \eqref{Cond-Flux}:

\begin{lemma}
	\label{Lem-Finite-Flux-General}
	Let $f$ and $K$ satisfy \eqref{Cond-Flux},  $\delta >0$ and $\omega= (0,A) \subset \T$. Then the solution $u$ of \eqref{eq:sol-u-K} satisfies the following property: There exists $t_*\geq 0$ such that for all $t \ge t_*$, there exists an open subinterval $\omega(t) \subset \omega$ such that $u(t)_{\mid \omega(t)} = 0$, and 
	\begin{equation*}
		|\omega \setminus \omega(t) | \le \frac{C}{t^{1+\alpha}}, 
		\quad 
		\| u(t) \|_{L^\infty(\T)} \le \frac{C}{t}.
	\end{equation*} 
\end{lemma}

\subsection{Negative constant initial datum}\label{Subsec-Negative-Cstt}

Let $K >0$ and consider the solution $u$ of 
\begin{equation}
 	\label{eq:flux--K}
	\left\{
		\begin{array}{ll}
			\displaystyle \d_t u + \d_x f(u) +a(x)\frac{u}{|u|^\alpha}=0,\quad (t,x)\in\R_+\times \T,
			\\
			u_{\mid t=0}=-K.
		\end{array}
	\right.
\end{equation}
Then, setting 
$$
	w(t,x) = - u(t,x), \quad (t,x) \in \R_+ \times \T, 
$$
$w$ formally solves 
\begin{equation}
	\label{eq:flux-+K}
	\left\{
		\begin{array}{ll}
			\displaystyle \d_t w + \d_x h(w) + a(x)\frac{w}{|w|^\alpha}=0,\quad (t,x)\in\R_+\times \T,
			\\
			w_{\mid t=0}=K, 
		\end{array}
	\right.
\end{equation}
where the flux $h$ is given by
\begin{equation*}
	\forall s \in [-K,K], \quad h(s) = - f(-s).
\end{equation*}
It is easy to check that $h'(0) = 0$ and $\inf_{[-K,K]} |h''(s)| >0$. Besides, this transformation also establishes the correspondence between the admissible solution $u$ of \eqref{eq:flux--K} and the admissible solution $w$ of \eqref{eq:flux-+K}. 

Therefore, Lemma \ref{Lem-To-Finite} item $(ii)$ also holds when the
initial datum is constant $= - K$, under the only condition that the
flux $f$ satisfies \eqref{Cond-Flux}. 

\begin{lemma}
	\label{Lem-omega-General}
	Let $f$ and $K$ satisfy \eqref{Cond-Flux},  $\delta >0$ and $\omega= (0,A) \subset \T$. Then the solutions $u_\pm $ of 
	\begin{equation}
  \label{eq:sol-u-K-pm}
  \d_t u_\pm+\d_x f(u_\pm)+\delta 1_{(0,A)} \frac{u_\pm}{|u_\pm|^\alpha} =0 \text{ in } \R_+ \times \T, \quad u_\pm \mid_{ t=0}= \pm K \text{ in } \, \T.
	\end{equation}
 satisfy the following property: There exists $t_* \geq 0$ such that for all $t \ge t_*$, there exists an open subinterval $\omega_\pm(t) \subset \omega$ such that $u_\pm(t)_{\mid \omega_\pm(t)} = 0$, and 
	\begin{equation*}
		|\omega \setminus \omega_\pm (t) | \le \frac{C}{t^{1+\alpha}}, 
		\quad 
		\| u_\pm(t) \|_{L^\infty(\T)} \le \frac{C}{t}.
	\end{equation*} 
\end{lemma}

\subsection{Proof of Theorem \ref{thm:Burgers-Main}}
The proof of Theorem \ref{thm:Burgers-Main} follows by comparing the solution $u$ of \eqref{eq:flux-general}--\eqref{eq:ci} with initial datum $u_0 \in L^\infty(\T)$ with some reference solutions. 

Namely, we assume that $a$ satisfies \eqref{eq:local} for some open interval $\omega \subset \T$. Up to a translation in space, we can assume that $\omega = (0,A)$. Therefore, the solution $u$ of \eqref{eq:flux-general}--\eqref{eq:ci} with initial datum $u_0 \in L^\infty(\T)$ with $\|u_0\|_{L^\infty} \leq K$, $K$ as in \eqref{eq:def-K}, is sandwiched between the solutions $u_\pm$ of \eqref{eq:sol-u-K-pm}.

We then immediately conclude Theorem \ref{thm:Burgers-Main} from Lemma \ref{Lem-omega-General}.


\section{Numerical simulations and open problems}
\label{sec:num}
We present in this section some numerical experiments for various equations to which our theoretical results do not apply. The main numerical technique relies on the time-splitting scheme as time integrator. If one considers a general evolution equation
\begin{equation}
\label{eq:split}
\left \{
\begin{array}{ll}
\partial_t u = \A u + \B u, & (t,x)\in \R_+\times \T,\\
u(0,x)=u_0(x), & x\in \T,
\end{array}
\right .
\end{equation}
where $\A$ and $\B$ are (possibly non-linear) operators which need not
 commute. For a given time step $\delta t>0$, set $t_n=n\delta
t$, $n=0,1,2,\dots$ Define the operators $S_\A$ and $S_\B$
associated respectively to the evolution equations 
\[
\partial_t u_\A = \A u_\A, \quad \partial_t u_\B = \B u_\B, \quad
(t,x)\in \R_+\times \T, 
\]
The operators satisfy the following relations involving the exact solutions of the associated equations:
\[
u_\A(t+\delta t)=S_\A(\delta t)u_\A(t) \quad \text{and} \quad u_\B(t+\delta t)=S_\B(\delta t)u_\B(t).
\]
The splitting idea (see for example \cite{HLW}) consists in
approximating the continuous flow associated to \eqref{eq:split} by a
composition of operators $S_\A$ and $S_\B$ in the spirit of
Trotter-Kato formula, the key for an efficient implementation being to
solve efficiently these two reduced equations. We consider in this
paper the second order Strang splitting scheme. 
Let $u^n(x)$ be the approximation of $u(t_n,x)$. The approximate solution to \eqref{eq:split} at time $t_{n+1}$ reads
\begin{equation}
  \label{eq:strang}
	u^{n+1}=S_\A(\delta t/2)S_\B(\delta t)S_\A(\delta t/2) u^n.  
\end{equation}
\subsection{Scalar conservation laws}
Let us now describe how it is applied to the equation \eqref{eq:flux-general}. It involves the two reduced equation
\begin{equation}
  \label{eq:hyperb}
  \partial_t u+\partial_xf(u)=0 , \quad (t,x) \in \R_+ \times \T,
\end{equation}
and
\begin{equation}
  \label{eq:ode}
  \partial_t u = -a(x) \frac{u}{|u|^\alpha}, \quad (t,x) \in \R_+ \times \T.
\end{equation}
The equation \eqref{eq:hyperb} is a standard nonlinear conservative hyperbolic equation. In the second equation \eqref{eq:ode}, the space variable can be considered as a parameter and the equation reduced to an ordinary equation with solution
\begin{equation}
  \label{eq:sol_ode}
  u(t,x)=\text{sign}(u_0(x)) \left(|u_0(x)|^\alpha-\alpha a(x) t\right)^{1/\alpha}_+.
\end{equation}
If the flux $f$ is linear, $f(u)=c u$, the solution to \eqref{eq:hyperb} is obviously
\[
u(t,x)=u_0(x-ct).
\]
For a general flux, we compute an approximate solution thanks to
Rusanov scheme (see for example \cite[p.233]{Lev2}). We identify the torus with $(0,1)$ endowed with
periodic boundary conditions and choose the spatial mesh size $\delta
x>0$ with $\delta x=1/J$ with $J$ denoting the number of nodes. The
grid points are $x_j=j\delta x$, $j=0,1,\cdots,J$. Let $u_j^n$ be the
full approximation to $u(t_n,x_j)$. The Rusanov scheme reads 
\[
u_j^{n+1}=u_j^{n}-\frac{\delta t}{\delta x}\left (F_{j+1/2}^n-F_{j-1/2}^n\right ),
\]
where the Rusanov numerical flux is given by
\[
  \begin{array}{ll}
    F_{j+1/2}^n&=F^{\text{Rus}} (u_j^n,u_{j+1}^n)\\
    & \displaystyle=\frac{f(u_j^n)+f(u_{j+1}^n)}{2}-\frac{\text{max}(|f'(u_j^n)|,|f'(u_{j+1}^n)|)}{2} (u_{j+1}^n-u_j^n).
  \end{array}
\]
The theoretical results of previous sections apply to fluxes with assumption $f'(0)\neq 0$ or $f'(0)=0$ with convexity hypothesis (convex or concave flux). Some fluxes do not satisfy such hypothesis. This is the case of the Buckley-Leverett flux which models two phase fluid flow in a porous medium (\cite{Lev}). In one space dimension the equation has the standard conservation law form ($k>0$ is a parameter)
\begin{equation}
  \label{eq:buck_lev_flux}
  f_k^{BL}(u)=\frac{u^2}{u^2+k(1-u)^2}.
\end{equation}
We compute the evolution of the solution to \eqref{eq:flux-general}
with $f^{BL}_{1/4}$ flux and the damping function $a(x)$ given by
\eqref{eq:damping-profile}. The numerical parameters are
$u_0(x)=K=1.25$, $A=1/4$, $\delta = 1$, $\delta t=10^{-5}$ and $\delta
x = 5\cdot 10^{-5}$. The evolution of the solutions for $\alpha=3/4$
and $\alpha=1$ are plotted on Figure \ref{fig:evo_sol_buck} and the
evolution of their characteristic curves on Figure
\ref{fig:evo_buck}. The characteristic curves are computed as the
evolution of a vector field with velocity given by the solution to
\eqref{eq:flux-general}. We see that contrary to convex (or concave)
fluxes with $f'(0)=0$, shock and rarefaction waves appear in finite
time. Since the domain is a torus, the shock wave initiated from
$x=A=1/4$ propagates until the influence of the damping function $a$
is enough important to annihilate the solution inside the support of
$a$. We then recover a similar process as the one observed on Figure~\ref{fig:evo_burgers} where characteristic curves become vertical
lines in finite time inside the support of the damping function
$a$. The effect of decreasing $\alpha$ is to delay the extinction of the solution
in $(0,A)$. The proof of this phenomenon is still missing. 
\begin{figure}[h]
  \centering
\begin{tabular}{cc}
\input{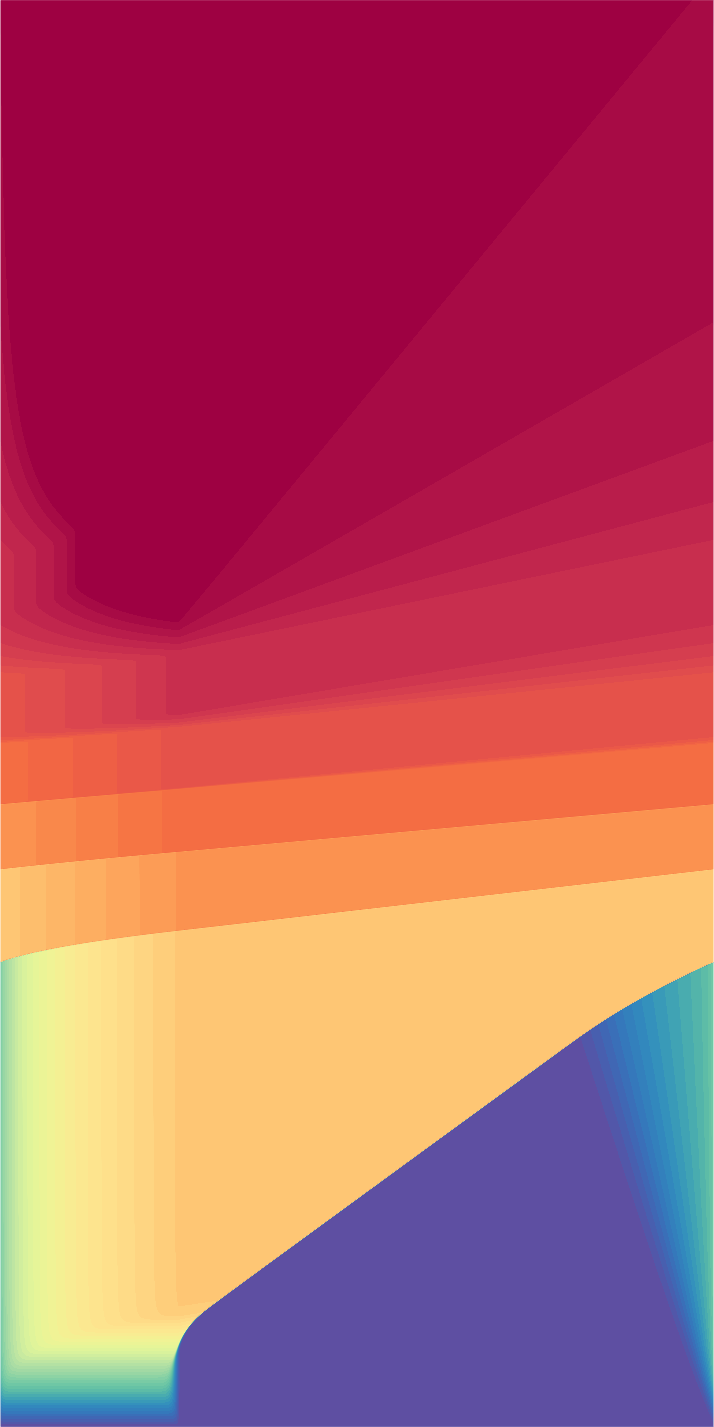}
  &
\input{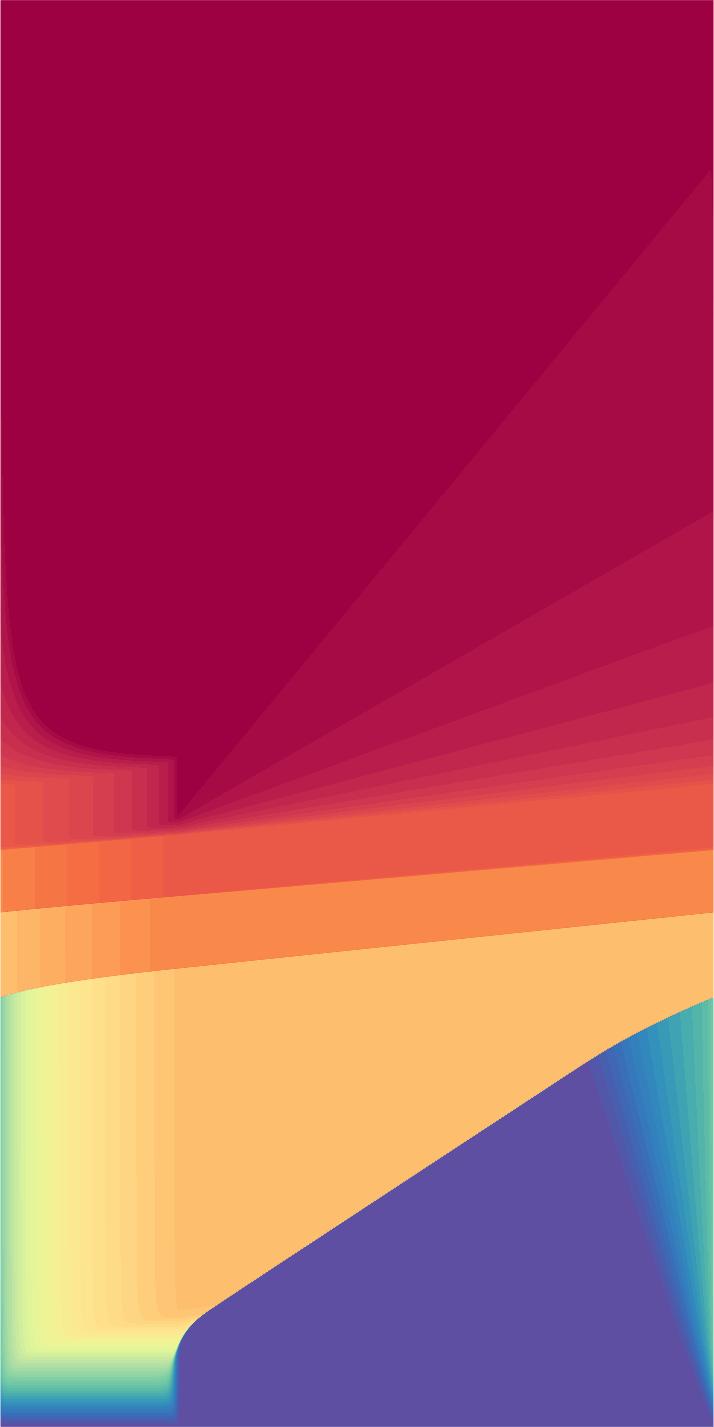}\\
\end{tabular}   
  \caption{Evolution of the solution for Buckley-Leverett flux $f^{BL}_{1/4}$ with $\alpha=0.75$ (left) and $\alpha=1$ (right)}
  \label{fig:evo_sol_buck}
\end{figure}

\begin{figure}[h]
  \centering
\begin{tabular}{cc}
\input{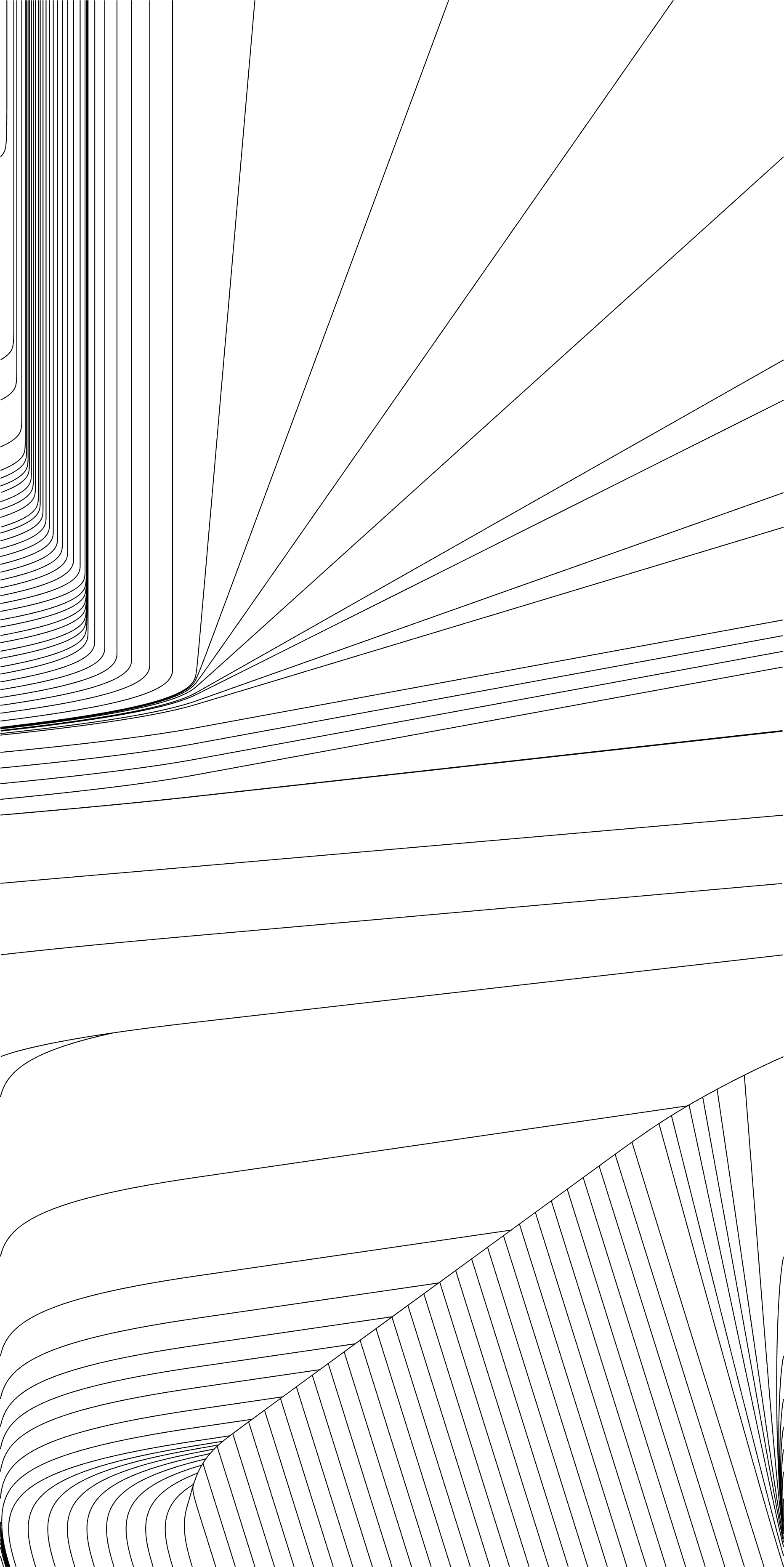}
  &
\input{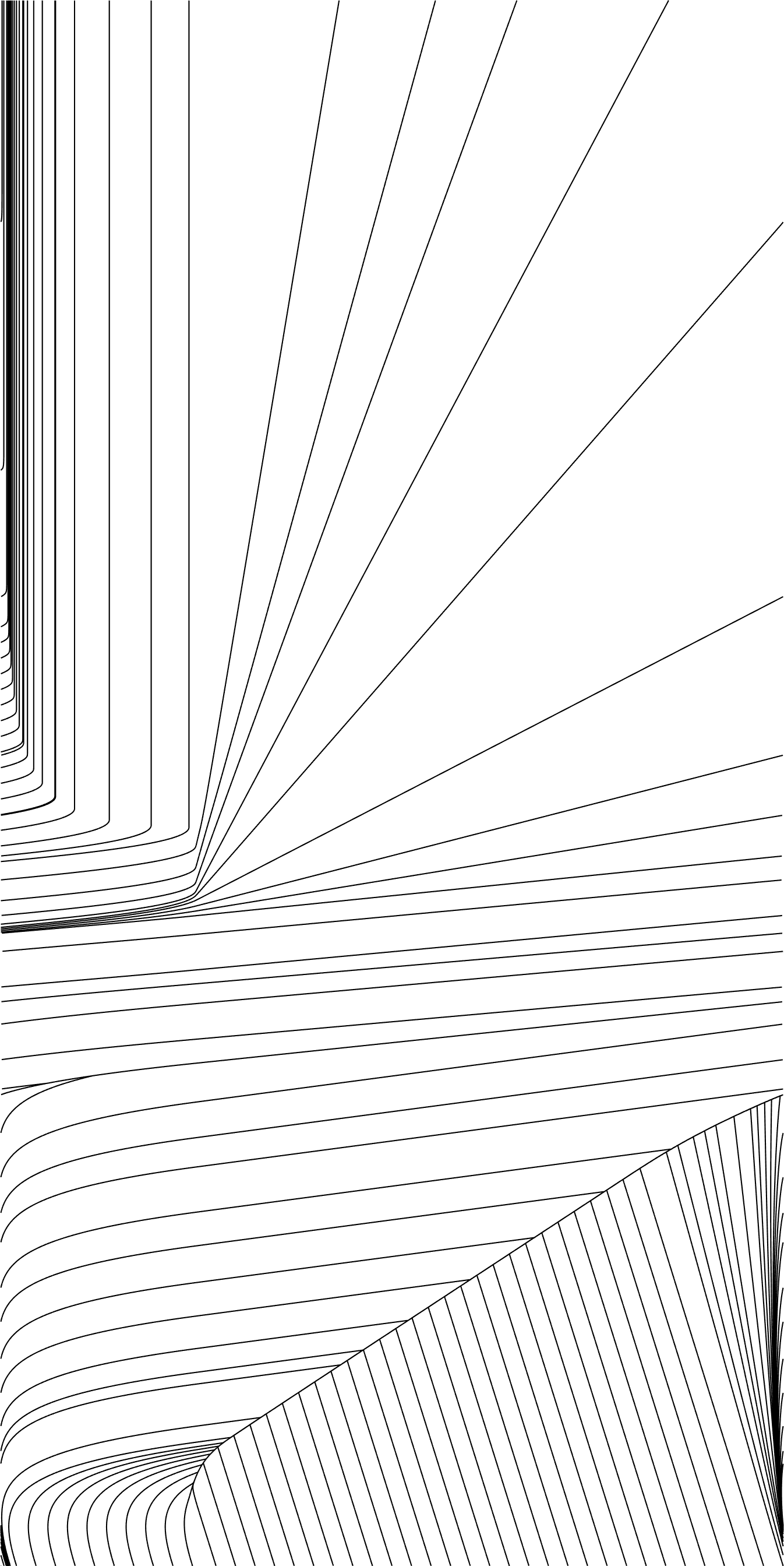}\\
\end{tabular}   
  \caption{Evolution of the characteristic curves for Buckley-Leverett flux $f^{BL}_{1/4}$ with $\alpha=0.75$ (left) and $\alpha=1$ (right)}
  \label{fig:evo_buck}
\end{figure}

\subsection{Viscous Burgers equations}
We consider here the convection diffusion equation given for $\mu >0$ by 
\begin{equation}
  \label{eq:burg_heat}
  \partial_t u + \partial_x \left (\frac{u^2}{2}\right ) = \mu \partial_x^2 u - a(x) \frac{u}{|u|^\alpha}, \quad (t,x) \in \R_+\times \T.
\end{equation}
We look at the solutions when $a(x)=\delta \1_{\omega}$, $\omega = (0,
A)$, $A<1$, the torus $\T$ being the circle $(0,1)$. This equation
involves three different operators. We have to apply a second order
three-operators splitting scheme which reads for the evolution
equation $\partial_t u=(\A+\B+{\CC})u$:

\begin{equation}
\label{eq:strang3}
u^{n+1}=S_\A(\delta t/2)S_\B(\delta t/2)S_{\CC}(\delta t)S_\B(\delta t/2)S_\A(\delta t/2) u^n,
\end{equation}
where $S_\A$, $S_\B$ and $S_{\CC}$ denote the flows associated to
operators $\A$, $\B$ and $\CC$. Since we study the equation
\eqref{eq:burg_heat} on a torus, we benefit from the periodicity to
use fast Fourier transform in order to make space approximation of the
solutions of the  heat equation 
\[
\partial_t w=\mu \partial_x^2 w.
\]
The spatial mesh size is defined by $\delta x=1/J$, $J=2^{P}$, $P\in \N^*$. Since we discretize the heat equation by the Fourier spectral method, $w_j^n$ and
its Fourier transform satisfy the following relations:
\[
w_j^n = \frac{1}{J} \sum_{m=-J/2}^{J/2-1} \hat{w}_m^n e^{i\xi_m(x_j-x_\ell)}, \quad j=0,\cdots,J-1,
\]
and
\[
\hat{w}_m^n=\sum_{j=0}^{J-1} w_j^n e^{-i\xi_m(x_j-x_\ell)}, \quad m=-\frac{J}{2},\cdots,\frac{J}{2}-1,
\]
where $\xi_m=2\pi m$ for all $m=-\frac{J}{2},\cdots,\frac{J}{2}-1$. The discrete Laplace operator $\Delta_P$ is therefore define by
\[
\widehat{(\Delta_P v)}_m = - \xi_m^2  \hat{v}_m, \quad v\in \C^M.
\]

We present on Figure \ref{fig:evo_sol_burg_heat} the evolution of the
logarithm of the solution. We choose the same numerical parameters
used for Buckley-Leverett equation, the only difference relying on the
mesh size which is $\delta x = 2^{-14}$.  We present the logarithm to
show that like in the hyperbolic case, the solution becomes zero on
the support of the damping function $a$ after a time $T^*$ which
depends on the parameter $\alpha$. What is more surprising is the fact
that after the time $T^*$, the solution on $(A,1)$ behaves like the
solution of the heat equation with homogeneous Dirichlet boundary
conditions associated to the first eigenvalues of the Laplacian. We
know that this solution on $(A,1)$ is  
\[
v(t,x)=\exp(-\mu \lambda^2 t) \sin(\lambda (x-A)),
\]
with $\lambda=\pi/(1-A)$. We clearly identify this phenomenon by
displaying the evolution of the $L^\infty$ norm of the solution with
respect to time on Figure \ref{fig:evo_linf_burg_heat}. We plot both
the $L^\infty$ norm and a dashed line in log-scale with slope $-\mu
\lambda^2$.  The sin-like behavior of the solution on $[A,1]$ for time
$t=10$ is also clearly present on Figure
\ref{fig:evo_linf_burg_heat}. A rigorous mathematical proof of the
above observations is, to our knowledge, missing, despite the works
\cite{Be10,BeSh10}, where conditions for complete extinctions of the
solutions are discussed (see also \cite{Be01,BeHeVe01,BeSh07} for related results).
\setlength{\malargeur}{4.3cm}
\setlength{\figureheight}{0.85\malargeur}
\setlength{\figurewidth}{1.385\malargeur} 
\begin{figure}[h]
  \centering
\input{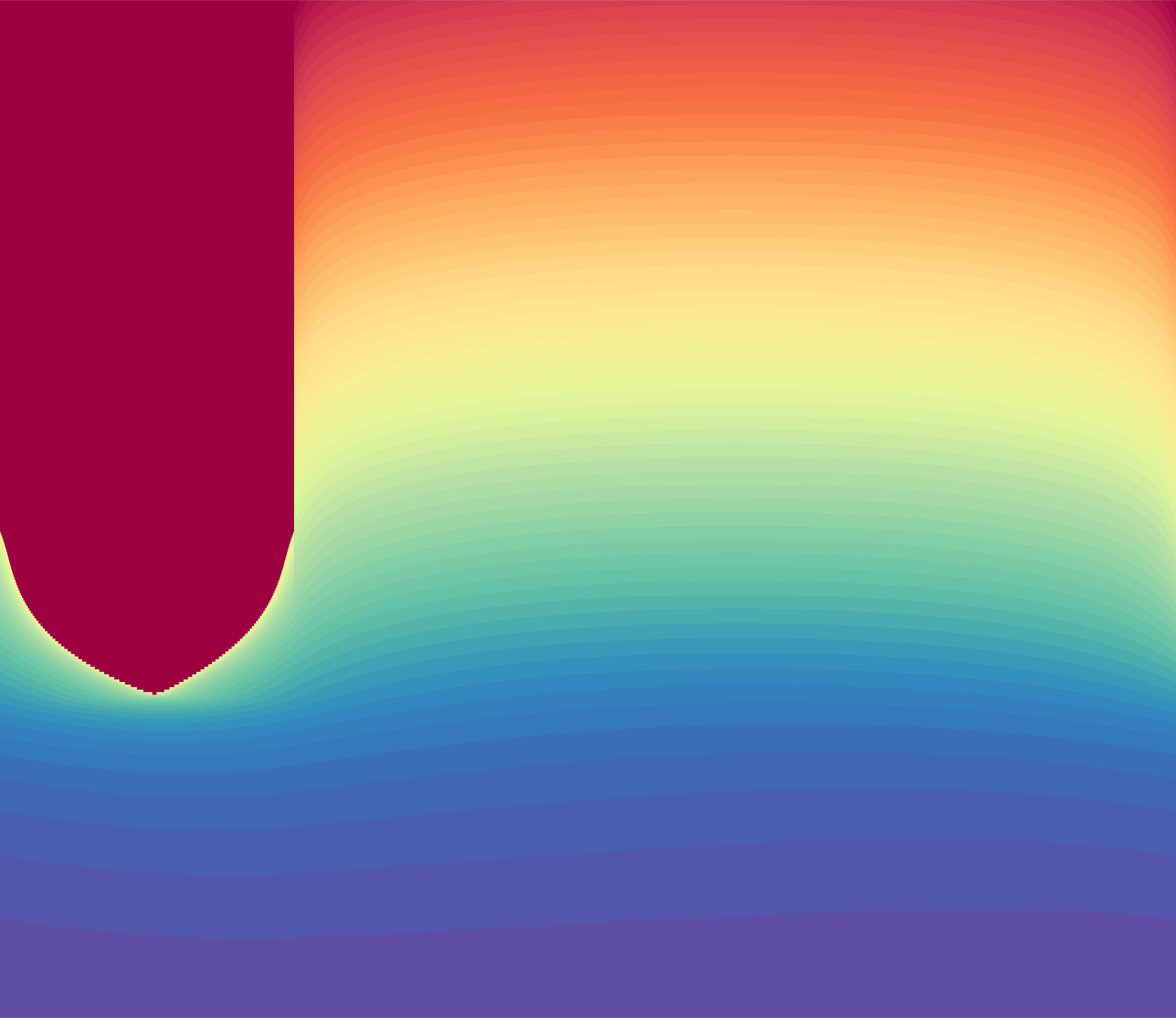}
 \caption{Evolution of the $\log_{10}$ of the solution to \eqref{eq:burg_heat} with $\alpha=0.75$}
  \label{fig:evo_sol_burg_heat}
\end{figure}
\setlength{\malargeur}{4.3cm}
\setlength{\figureheight}{1\malargeur}
\setlength{\figurewidth}{1.63\malargeur} 
\begin{figure}[h]
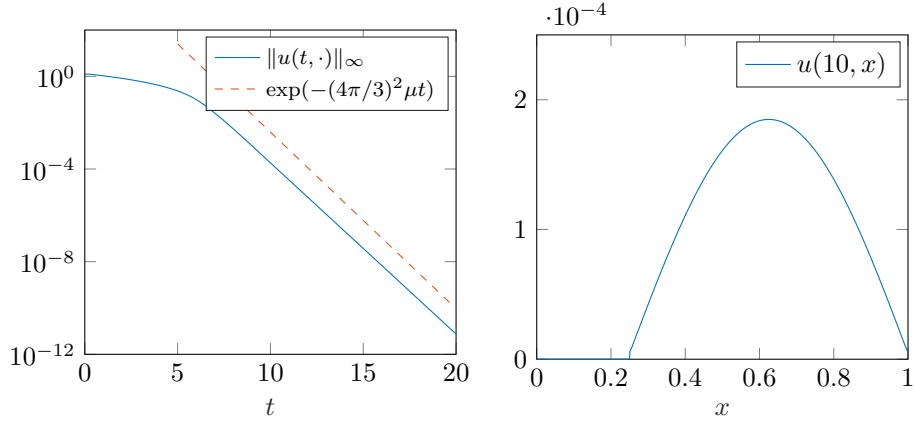

  \centering
\begin{tabular}{cc}
\input{evol_Linfnorm_visq_burgers_alpha_0_75}
  &
\input{evol_sol_t_10_burg_heat}\\
\end{tabular}   
  \caption{Evolution of the $L^\infty$ norm of the solution to \eqref{eq:burg_heat} in log-scale (left) and solution at time $t=10$ (right) for $\alpha=0.75$}
  \label{fig:evo_linf_burg_heat}
\end{figure}

We end up this paragraph by emphasizing that \eqref{eq:burg_heat} is a
viscous approximation of the Burgers equation, which is a conservation
law fitting the assumptions of Theorem \ref{thm:Burgers-Main}. It is
thus completely natural to ask the behavior of \eqref{eq:burg_heat} in
large times, similarly to what has been done in Theorem
\ref{thm:Burgers-Main}. Though, as our numerical simulations
underline, the large-time behavior of \eqref{eq:burg_heat} is very
different from the one of the Burgers equations predicted by Theorem
\ref{thm:Burgers-Main}. This is an evidence of the fact that the limit
of large times and the limit of small viscosities do not commute, as
observed in other contexts for instance in \cite{Ignat-Pozo-Zuazua}.

\subsection{Wave equation}
We consider now the wave equation with homogeneous Dirichlet boundary conditions
\begin{equation}
  \label{eq:wave}
\left\{
  \begin{array}{l}
\displaystyle  \partial_t^2u-c^2\partial_x^2u=-a(x) \frac{\partial_t u}{|\partial_t u|^\alpha}, \quad  (t,x) \in \R_+\times (0,1),\\
u(t,0)=u(t,1)=0,
  \end{array}
\right.
\end{equation}
completed with initial conditions $u(0,x)=u_0(x)$ and $\partial_t u(0,x)=u_1(x)$.

To numerically simulate the solution to \eqref{eq:wave}, we begin by transforming the equation as the first order hyperbolic system 
\[
\partial_t
\begin{pmatrix}
  u\\v
\end{pmatrix}
=
\begin{pmatrix}
  0 & 1 \\ c^2 \partial_x^2 & 0
\end{pmatrix}
\begin{pmatrix}
  u\\v
\end{pmatrix}
+
\begin{pmatrix}
  0 \\ - a(x) v/|v|^\alpha
\end{pmatrix}.
\]
We can therefore apply the Strang splitting method \eqref{eq:strang}. The solution to the ODE $\partial_t v = - a(x) v/|v|^\alpha$ is given by \eqref{eq:sol_ode} and we approximate the free wave equation with the Newmark scheme (\cite{Quarteroni})
\begin{equation}
  \label{eq:Newmark}
  \begin{array}{l}
\displaystyle    u_j^{n+1}=u_j^n + \delta t v_j^n + \delta t^2 \left [ \zeta c^2 w_j^{n+1} + (1/2-\zeta) c^2 w_j^n\right],\\[3mm]
\displaystyle v_j^{n+1}=v_j^n +\delta t \left [ (1-\theta) c^2 w_j^n +\theta c^2 w_j^{n+1} \right ],
  \end{array}
\end{equation}
with $u_j^0=u_0(x_j)$, $v_j^0=u_1(x_j)$ and $w_j^k=(u_{j+1}^k-2u_j^k +u_{j-1}^k)/(\delta x)^2$. We select for our numerical simulations $\theta=1/2$ and $\zeta=1/4$ for which the scheme is both second order in space and time and unconditionally stable.

We select the initial conditions
\[
u_0(x) = K \left \{
\begin{array}{ll}
1-e^1\exp(-0.1/(0.1-x)), & \text{if } x<0.1,\\
1, & \text{if } 0.1\leq x \leq 0.9,\\
1-e^1\exp(-0.1/(x-0.9)), &\text{if } x>0.9,
\end{array}
\right .
\]
and $u_1(x)=0$. The damping function is $a(x)=\delta \1_{\omega}$, $\omega = (3/8, 5/8)$. The numerical parameters are $\alpha=1$, $c=0.1$, $\delta t=5\cdot 10^{-4}$, $\delta x = 10^{-3}/3$, $\delta =1$ and $K=1.25$. 
\setlength{\malargeur}{4.8cm}
\setlength{\figureheight}{0.683\malargeur}
\setlength{\figurewidth}{\malargeur} 

\begin{figure}[h]
  \centering
\begin{tabular}{cc}
\input{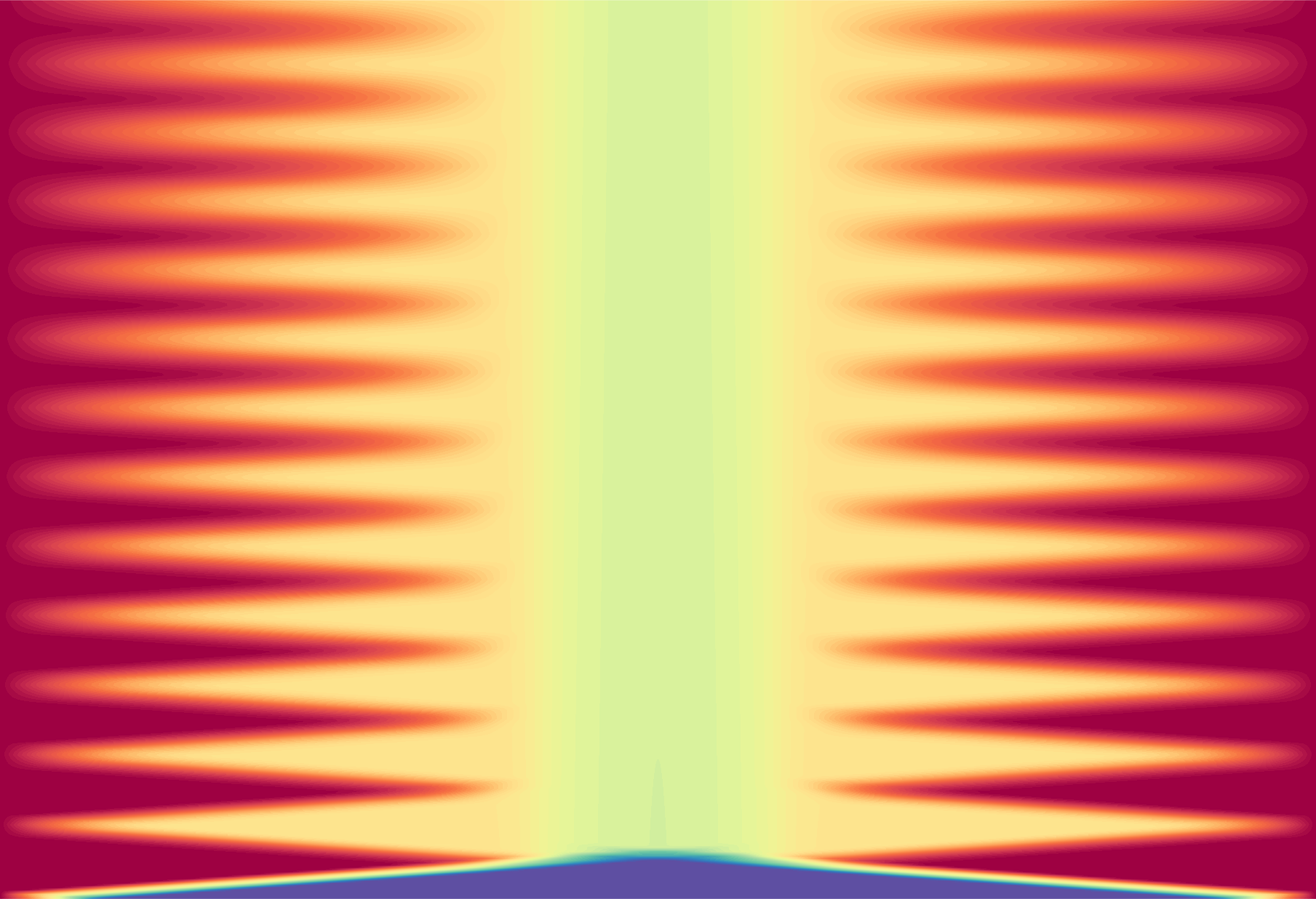}
  &
\input{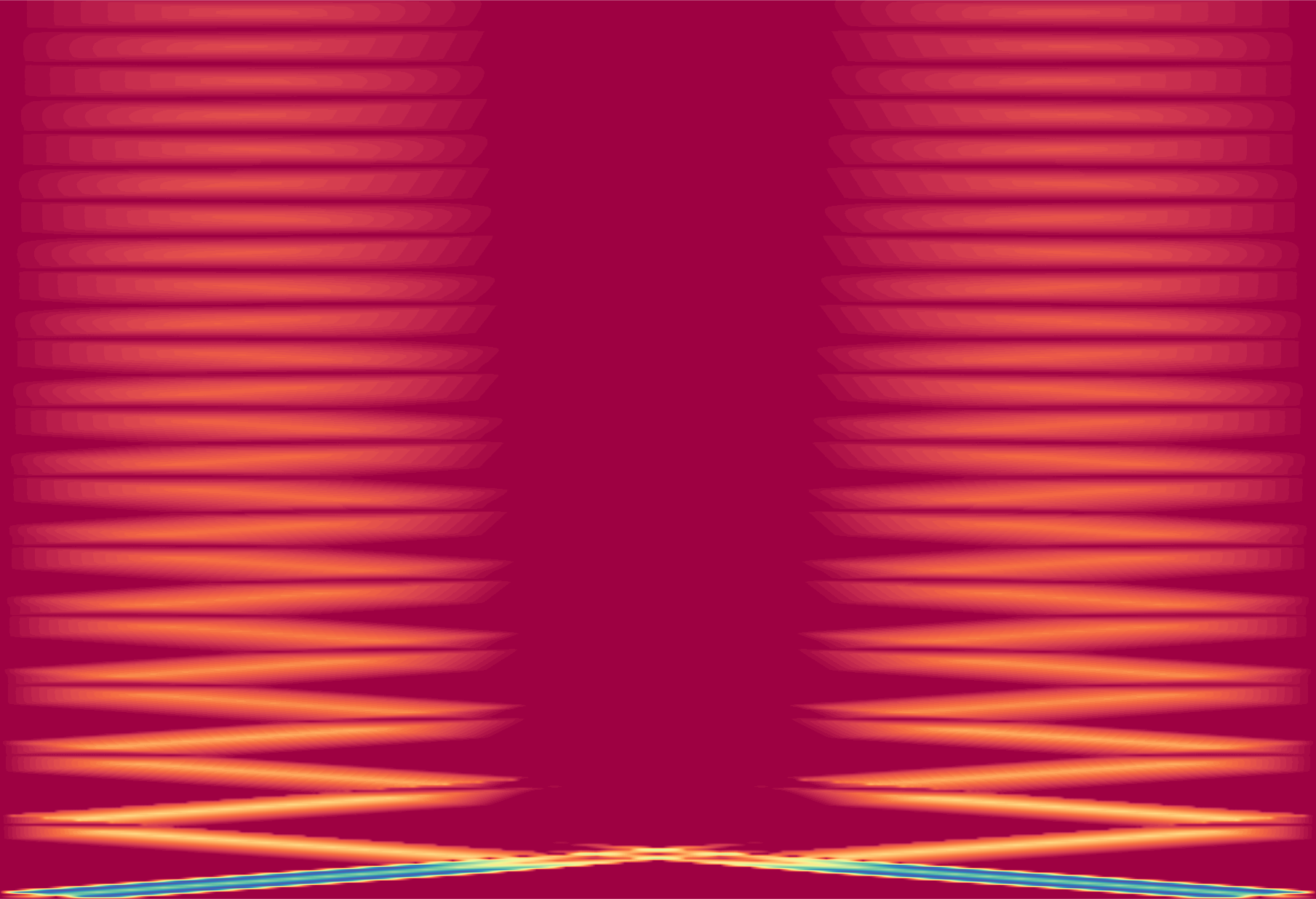}\\
\end{tabular}   
  \caption{Evolution of the solution to \eqref{eq:wave}, $u$ on the left and $\partial_t u$ on the right}
\label{fig:wave}
\end{figure}
As can be expected (see Figure \ref{fig:wave}), the time derivative of the solution is annihilated on the support of $a$ after a time $T^*$, the solution $u$ becoming constant for $t>T^*$.
\\
Let us underline that the linear wave equation is the prototype of a $2\times 2$ system of conservation laws, which can be easily seen with the use of characteristics. It is thus natural to consider such models as a generalization of \eqref{eq:flux-general}--\eqref{eq:ci}. Note that the behavior of the solution of \eqref{eq:wave} when the damping acts everywhere in the domain has been studied in \cite{BaCaDi07}, or when the damping acts on the boundary \cite{Perrollaz-Rosier-2014}, but the case of a localized damping term involving $\partial_t u$ still does not seem to be precisely described in the literature. In fact, the interested reader should also notice the close connection of this problem with the non-linear damped oscillator of the form 
$$
	m x'' + \delta \frac{x'}{|x'|^\alpha} + \omega^2 x = 0, \quad t \geq 0,
$$
 with $m>0$, $\alpha \in (0,1]$, and $\omega >0$, whose large time behavior is rather subtle, see e.g. \cite{AmDi03,Vazquez-2003}. 
 
\subsection{Schr\"odinger equation}
The last equation we consider is the strongly damped cubic nonlinear Schr\"odinger (NLS) equation, motivated by the works \cite{CaGa11,CaOz15} in which the damping is effective everywhere. We thus wonder if the previous results for hyperbolic equations can be extended to the Schr\"odinger equation
\begin{equation}
  \label{eq:schrod}
i  \partial_t u + \partial_x^2 u = -q|u|^2 u - i a(x) \frac{u}{|u|^\alpha}, \quad (t,x) \in \R\times \T.
\end{equation}
with initial datum $u(0,x)=u_0(x)$. If the damping function $a$ is
zero, then the cubic NLS equation admits a special solution known as
soliton. This solution is given by 
\begin{equation}
  \label{eq:soliton}
  u_{\text{sol}}(t,x)=\sqrt{\frac{2 k}{q}} \text{sech}(\sqrt{k} (x-ct)\exp(i\frac{c}{2}(x-ct))\exp(i(k+\frac{c^2}{4})t).
\end{equation}
This solution evolves as the profile $u_0$ propagating at velocity $c$ with time phase change. This solution for $c=20$, $k=0.81$ and the torus $x\in (-10,10)$ is plotted on Figure \ref{fig:soliton}.
\begin{figure}[h]
  \centering
\input{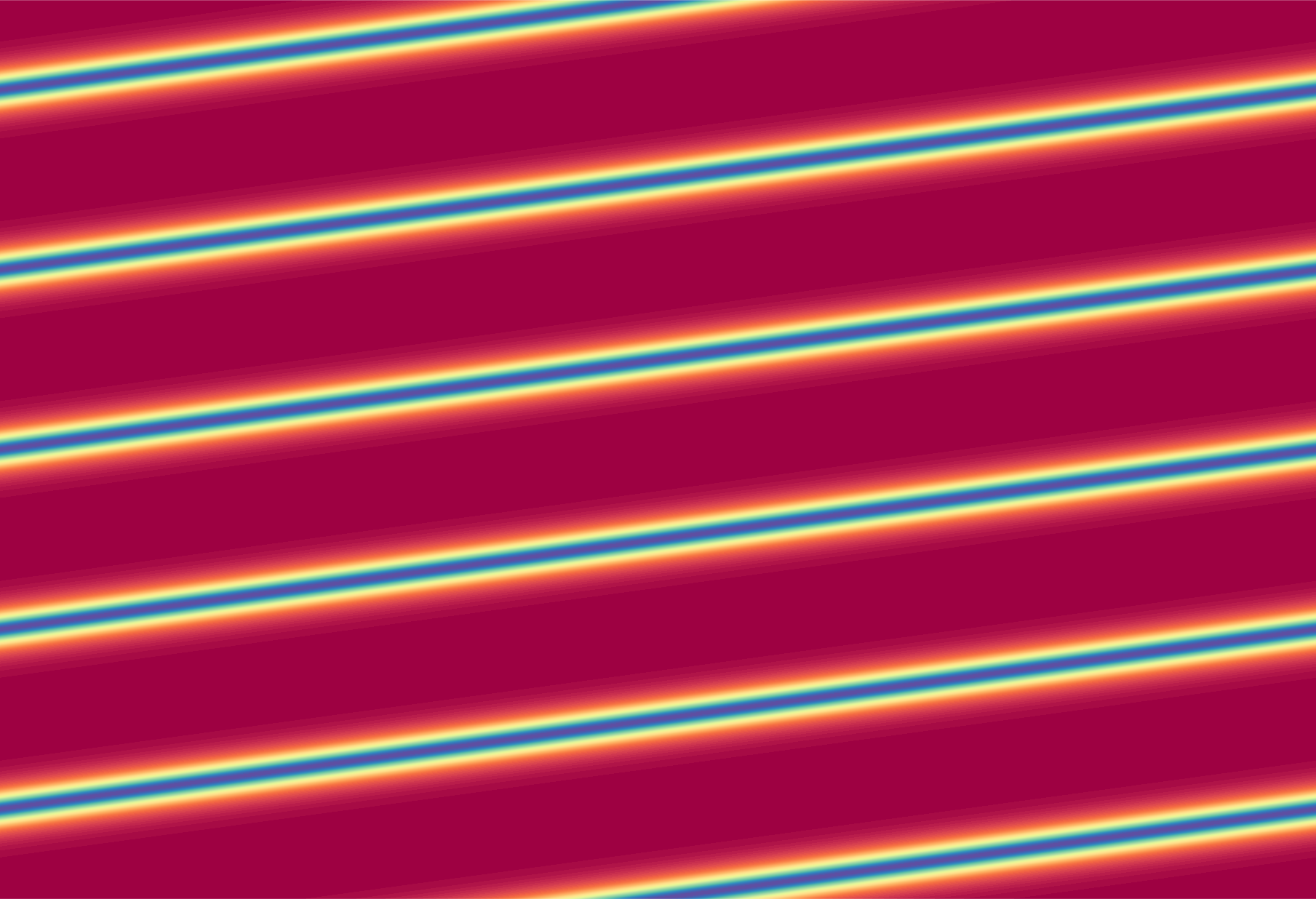}
  \caption{Evolution of the modulus of the soliton \eqref{eq:soliton} for $c=20$ and $k=0.81$.}
\label{fig:soliton}
\end{figure}
 The numerical scheme again relies on the Strang splitting scheme for three operators \eqref{eq:strang3}. As for the Burgers heat equation, the space approximation is performed thanks to fast Fourier transform. The complex solution to ODE $\partial_t u = - a(x) u/|u|^\alpha$ is given by
\[
u(t,x)=(|u_0|^\alpha-\alpha a(x) t)^{1/\alpha}_+\exp(i \text{Arg}(u_0)),
\]
whereas the solution to the ODE $i\partial_t u= -q|u|^2 u$ is given by
\[
u(t,x)=\exp(itq |u_0(x)|^2)u_0(x).
\]
We present the effects of the damping function $a(x)=\delta \1_{\omega}$, $\omega = (-10,-6)\cup(6,10)$ on the soliton for $\T=(-10,10)$ and  $\alpha=1$.
The soliton initial datum overlaps the support of $a$.
The numerical parameters are $\delta t=5\cdot 10^{-4}$ and $\delta x=10\cdot 2^{-12}$. As expected, the solution begins to propagate to the right direction and then vanishes on the support of $a$ (see Figure \ref{fig:schrod}). This is more clear on log scale. Again, to our knowledge, this behavior has not been proved rigorously in the literature. 

\begin{figure}[h]
  \centering
\begin{tabular}{cc}
\input{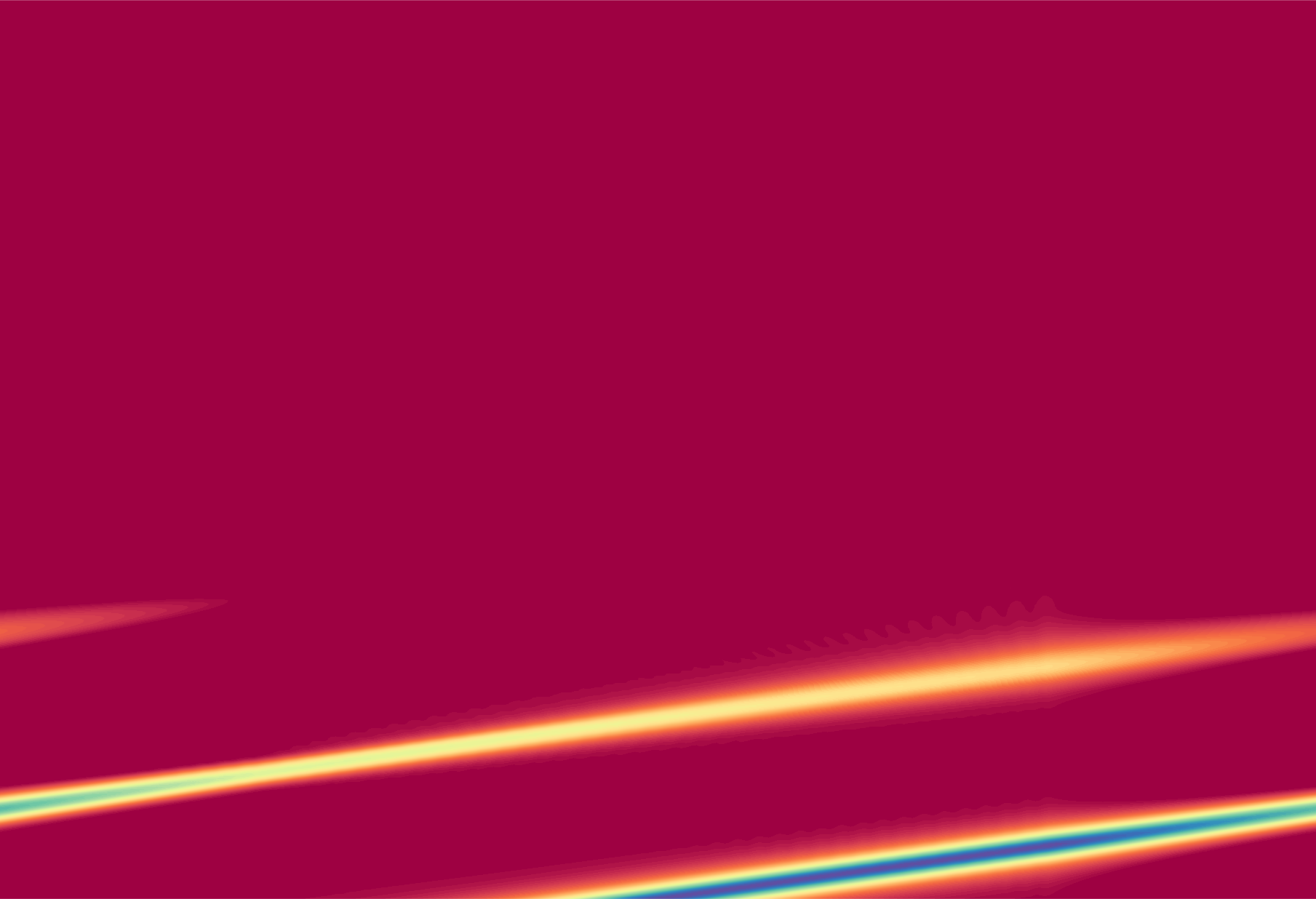}
  &
\input{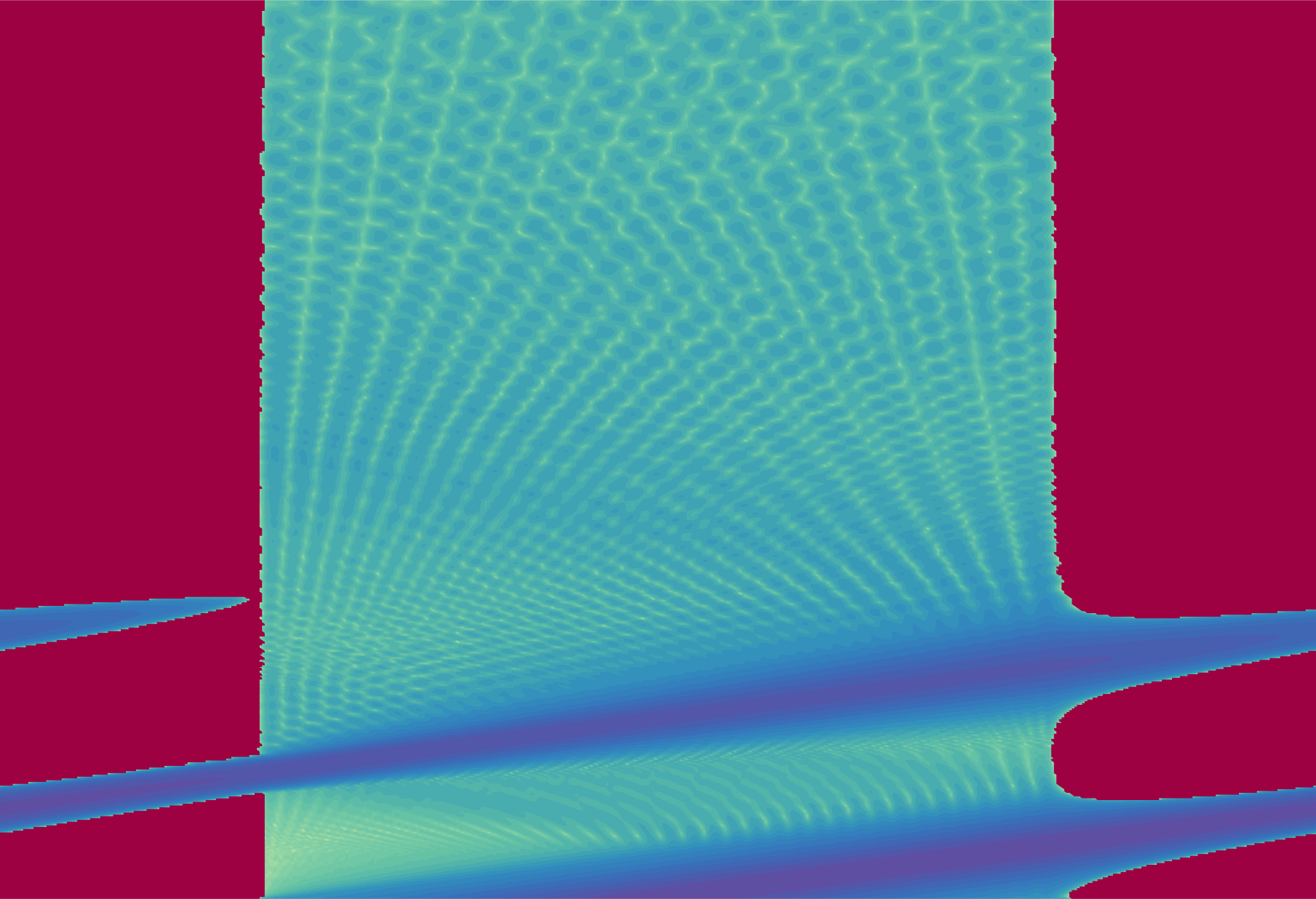}\\
\end{tabular}   
  \caption{Evolution of the solution to \eqref{eq:schrod} in standard scale (left) and in log scale (right)}
\label{fig:schrod}
\end{figure}

\bibliographystyle{siam}

\bibliography{biblio}

\end{document}